\documentclass[12pt]{article}

\usepackage{amsfonts,graphicx,amsmath,amssymb,amsthm,color,nicefrac,enumerate, layout, bbm, vmargin, listings, url}
\usepackage[latin1]{inputenc}
\newcommand{\R}{\mathbb{R}}
\newcommand{\W}{\mathcal{W}}
\newcommand{\N}{\mathbb{N}}
\newcommand{\Z}{\mathbb{Z}}
\newcommand{\E}{\mathbb{E}}
\newcommand{\und}{\underline}

\newcommand{\F}[0]{\ensuremath{\mathcal{F}}}

\newcommand{\smallsum}{\textstyle\sum}
\renewcommand{\P}{\mathbb{P}}

\newcommand{\lf}{\lfloor}
\newcommand{\rf}{\rfloor}

\newcommand{\ctwo}{\varphi}
\newcommand{\cthree}{\theta}
\newcommand{\cfour}{\vartheta}
\newcommand{\vt}{{\eta}}
\newcommand{\Y}{\mathcal{X}}
\newcommand{\fl}[1]{\lfloor #1 \rfloor_{h}}
\newcommand{\Id}{\mathrm{Id}}

\newtheorem{theorem}{Theorem}[section]
\newtheorem{lemma}[theorem]{Lemma}

\newtheorem{cor}[theorem]{Corollary}

\newtheorem{prop}[theorem]{Proposition}

\begin{document}
\title{Strong convergence of full-discrete nonlinearity-truncated 
accelerated exponential Euler-type
approximations for stochastic Kuramoto-Sivashinsky  equations}

\author{Martin Hutzenthaler, Arnulf Jentzen, and Diyora Salimova \\
	\small University of Duisburg-Essen (Germany), ETH Zurich (Switzerland), and ETH Zurich (Switzerland)}

\maketitle

\tableofcontents

\begin{abstract}
This article introduces and analyzes a new explicit,
easily implementable, and full discrete 
accelerated exponential Euler-type approximation scheme
for additive space-time white noise driven stochastic partial differential 
equations (SPDEs) with possibly non-globally monotone nonlinearities
such as stochastic Kuramoto-Sivashinsky  equations.
The main result of this article proves 
that the proposed approximation scheme
converges strongly and numerically weakly to the solution process 
of such an SPDE.
Key ingredients in the proof of our convergence result
are a suitable generalized coercivity-type condition, 
the specific design of the accelerated exponential Euler-type 
approximation scheme, and an application of 
Fernique's theorem.
\end{abstract}

\section{Introduction}

For strong $L^2$-convergence of a sequence of approximations it is necessary that the the $L^2$-norms
of the approximations are uniformly bounded.
In the case of finite-dimensional stochastic differential equations (SDEs) this is ensured by
the well-known \textit{coercivity condition}. If $d\in\N$ is the dimension of the SDE, $\mu\colon\R^d\to\R^d$ is the drift coefficient
and $\sigma\in\R^{d\times d}$ is the diffusion coefficient, then the coercivity condition is satisfied
if there exists $c\in\R$ such that for all $x\in\R^d$ it holds that
\begin{equation}  \begin{split}
  \langle x,\mu(x)\rangle_{\R^d}+\tfrac{1}{2}\|\sigma(x)\|_{\textup{HS}(\R^d,\R^d)}^2 \leq c(1+\|x\|_{\R^d}^2).
\end{split}     \end{equation}
In an infinite-dimensional separable Hilbert space $H$, this coercivity condition requires the diffusion coefficient
$\sigma$ to satisfy $\|\sigma(0)\|_{\textup{HS}(H,H)}<\infty$. In particular, the coercivity condition is not satisfied
in the important
case of additive space-time white noise where the diffusion coefficient is constantly equal to the identity operator
or a non-zero multiple hereof
(note for every $d\in\N$ that the Hilbert-Schmidt norm of the identity operator $\textup{I}_{\R^d}$ is equal
to $\|\textup{I}_{\R^d}\|_{\textup{HS}(\R^d,\R^d)}=\sqrt{d}$).
This is one central reason why almost all temporal strong convergence results in the literature (see the discussion
in the next paragraph) apply only to trace-class noise.
In particular, to the best of our knowledge, there exists no strong approximation result for stochastic Kuramoto-Sivashinsky (K-S) equations 
with space-time white noise in the 
scientific literature.
The key contribution of this work is to impose an appropriately generalized coercivity-type condition
in which the coercivity
 constant may depend on the noise process (cf.~\eqref{eq:intro_coercivity},  Theorem~\ref{theorem:main}, and Corollary~\ref{kura:cor:last}  below) and to introduce a suitable new 
explicit approximation scheme which is, roughly speaking, designed in a way so that it respects this generalized coercivity-type condition (see \eqref{eq:intro_coercivity}--\eqref{eq:intro_bound_ap} and Proposition~\ref{prop:priori_bound} below).
This new coercivity-type condition allows us to analyse
a number of additive space-time white noise driven
SEEs with superlinearly growing nonlinearities which could not be handled before.
In particular, it enables us to prove strong convergence
of the proposed scheme in the case of stochastic K-S equations (see Theorem~\ref{theorem:main} 
and Corollary~\ref{kura:cor:last} below).
The analysis of further SEEs is subject to future research.

Next we review the literature on strongly converging approximations of
additive noise-driven stochastic evolution equations (SEEs) with   superlinearly growing nonlinearities.   It was shown that the explicit Euler scheme and the linear-implicit Euler scheme do, in general, not converge strongly and numerically weakly 
in the case of  such SEEs; cf., e.g., Theorem~2.1 in~\cite{hjk11},
Theorem~2.1 in~\cite{HutzenthalerJentzenKloeden2013},
and Section~5.1 in Kurniawan~\cite{Kurniawan2014}. Fully drift-implicit Euler methods, by contrast, converge strongly for some SEEs with superlinearly growing nonlinearities;
see, e.g.,
Theorem~2.4 in Hu~\cite{Hu1996},
Theorem~2.10 in Gy{\"o}ngy \& Millet~\cite{gm05}, Theorem~7.1 in Brze{\'z}niak~\cite{Brzezniak2013}, and  Theorem~1.1 in Kov{\'a}cs et al.~\cite{kll2015}. However, to implement these methods a nonlinear equation has to be solved in each time step approximatively and this results in an additional computational cost  (especially, when the state space of the considered
SEE is high dimensional, see, e.g., Figure~4 
in~\cite{HutzenthalerJentzenKloeden2012}). Moreover, it is not  yet known whether this approximate implementation of fully drift-implicit Euler schemes converge strongly. Recently, a series of appropriately modified versions of the explicit Euler scheme have been proposed and shown to converge strongly for some SEEs with superlinearly growing nonlinearities; cf., e.g., Hutzenthaler et al.~\cite{HutzenthalerJentzenKloeden2012}, Wang \& Gan~\cite{WangGan2013},
Hutzenthaler \& Jentzen~\cite{Hutzenthaler2015}, Tretyakov \& Zhang~\cite{TretyakovZhang2013},
and Sabanis~\cite{Sabanis2013,Sabanis2013E}
in the case of finite dimensional SEEs
and cf., e.g.,
Gy{\"o}ngy et al.~\cite{GoengySabanisS2015},
Kurniawan~\cite{Kurniawan2014},
Jentzen \& Pu{\v s}nik~\cite{jp2015}, and Becker \& Jentzen~\cite{BeckerJentzen2016}
in the case of infinite dimensional SEEs.  These methods are explicit, easily realizable, and somehow tame/truncate superlinearly growing nonlinearities to prevent from strong divergence. 
However, except of Becker \& Jentzen~\cite{BeckerJentzen2016}, each of the above mentioned 
temporal strong convergence results
for implicit (see, e.g., \cite{Hu1996,gm05,Brzezniak2013,kll2015}) or explicit (see, e.g., \cite{HutzenthalerJentzenKloeden2012,WangGan2013,Hutzenthaler2015,TretyakovZhang2013,Sabanis2013,Sabanis2013E,GoengySabanisS2015,Kurniawan2014,jp2015,BeckerJentzen2016}) schemes applies
merely to trace class noise driven SEEs and excludes the important case of 
the more irregular space-time white noise.
In  Becker \& Jentzen~\cite{BeckerJentzen2016} a coercivity/Lyapunov-type condition
has been imposed and used to establish 
strong convergence rates in the case of stochastic Ginzburg-Landau equations
with additive space-time white noise; 
cf.\ (85) in \cite{BeckerJentzen2016},
Lemma~6.2 in \cite{BeckerJentzen2016},
and Corollaries~6.16--6.17 in \cite{BeckerJentzen2016}.
However, the machinery in  \cite{BeckerJentzen2016}
does not exploit the powerful negativity of the linear operator 
(cf.\ (85) in  \cite{BeckerJentzen2016} with \eqref{eq:intro_coercivity} below where the $H_{1/2}$-norm
appears on the right-hand side)
and thereby applies merely to 
stochastic Ginzburg-Landau equations 
but excludes most of the challenging additive space-time white noise driven 
SEEs with superlinearly growing nonlinearities
such as stochastic   K-S equations.

In the following we illustrate the main result 
of this article (see Theorem~\ref{theorem:main}
in Section~\ref{sec:strong_error} below)
by means of an application of this result in the case of
stochastic K-S equations
(see Corollary~\ref{kura:cor:last} in Section~\ref{sec:KS} below).
More formally, let 
$ T \in (0,\infty) $, 
$ \xi \in H^1_P( (0,1), \R) $,
$ H = L^2( (0,1); \R ) $,
let $ F \colon L^4( (0,1); \R ) \rightarrow H^{ - 1 }( (0,1), \R ) $
be the function with the property that for all
$ v \in L^4( (0,1); \R ) $
it holds that
$ 
  F(v) 
  = 
  v - 
  \nicefrac{ 
    1
  }{ 2 } \, ( v^2 )' 
$,
let $ \mathcal{A} \colon D( \mathcal{A} ) \subseteq H \rightarrow H $
be the Laplacian with periodic boundary conditions on $ H $,
let $ A \colon D(A) \subseteq H \rightarrow H $
be the linear operator which satisfies  for all $ v \in D(A) $ that $ D(A) = D( \mathcal{A}^2 ) $
and
$
  A v = - \mathcal{A}^2 v - \mathcal{A} v - v
$,
let $ B \in L( H, H^{ - 1 }( (0,1) , \R ) ) $
be the linear operator with the property that for all 
$ v \in H $ 
it holds that
$
  B v = v'
$,
let $ (\Omega, \F, \P ) $ be a probability space
with a normal filtration $ (\F_t)_{ t\in[0,T] } $,
and let $ (W_t)_{ t\in[0,T] } $ be an 
$ \Id_{ H } $-cylindrical
$ (\F_t)_{t\in[0,T]} $-Wiener process.
The above assumptions ensure that
there exists an up to indistinguishability
unique $ (\F_t)_{t\in[0,T]} $-adapted stochastic process
$ X \colon [0,T] \times \Omega \rightarrow L^4( (0,1) ; \R ) $
with continuous sample paths
which satisfies that for all $ t \in [0,T] $
it holds $ \P $-a.s.\ that
\begin{align}
\label{eq:intro_SPDE_mild} 
&
  X_t
=
  e^{tA} \, \xi
  +
  \int_0^t
    e^{(t-s)A}  \,
    F( X_s ) 
  \, ds
  +
  \int_0^t
    e^{(t-s)A} \, B 
  \, dW_s
\end{align}
(cf., e.g., Duan \& Ervin~\cite{DuanErvin2001_Kuramoto_Sivashinsky}).
The stochastic process $ X $ is thus 
a mild solution of the stochastic K-S  equation
\begin{align}
\label{eq:intro_SPDE}
&
  \tfrac{ \partial }{ \partial t }
  X_t(x)
  =
  -
  \tfrac{\partial^4}{\partial x^4}
  X_t(x)
  -
  \tfrac{\partial^2}{\partial x^2}
  X_t(x)
  -
  X_t( x )
  \cdot 
  \tfrac{ \partial }{ \partial x } X_t(x)
  +
  \tfrac{ \partial^2 }{ \partial x \partial t } W_t( x )
\end{align}
with $X_t(0)=X_t(1)$, $X_t'(0)=X_t'(1)$, $X_t''(0)=X_t''(1)$, $X_t^{(3)}(0)=X_t^{(3)}(1)$, and $ X_0(x) = \xi(x) $ for
$ x \in (0,1) $, $ t \in [0,T] $. Note that the noise in \eqref{eq:intro_SPDE_mild} and \eqref{eq:intro_SPDE} is quite rough in the sense that $\frac{ \partial^2 }{ \partial x \partial t } W_t( x )$, $x \in (0,1)$, $ t \in [0, T]$, is the distributional space derivative of the space-time white noise $\frac{\partial}{\partial t} W_t(x)$, $x \in (0,1)$, $ t \in [0, T]$.
In this article we introduce the following 
\emph{nonlinearity-truncated accelerated exponential Euler-type 
scheme} to
approximate the solution process $ X $
of the SPDE~\eqref{eq:intro_SPDE}. Let $
  ( e_n )_{ n \in \Z } \subseteq H
$,
$
  ( P_n )_{ n \in \N } \subseteq L(H)
$,
$ 
  ( h_n )_{ n \in \N } \subseteq (0,\infty) 
$, $ \varrho   \in (\nicefrac{1}{16}, \nicefrac{3}{32})$,
$ \chi \in (0, \nicefrac{\varrho}{2}- \nicefrac{1}{32} ] $ satisfy 
for all 
$ n \in \N $, $ v \in H $
that
$
  e_0 = 1
$,
$
  e_n(\cdot) = \sqrt{2} \cos( 2 n \pi ( \cdot ) )
$,
$
  e_{ - n }(\cdot) = \sqrt{2} \sin( 2 n \pi ( \cdot ) )
$,
$
  P_n( v ) =
  \sum_{ k =-n }^n
  \left< e_k, v \right>_H
  e_k
$,
$
  \limsup_{ k \to \infty } h_k = 0
$,
let 
$ 
  \lfloor \cdot \rfloor_h 
  \colon \R \rightarrow \R 
$, 
$ h \in (0,\infty) $,
be the mappings which satisfy
for all $ h \in (0,\infty) $, $ t \in \R $ that 
$
  \lfloor t \rfloor_h 
=
  \max\!\left(
    (-\infty, t]
    \cap 
    \{ 0, h, -h, 2h, -2h, \ldots \} 
  \right)
$,
and
let $ \mathcal{O}^n \colon [0,T] \times \Omega \to P_n( H ) 
$,
$ n \in \N $, and
$ \mathcal{X}^n \colon [0,T] \times \Omega \to P_n( H ) 
$,
$ n \in \N $,
be stochastic processes which satisfy 
that for all $ n \in \N $, $ t \in [0,T] $
it holds $ \P $-a.s.  that $\mathcal{O}_t^n = \int_{0}^{t} P_n \, e^{(t-s)A} \, B \, dW_s$ and 
\begin{equation}
\begin{split}
\label{eq:acc_exp_scheme}
  \mathcal{X}^n_t
& =
 P_n \, e^{ t A } \, \xi
  +
  \int_0^t P_n \,
    e^{ ( t - s ) A } 
      \,
      \mathbbm{1}_{ 
        \{
          \| 
            ( - A )^{ \varrho } \mathcal{X}^n_{ \lfloor s \rfloor_{ h_n } }
          \|_H
          +
          \| 
            ( - A )^{ \varrho }[ \mathcal{O}^n_{ \lfloor s \rfloor_{ h_n } } +P_n  e^{ \lf s \rf_{ h_n } A } \xi]
          \|_H
          \leq
          | h_n |^{-\chi}
        \}
      } \,
      F( \mathcal{X}^n_{ \lfloor s \rfloor_{ h_n } } ) \, ds
      +
 \mathcal{O}_t^n 
  .
\end{split}
\end{equation}
In Corollary~\ref{kura:cor:last} in Section~\ref{sec:KS} below we demonstrate that the approximation scheme \eqref{eq:acc_exp_scheme} converges strongly to the solution of the SPDE \eqref{eq:intro_SPDE}. More precisely, Corollary~\ref{kura:cor:last} (with $\beta= \nicefrac{3}{16}$, $\eta=\kappa=1$, $\varrho= \varrho$, $b_k=0$, $\tilde{b}_k= 2 k \pi$, $\mathcal{X}^n = \mathcal{X}^n$, $\mathcal{O}^n = \mathcal{O}^n$, $X =X$ for $k \in \Z$, $n \in \N$ it the notation of Corollary~\ref{kura:cor:last})  proves that for all $p \in (0, \infty)$ it holds that
\begin{equation}
\label{eq:strong_convergence_space_time}
\limsup\nolimits_{ n \to \infty }
\sup\nolimits_{ t \in [0,T] }
\E\big[ 
\| 
X_t - \mathcal{X}^n_t 
\|_H^p
\big]
= 0
.
\end{equation}
Corollary~\ref{kura:cor:last} follows from an application of Theorem~\ref{theorem:main} below, which is the main result of this paper. Theorem~\ref{theorem:main} establishes strong convergence for a more general class of SPDEs as well as for a more general type of approximation schemes.

We now add a few comments on the approximation scheme~\eqref{eq:acc_exp_scheme} and on key ideas in the proof of Corollary~\ref{kura:cor:last} and Theorem~\ref{theorem:main}, respectively. First, we note that the approximation scheme~\eqref{eq:acc_exp_scheme} does not temporally discretize the semigroup $(e^{tA})_{t \in [0, \infty)}$ appearing in \eqref{eq:intro_SPDE_mild} and is thus an appropriate modification of the accelerated exponential Euler scheme in Section~3 in Jentzen \& Kloeden~\cite{Jentzen2009c} (cf., e.g., also Section~4 in Jentzen \& Kloeden~\cite{jk09d} for an overview and e.g., Lord \& Tambue~\cite{lt13} and Wang \& Qi~\cite{wang2015note} for further results on accelerated exponential Euler approximations). This lack of discretization of the semigroup in the stochastic integral \eqref{eq:intro_SPDE_mild} has been proposed in  Jentzen \& Kloeden~\cite{Jentzen2009c}  to obtain an approximation scheme which converges under suitable assumptions with a significant higher convergence rate than previously analyzed approximation schemes such as the linear implicit Euler scheme or the exponential Euler scheme (cf., e.g., Theorem~3.1 in Jentzen \& Kloeden~\cite{Jentzen2009c}, Theorem~1 in~\cite{Jentzen2011b}, Theorem~3.1 in Wang \& Qi~\cite{wang2015note}, and Theorem~3.1 in Qi \& Wang~\cite{qi2016accelerated}). In this article the lack of discretization of the semigroup in the non-stochastic integral in \eqref{eq:intro_SPDE_mild} is employed for a different purpose, that is, here this lack of discretization is used to obtain a scheme that inherits an appropriate a priori estimate from the exact solution process of the SPDE~\eqref{eq:intro_SPDE}. More specifically, we observe that the nonlinearity $F \colon  L^4( (0,1); \R ) \rightarrow H^{ - 1 }( (0,1), \R )$ appearing in \eqref{eq:intro_SPDE_mild} satisfies that there exist suitable measurable functions $\phi, \Phi \colon C([0,1], \R) \to [0, \infty)$ and a real number $\varphi \in [0, 1)$ such that for all $v, w \in H_1$ it holds that
\begin{align}\label{eq:intro_coercivity}
\left< v, F( v + w ) \right>_H \leq  \phi(w) \| v \|^2_H+ \ctwo\| v \|^2_{ H_{ 1 / 2 } }+ \Phi( w )
\end{align}
(see Lemma~\ref{coer:kura} for the proof of \eqref{eq:intro_coercivity} and see also the proof of Corollary~\ref{kura:cor:last} for the specific choice of $\phi$, $\Phi$, and $\varphi$). Inequality \eqref{eq:intro_coercivity}, in turn, ensures that for every continuous stochastic process $O \colon [0, T] \times \Omega \to C([0,1], \R)$ with $\forall \, u \in [0, T] \colon \P(O_u= \int_0^u e^{(u-s)A} \, B\, dW_s)=1$ and every $t \in [0, T]$ it holds $\P$-a.s. that 
\begin{align}\label{eq:intro_bound}
\|X_t\|_H \leq \|O_t\|_H + \sqrt{ e^{\int_0^t 2\phi(O_s) \, ds } \, \| \xi\|^2_H +   2\int_0^te^{ \int_s^t 2\phi( O_u  ) \, du}  \,    \Phi( O_s ) \, ds}.
\end{align}
Note that \eqref{eq:intro_coercivity} is an appropriate generalized coercivity-type condition for the SPDE under consideration (cf., e.g., Chapter~4 in Pr{\'e}v{\^o}t \& R{\"o}ckner~\cite{pr07}). A key contribution of this paper is to reveal that the approximation scheme \eqref{eq:acc_exp_scheme} inherits \eqref{eq:intro_bound} in the sense that there exists $\theta \in (0, \infty)$ such that for all $t \in [0, T]$, $n \in \N$ it holds $\P$-a.s. that  
\begin{align}\label{eq:intro_bound_ap}
\begin{split}
&\|\mathcal{X}^n_t\|_H \leq \|\mathcal{O}^n_t\|_H + \sqrt{ e^{\int_0^t 2\phi(\mathcal{O}^n_{\lf s \rf_{h_n}}) \, ds } \, \| \xi\|^2_H +   2\int_0^te^{ \int_s^t 2\phi( \mathcal{O}^n_{\lf u \rf_{h_n}} ) \, du}  \left[   \Phi( \mathcal{O}^n_{\lf s \rf_{h_n}}) +  \theta \,|h_n|^{\nicefrac{1}{\theta} } \right] ds}
\end{split}
\end{align}
(see Proposition~\ref{prop:priori_bound} for the proof of \eqref{eq:intro_bound_ap}). Strong convergence results for explicit (see, e.g., \cite{HutzenthalerJentzenKloeden2012,WangGan2013,Hutzenthaler2015,TretyakovZhang2013,Sabanis2013,Sabanis2013E,GoengySabanisS2015,Kurniawan2014,jp2015,BeckerJentzen2016}) and implicit (see, e.g., \cite{Hu1996,gm05,Brzezniak2013,kll2015}) numerical approximation schemes for SEEs in the literature impose coercivity-type assumptions in which $\phi$ and $\Phi$ are constants (cf., e.g., Assumption~(B)' in Hu~\cite{Hu1996}, (C2) in Gy{\"o}ngy \& Millet~\cite{gm05}, Section~1 in Hutzenthaler et al.~\cite{HutzenthalerJentzenKloeden2012}, Assumption~2.1 in Wang \& Gan~\cite{WangGan2013}, (2.11) in Hutzenthaler \& Jentzen~\cite{Hutzenthaler2015}, Assumption~2.1 in Tretyakov \& Zhang \cite{TretyakovZhang2013}, Section~7 in Brze{\'z}niak~\cite{Brzezniak2013}, (A-1) in Sabanis~\cite{Sabanis2013},  (A-4) in Sabanis~\cite{Sabanis2013E}, Assumption~1 in Gy{\"o}ngy et al.~\cite{GoengySabanisS2015}, (4.79) in Kurniawan~\cite{Kurniawan2014}, Section~7.4 in Jentzen \& Pu{\v s}nik~\cite{jp2015}, Section~3.1 in Kov{\'a}cs et al.~\cite{kll2015}, and (85) in Becker \& Jentzen~\cite{BeckerJentzen2016}). Such a coercivity-type condition is not fulfilled in the case of a number of nonlinear SPDEs with rough noise such as \eqref{eq:intro_SPDE}. In particular, none of the above mentioned references applies to the stochastic K-S equation \eqref{eq:intro_SPDE} and Theorem~\ref{theorem:main} and Corollary~\ref{kura:cor:last} below, respectively, are -- to the best of our knowledge -- the first strong approximation results for the stochastic K-S equation~\eqref{eq:intro_SPDE}. We would also like to add that in the above mentioned articles on accelerated exponential Euler approximations it was crucial to avoid the discretization of the semigroup in the stochastic integral while our analysis exploits the fact that the semigroup  in the non-stochastic integral in \eqref{eq:intro_SPDE_mild} is not discretized but allows discretizations of the semigroup in the stochastic integral (cf.~Theorem~\ref{theorem:main} in Section~\ref{sec:strong_error}). Next we observe that the approximation scheme \eqref{eq:acc_exp_scheme} can be easily realized on a computer. More formally, note that for all $n \in \N$, $k\in \N \cap (-\infty, \nicefrac{T}{h_n}-1]$ it holds  $ \P $-a.s.\ that
\begin{equation}\label{eq:intro_recursion}
\begin{split}
\mathcal{O}^n_{ ( k + 1 ) h_n } & =
e^{ h_n A }
\mathcal{O}^n_{ k h_n } 
+ \int_{k h_n}^{(k+1)h_n} P_n \, e^{((k+1)h_n -s)A} \, B \, dW_s ,\\
\mathcal{X}^n_{ ( k + 1 ) h_n } & =
e^{ h_n A }
\mathcal{X}^n_{ k h_n } 
+
P_n \, A^{ - 1 } (e^{h_n A}- \mathrm{Id}_H) \, \mathbbm{1}_{ 
	\{
	\| 
	( - A )^{ \varrho } \mathcal{X}^n_{ k h_n}
	\|_H
	+
	\| 
	( - A )^{ \varrho }[ \mathcal{O}^n_{ k h_n }   +P_n  e^{ k h_n A } \xi]
	\|_H 
	\leq
	| h_n |^{-\chi}
	\}
} \, F( \mathcal{X}^n_{ k h_n } ) \\
& \quad + \mathcal{O}^n_{ ( k + 1 ) h_n } -
e^{ h_n A }
\mathcal{O}^n_{ k h_n }, 
\end{split}
\end{equation}
and \eqref{eq:intro_recursion} can be used directly in an implementation. We illustrate this in Figures~\ref{fig:sample_path} and \ref{fig:matlab_code} where three realizations of $X_T(\omega)$, $\omega \in \Omega$, are calculated approximatively with the numerical approximation method \eqref{eq:acc_exp_scheme} in the case where $T=1$, $n=10000$, $h_n= 1/\sqrt{n}$, $\varrho = \nicefrac{5}{64}$, $\chi= \nicefrac{1}{128}$, and $\xi=0$. The {\sc Matlab} code used to generate Figure~\ref{fig:sample_path} can be found in Figure~\ref{fig:matlab_code} below.  The approximation scheme \eqref{eq:acc_exp_scheme} is thus an easily implementable strongly convergent approximation method for the SPDE~\eqref{eq:intro_SPDE}. In particular, to the best of our knowledge, the scheme~\eqref{eq:acc_exp_scheme} is the first approximation method in the scientific literature that has been shown to converge strongly to the solution of the stochastic K-S equation~\eqref{eq:intro_SPDE}.
\begin{figure}
	\centering
	\includegraphics[width=.9\textwidth]{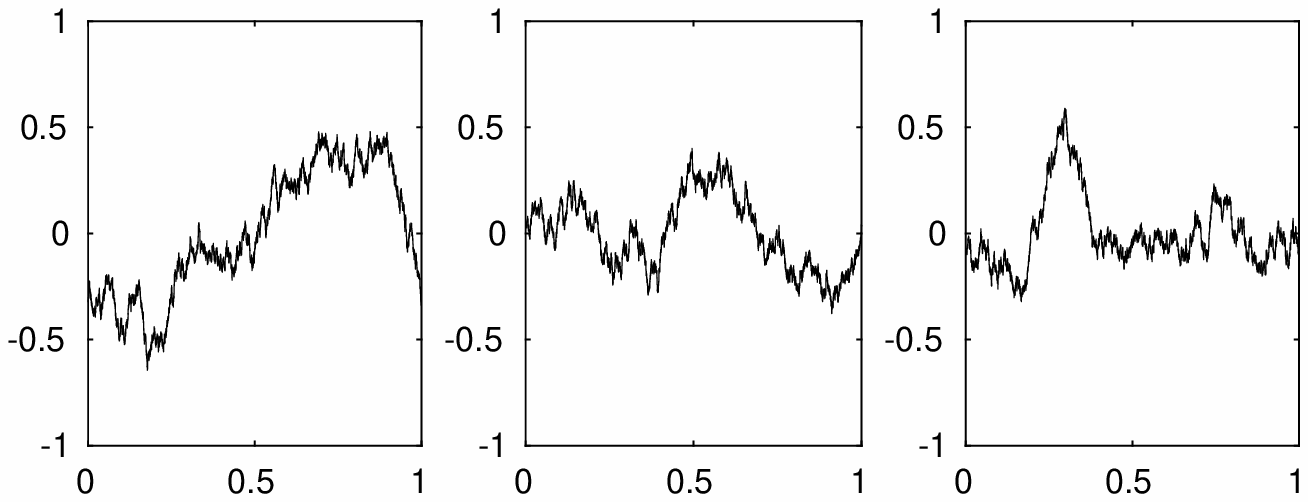}
	\caption{Result of a call of the {\sc Matlab} code in Figure~\ref{fig:matlab_code}.}
	\label{fig:sample_path}
\end{figure}

The remainder of this article is organized as follows.
In Section~\ref{sec:a_priori} the required a priori moment bounds 
for the nonlinearity-truncated 
approximation schemes are established and
in Section~\ref{sec:pathwise_error} 
the error analysis is performed in the pathwise sense 
under the hypothesis of suitable a priori bounds
for the approximation processes. Section~\ref{sec:strong_error} combines the 
results of Section~\ref{sec:a_priori}
and Section~\ref{sec:pathwise_error} and 
thereby establishes strong convergence in Theorem~\ref{theorem:main},
which is the main result of this article.
The analysis in Sections~\ref{sec:a_priori}--\ref{sec:strong_error}
is carried out for abstract stochastic evolution equations
on separable Banach and Hilbert spaces, respectively.
Section~\ref{sec:KS} then verifies that the 
assumptions of Theorem~\ref{theorem:main}
in Section~\ref{sec:strong_error} are satisfied in the case of concrete 
stochastic partial differential equations of the type~\eqref{eq:intro_SPDE} and, in particular,
establishes Corollary~\ref{kura:cor:last}.
\begin{figure}
\centering
\begin{lstlisting}
rng('default');
N = 10000;  h = 1/sqrt(N);  chi = 1/128;  varrho = 5/64;
A = -(-N:N).^4.*16*pi^4+(-N:N).^2.*4*pi^2-1;
S = sqrt((exp(2*h*A)-1)./A/2)*2*pi.*(-N:N);  X = zeros(3,2*N+1);
for m = 1:3   
  Y = zeros(1,2*N+1);  O_old = zeros(1,2*N+1);
  for n = 1:sqrt(N)
    O_new = exp(A*h).*O_old+S.*randn(1,2*N+1);
    y = [Y(N+1),1i*Y(N:-1:1)*sqrt(2)+Y(N+2:end)*sqrt(2),zeros(1,N)];  
    y = real(fft(y));  y1 = ifft(y.^2);  
    y2 = [imag(y1(N+1:-1:2))*sqrt(2),real(y1(1)),real(y1(2:N+1))*sqrt(2)];
    FY = (Y-pi*fliplr(y2).*(-N:N))...
    .*(norm((-A).^varrho.*Y)+norm((-A).^varrho.*O_old)<=h^(-chi));
    Y = exp(A*h).*Y+A.^(-1).*(exp(A*h)-1).*FY+O_new-exp(A*h).*O_old; 
    O_old = O_new;
  end
  X(m,:) = [Y(N+1),1i*Y(N:-1:1)*sqrt(2)+Y(N+2:end)*sqrt(2),zeros(1,N)];
  X(m,:) = real(fft(X(m,:)));
end
figure(1);  subplot(1,3,1);
plot((1:2*N+1)/(2*N+2),X(1,:),'k','LineWidth',0.3);  ylim([-1 1]);  
subplot(1,3,2);  
plot((1:2*N+1)/(2*N+2),X(2,:),'k','LineWidth',0.3);  ylim([-1 1]);
subplot(1,3,3);
plot((1:2*N+1)/(2*N+2),X(3,:),'k','LineWidth',0.3);  ylim([-1 1]);
\end{lstlisting}
\caption{ {\sc Matlab} code for 
	Figure~\ref{fig:sample_path}. }
\label{fig:matlab_code}
\end{figure}

\subsection{Notation}

Throughout this article the following notation is used. 
We denote by $ \mathbb{N} = \{1, 2, 3, \ldots \}$ the set of all natural numbers. 
For two sets $ A $ and $ B $ we denote by 
$ \mathbb{M}( A, B ) $ the set of all mappings 
from $A$ to $B$. For a set $A$ we denote by $\mathcal{P}(A)$ the power set of $A$, we denote by $\#_A \colon \mathcal{P}(A) \to [0, \infty]$ the counting measure on $A$, 
and we denote by $\mathcal{P}_0(A)$ the set given by $\mathcal{P}_0(A)= \{ B \in \mathcal{P}(A) \colon \#_A(B) < \infty\}$. For two measurable spaces $( A, \mathcal{A})$ and $( B, \mathcal{B})$  we denote by $ \mathcal{M}(\mathcal{A}, \mathcal{B})$ the set of all $\mathcal{A} / \mathcal{B}$-measurable mappings. Let $ \Gamma \colon (0, \infty) \to (0, \infty)$ be the function with the property that for all $x \in (0, \infty)$ it holds that $\Gamma(x)= \int_0^{\infty} t^{(x-1)} \, e^{-t} \, dt$ (Gamma function). Let $ \mathrm{E}_r \colon [0, \infty) \to [0, \infty)$, $r \in (0, \infty),$ be the functions with the property that for all $ r \in (0, \infty)$, $x \in [0, \infty)$ it holds that $ \mathrm{E}_r[x]= \sum_{n=0}^{\infty} \tfrac{x^{nr}}{\Gamma(nr+1)}$ (cf. Chapter~7 in Henry~\cite{h81} and see, e.g.,  Definition~1.3.1 in \cite{Jentzen2015SPDElecturenotes}). 
For a topological space $ (X, \tau) $ 
we denote by $ \mathcal{B}(X) $ the Borel sigma-algebra of 
$ (X, \tau) $.
For a set $A$ we denote by $\Id_A \colon A \to A$ the mapping with the property that for all $ a \in A$ it holds that $\Id_A(a)=a$ 
(identity mapping on $ A $). For a set $A \in \mathcal{B}(\R)$ we denote by $\lambda_A \colon \mathcal{B}(B) \to [0, \infty]$ the Lebesgue-Borel measure on $A$. For a measure space $(\Omega, \mathcal{F}, \mu)$, a measurable space $(S, \mathcal{S})$, a set $R \subseteq S$, and a function $f \colon \Omega \to S$ we denote by $\left[ f \right]_{\mu, \mathcal{S}} $ the set given by $\left[f \right]_{\mu, \mathcal{S}} = \left\{ g \in \mathcal{M}(\mathcal{F}, \mathcal{S}) \colon \left( \exists \, A \in \mathcal{F} \colon \mu(A)=0 \,\, \text{and} \,\, \{ w \in \Omega \colon f(w) \neq g(w)\} \subseteq A \right) \right\}$. We denote by $ \lf \cdot \rf_h \colon \R \to \R$, $ h \in (0, \infty)$, and $\lceil \cdot \rceil_h \colon \R \to \R$, $h \in (0, \infty)$, the mappings with the property that for all $t \in \R$, $h \in (0, \infty)$ it holds that $\lf t \rf_h = \max( (-\infty, t] \cap \{0, h, -h, 2h, -2h, \ldots\} )$ and $\lceil t \rceil_h= \min([t, \infty), \{0, h, -h, 2h, -2h, \ldots\})$. For real numbers $p \in [1, \infty)$, $\theta \in (0,1)$ and a $\mathcal{B}((0,1)) /\mathcal{B}(\R)$-measurable function $v \colon (0,1) \to \R$ we denote by $\|v\|_{\W^{\theta, p}((0,1), \R)}$  the extended real number given by  
\begin{align}
\|v\|_{\W^{\theta, p}((0,1), \R)} = \left[ \int_0^1 |v(x)|^p \, dx + \int_0^1 \int_0^1 \frac{|v(x)-v(y)|^p}{|x-y|^{1+ \theta p}} \, dx \, dy\right]^{\frac{1}{p}}.
\end{align}

\section{A priori bounds}
\label{sec:a_priori}

In this section we accomplish in Proposition~\ref{prop:priori_bound} and Corollary~\ref{cor:bounds} below appropriate a priori bounds for our approximation scheme. Before we establish Proposition~\ref{prop:priori_bound} and Corollary~\ref{cor:bounds} below, we present in Lemma~\ref{lem:YX_in_DA}, Lemma~\ref{lemma:cont}, Lemma~\ref{lemma:meas}, and Lemma~\ref{lemma:cond} a few elementary auxiliary results for Proposition~\ref{prop:priori_bound} and Corollary~\ref{cor:bounds}.

\subsection{Regularity of the numerical approximations}
\label{subsec:regularity}

The following elementary and well-known lemma is a slight generalization of Lemma 3.3 in Becker \& Jentzen~\cite{BeckerJentzen2016}. In particular, the proof of Lemma \ref{lem:YX_in_DA} is a slight adaptation of the proof of Lemma 3.3 in Becker \& Jentzen~\cite{BeckerJentzen2016}. 

\begin{lemma} \label{lem:YX_in_DA}
Let $(V, \left\| \cdot \right\|_V)$ be an $\R$-Banach space, let $A \colon D(A) \subseteq V \to V$ be a generator of a strongly continuous analytic semigroup with $ \mathrm{spectrum}(A) \subseteq \{ z \in \mathbb{C} \colon \mathrm{Re}(z) <0 \}$, and let $T, h \in (0, \infty)$, $ Y, Z \in \mathbb{M}([0,T], V)$ satisfy for all $t \in [0,T]$ that $Y_{t} = \int_{0}^{t} e^{(t-s)A} Z_{\fl{s}}  \, ds$. Then
\begin{enumerate}[(i)] 
\item 
it holds for all $ t \in [0,T] $ that $ Y_{t}\in D(A) $,
\item\label{eq:YX_in_DA_2}
 it holds that the function $ [0,T] \ni t \mapsto Y_{t} \in D(A) $ is continuous,
\item it holds that the function $ [0,T] \ni t \mapsto Y_{t}  \in V $ is Lipschitz continuous, 
\item it holds that the function $ [0,T] \backslash \{ 0, h, 2h, \ldots\} \ni t \mapsto Y_{t} \in V $ is continuously differentiable, 
\item it holds for all $ t \in [0,T] \backslash \{ 0, h, 2h, \ldots \} $ that $ \frac{ d Y_{t} }{ d t } = A Y_{t}  + Z_{\fl{t}} $, and
\item it holds for all  $ t \in [0,T] $ that $ Y_{t} = \int_0^t \big[ A Y_{s} + Z_{\fl{s}}  \big] \, ds $. 
\end{enumerate}
	
\end{lemma}

\begin{lemma}\label{lemma:cont}
Let $( V, \left\| \cdot \right\|_V )$ be an $ \R $-Banach space, let $ A \colon D(A) \subseteq V \to V$ be a generator of a strongly continuous semigroup, and let $ T \in (0,\infty) $, $ Y \in \mathbb{M}( [0,T] , V ) $, $ Z \in \mathcal{M}( \mathcal{B}( [0,T]), \mathcal{B}(V)  ) $ satisfy for all $ t \in [0,T] $ that $\sup_{ s \in (0,T) } \| Z_s \|_V < \infty$ and $Y_t = \int_0^t e^{ ( t - s ) A } \, Z_s \, ds$. Then it holds that $ Y $ is right-continuous.
\end{lemma}
\begin{proof}[Proof of Lemma \ref{lemma:cont}.]
Note that for all $t \in [0, T), h \in (0, T-t]$ it holds that
\begin{align}
\begin{split}
&\|Y_{t+h} - Y_t \|_V = \left\| \int_0^{t+h} e^{ ( t+h - s ) A } \, Z_s \, ds - \int_0^t e^{ ( t - s ) A } \, Z_s \, ds \right\|_V\\
&\leq \int_0^t\left\| \left( e^{ ( t+h - s ) A } - e^{ ( t - s ) A }  \right) Z_s \right\|_V ds + \int_{t}^{t+h} \left\|e^{ ( t+h - s ) A } \, Z_s \right\|_V ds \\
& \leq \int_0^t\left\| \left( e^{ ( t+h - s ) A } - e^{ ( t - s ) A }  \right) Z_s \right\|_V ds + h \left( \sup\nolimits_{s \in [0,T]} \|e^{s A}\|_{L(V)} \right) \left(\sup\nolimits_{ s \in (0,T) } \| Z_s \|_V \right).
\end{split}
\end{align}
Combining Lebesgue's theorem of dominated convergence with the assumption that $ A \colon D(A) \subseteq V \to V$ is a generator of a strongly continuous semigroup and the assumption that $\sup_{ s \in (0,T) } \| Z_s \|_V < \infty$ hence yields that for all $t \in [0, T)$ it holds that
\begin{align}
\limsup_{h \searrow 0} \|Y_{t+h} - Y_t \|_V =0.
\end{align}
The proof of Lemma \ref{lemma:cont} is thus completed.
\end{proof}


\begin{lemma}\label{lemma:meas}
Let $ ( V, \left\| \cdot \right\|_V ) $ be a separable $\R$-Banach  space, let $ ( \Omega, \mathcal{F}, \P ) $ be a probability space, let $ A \in D(A) \subseteq V \to V $ be a generator of a strongly continuous  semigroup, let $O \colon [0,T] \times \Omega \to V $ be a stochastic process, and let $T, h \in (0, \infty)$, $Y \in \mathbb{M}([0, T] \times \Omega, V)$, $ F \in \mathcal{M}(\mathcal{B}(V^2), \mathcal{B}(V))$ satisfy for all $t \in [0,T]$ that $Y_t = \int_{0}^{t} e^{(t-s)A} \, F(Y_{\lf s \rf_h}, O_{\lf s \rf_h}) \, ds + O_t$. Then it holds that $Y \colon [0,T] \times \Omega \to V$ is a stochastic process and it holds that $Y - O \colon [0,T] \times \Omega \to V$ is a stochastic process with right-continuous sample paths.
\end{lemma}
\begin{proof}[Proof of Lemma~\ref{lemma:meas}.]
First, we claim that for all $k \in \N_0 $ it holds that
\begin{align}\label{lem:meas}
\forall \, t \in [0, \min\{T, kh\}] \colon Y_t \in \mathcal{M}(\mathcal{F}, \mathcal{B}(V)).
\end{align}
In the following we prove \eqref{lem:meas} by induction on $k \in \N_0$. The base case $k=0$ follows from the fact that $Y_0= O_0 \in \mathcal{M}(\mathcal{F}, \mathcal{B}(V))$. 
Next observe that the fact that 
$ \forall \, t \in [0,T] \colon O_t \in \mathcal{M}( \mathcal{F}, \mathcal{B}( V ) ) $ 
shows that for all 
$ k \in \N_0 $, $ t \in [ \min\{T, kh\}, \min\{T, (k+1)h\}] $ 
with 
$
  \forall \, s \in [0, \min\{ T, k h \} ] \colon 
  Y_s \in \mathcal{M}( \mathcal{F}, \mathcal{B}( V ) )
$ 
it holds that
\begin{align}\label{meas:ind}
\begin{split}
Y_t  &=  e^{(t-\min\{T, kh\}) A} \, Y_{\min\{T, kh\}} + \int_{\min\{T, kh\}}^{t}  e^{ ( t - s ) A } \, F(Y_{\lf s \rf_h}, O_{\lf s \rf_h}) \, ds  \\
& \quad + O_t - e^{(t-\min\{T, kh\}) A} \, O_{\min\{T, kh\}} \\
& = e^{(t-\min\{T, kh\}) A} \, Y_{\min\{T, kh\}} + \int_{\min\{T, kh\}}^{t}  e^{ ( t - s ) A } \,F(Y_{\min\{ \lf T \rf_h, kh\}}, O_{\min\{ \lf T \rf_h, kh\}}) \, ds  \\
& \quad  +  O_t - e^{(t-\min\{T, kh\}) A} \, O_{\min\{T, kh\}} \in \mathcal{M}( \mathcal{F}, \mathcal{B}( V ) ).
\end{split}
\end{align}
The induction step $\N_0 \ni k \to k+1 \in \N$ follows from \eqref{meas:ind} and the induction hypothesis. Induction hence proves \eqref{lem:meas}. In the next step we observe that \eqref{lem:meas} together with the assumption that  $O \colon [0,T] \times \Omega \to V $ is a stochastic process ensures that $ Y - O \colon [0,T] \times \Omega \to V $ is also a stochastic process. This and Lemma \ref{lemma:cont} show that  $Y - O \colon [0,T] \times \Omega \to V$ is a stochastic process with right-continuous sample paths. The proof of Lemma \ref{lemma:meas} is thus completed.
\end{proof}

\subsection{Semi-globally Lipschitz continuous functions}
\label{subsec:lipschitz}

\begin{lemma}
\label{lemma:cond}
Let $ ( V, \left\| \cdot \right\|_V) $, 
$(\mathcal{V}, \left\| \cdot \right\|_{\mathcal{V}})$, 
$(W, \left\| \cdot \right\|_W)$, and $(\mathcal{W}, \left\| \cdot \right\|_{\mathcal{W}})$ 
be normed $\R$-vector spaces with $V \subseteq \mathcal{V}$ continuously
and $W \subseteq \mathcal{W}$ continuously 
and let $ \epsilon, \theta \in [0,\infty) $, $ \varepsilon, \vartheta \in (0, \infty) $, 
$ F \in \mathbb{M}(V, W)$ satisfy for all $v, w \in V$ that $\|F(v)- F(w)\|_W \leq \epsilon \,(1+ \|v\|_{\mathcal{V}}^{\varepsilon} + \|w\|_{\mathcal{V}}^{\varepsilon}) \, \|v-w\|_{\mathcal{V}}$, $\vartheta= 2 \varepsilon$, and
\begin{multline}
\theta = \max \!\bigg\{ \! 3 \, \epsilon^2 \bigg[\!\sup_{u \in W\backslash\{0\}}\tfrac{\|u\|^2_{\mathcal{W}}}{\|u\|^2_W}\bigg] \bigg[1+ \! \sup_{u \in V\backslash\{0\}}\tfrac{\|u\|_{\mathcal{V}}^{2 \varepsilon}}{\|u\|_V^{2 \varepsilon}} \bigg]\big(1+2^{\max\{2\varepsilon-1, 0\}}\big) , \\
 \big(8\epsilon^2 + 2\|F(0)\|_W^2\big) \max \!\bigg\{ 1, \sup_{u \in V\backslash\{0\}}\tfrac{\|u\|_{\mathcal{V}}^{2+2 \varepsilon}}{\|u\|_V^{2+2 \varepsilon}} \bigg\} \! \bigg\}.
\end{multline}
Then  it holds for all $v, w \in V$ that $\|F(v)\|_W^2 \leq \cthree \max\{ 1, \|v\|_V^{2 + \cfour} \}$ and
\begin{align}
\|F(v)- F(w)\|_{\mathcal{W}}^2 \leq \theta \max\{1, \|v\|_V^{\vartheta}\}\|v-w\|_{\mathcal{V}}^2 + \theta \, \|v-w\|_{\mathcal{V}}^{2+\vartheta}.
\end{align}
\end{lemma}
\begin{proof}[Proof of Lemma~\ref{lemma:cond}]
Combining the fact that $\forall \, a, b, c \in \R \colon (a+b+c)^2 \leq 3(a^2+b^2+c^2)$ and the fact that $\forall \, a, b, x \in [0, \infty) \colon (a+b)^{x} \leq 2^{\max\{x-1, 0\}} (a^x + b^x)$ with the triangle inequality proves that for all $v, w \in V$ it holds that
\begin{align}\label{eq:lemma:cond}
\begin{split}
&\|F(v)- F(w)\|_{\mathcal{W}}^2  \leq  \bigg[\sup_{u \in W\backslash\{0\}}\tfrac{\|u\|_{\mathcal{W}}}{\|u\|_W}\bigg]^2 \|F(v)- F(w)\|_W^2 \\
&\leq   3  \, \epsilon^2 \bigg[\!\sup_{u \in W\backslash\{0\}}\tfrac{\|u\|_{\mathcal{W}}}{\|u\|_W} \bigg]^2 \big(1+ \|v\|_{\mathcal{V}}^{2\varepsilon} + \|w\|_{\mathcal{V}}^{2\varepsilon}\big)  \|v-w\|_{\mathcal{V}}^2\\
& \leq 3 \, \epsilon^2 \bigg[\!\sup_{u \in W\backslash\{0\}}\tfrac{\|u\|_{\mathcal{W}}}{\|u\|_W} \bigg]^2 \Big(1+ (1+2^{\max\{2\varepsilon-1, 0\}})\|v\|_{\mathcal{V}}^{2\varepsilon} + 2^{\max\{2\varepsilon-1, 0\}} \|v-w\|_{\mathcal{V}}^{2\varepsilon}\Big)  \|v-w\|_{\mathcal{V}}^2 \\
& \leq 3 \, \epsilon^2\bigg[\!\sup_{u \in W\backslash\{0\}}\tfrac{\|u\|_{\mathcal{W}}}{\|u\|_W} \bigg]^2  \! \bigg(\!1+ \bigg[\!\sup_{u \in V\backslash\{0\}}\tfrac{\|u\|_{\mathcal{V}}}{\|u\|_V}\bigg]^{2 \varepsilon}\big(1+2^{\max\{2\varepsilon-1, 0\}} \big)\|v\|_V^{2\varepsilon} + 2^{\max\{2\varepsilon-1, 0\}} \|v-w\|_{\mathcal{V}}^{2\varepsilon}\bigg) \|v-w\|_{\mathcal{V}}^2 \\
& \leq 3 \, \epsilon^2 \bigg[\!\sup_{u \in W\backslash\{0\}}\tfrac{\|u\|_{\mathcal{W}}}{\|u\|_W}\bigg]^2 \! \bigg(\!1+ \bigg[\!\sup_{u \in V\backslash\{0\}}\tfrac{\|u\|_{\mathcal{V}}}{\|u\|_V}\bigg]^{2 \varepsilon} \bigg)\big(1+2^{\max\{2\varepsilon-1, 0\}}\big) \big(\max\{1, \|v\|_V^{2\varepsilon}\}+ \|v-w\|_{\mathcal{V}}^{2\varepsilon}\big) \|v-w\|_{\mathcal{V}}^2\\
& \leq \theta \max\{1, \|v\|_V^{\vartheta}\} \, \|v-w\|_{\mathcal{V}}^2 + \theta \, \|v-w\|_{\mathcal{V}}^{2+\vartheta} .
\end{split}
\end{align}
Moreover, the fact that $\forall \, a, b \in \R \colon (a+b)^2 \leq 2(a^2+b^2)$ and the triangle inequality imply that for all $v \in U$ it holds that
\begin{align}
\begin{split}
&\|F(v)\|_W^2  \leq \big(\|F(v)-F(0)\|_W +\|F(0)\|_W\big)^2 \leq 2\big(\|F(v)-F(0)\|_W^2 + \|F(0)\|_W^2 \big) \\
& \leq 2\big(\epsilon^2(1+ \|v\|_{\mathcal{V}}^{\varepsilon})^2\|v\|_{\mathcal{V}}^2 + \|F(0)\|_W^2  \big) \leq 4\epsilon^2 \big(1+ \|v\|_{\mathcal{V}}^{2\varepsilon}\big)\|v\|_{\mathcal{V}}^2 + 2 \|F(0)\|_W^2 \\
& \leq \big(8\epsilon^2 + 2\|F(0)\|_W^2\big)\max\{ 1, \|v\|_{\mathcal{V}}^{2 + 2\varepsilon} \} \\
& \leq \big(8\epsilon^2 + 2\|F(0)\|_W^2\big) \max \!\Big\{ 1, \Big[\sup\nolimits_{u \in V\backslash\{0\}}\tfrac{\|u\|_{\mathcal{V}}}{\|u\|_V}\Big]^{2+2 \varepsilon} \Big\} \max\{ 1, \|v\|_V^{2 + 2\varepsilon} \}  \leq  \cthree \max\{ 1, \|v\|_V^{2 + \cfour} \}.
\end{split}
\end{align}
This and \eqref{eq:lemma:cond} complete the proof of Lemma \ref{lemma:cond}.
\end{proof}

\subsection{A priori bounds}\label{subsec:bounds}

\begin{prop}[A priori bounds]
\label{prop:priori_bound}
Let $ ( H, \left< \cdot , \cdot \right>_H, \left\| \cdot \right\|_H ) $ be a separable $\mathbb{R}$-Hilbert space, let $ ( \Omega, \mathcal{F}, \P ) $ be a probability space, let $ A \in L(H) $ be a diagonal linear operator with $\sup\!\big( \sigma_p(A) \big) < 0$ (see, e.g., Definition 3.4.5 in  \cite{Jentzen2015SPDElecturenotes}), let $ ( H_r, \left< \cdot , \cdot \right>_{ H_r }, \left\| \cdot \right\|_{ H_r } ) $, $ r \in \R $, be a family of interpolation spaces associated to $ - A $ (see, e.g., Definition~3.5.25 in \cite{Jentzen2015SPDElecturenotes}), and let $ Y, O \colon [0,T] \times \Omega \to H$ be stochastic processes, and let $F \in \mathcal{C}(H, H)$, $\phi, \Phi \in \mathcal{M}( \mathcal{B}(H) , \mathcal{B}([0,\infty)))$, $ \ctwo \in [0,1)$, $\alpha \in [0,\nicefrac{1}{2}]$, $\rho \in [0,1-\alpha)$, $ \varrho \in [\rho, \rho +1]$,  $ \cthree , \cfour \in [0,\infty)$, $T \in (0, \infty)$, $\chi \in \R$, 
$ h \in (0, 1] $ satisfy for all $v, w \in H_1 $, $t \in [0,T]$  that $\left< v, F( v + w ) \right>_H \leq \frac{1}{2} \phi(w) \| v \|^2_H+ \ctwo \| v \|^2_{ H_{ 1 / 2 } }+ \frac{1}{2}\Phi( w )$, $ \left\| F(v) - F(w) \right\|_{H_{ - 1 / 2 }}^2\leq \cthree \max\{ 1, \|v\|_{ H_{ \varrho } }^{\cfour} \} \|v-w\|_{ H_{ \rho } }^2 + \cthree \, \|v-w\|^{2 + \cfour}_{ H_{ \rho } }$, $\|F(v)\|_{H_{ - \alpha }}^2 \leq \cthree \max\{ 1, \|v\|_{ H_{ \varrho } }^{2 + \cfour} \} $,   and
\begin{equation}
	\label{eq:scheme_continuous}
	Y_t = e^{ t A} \big( Y_0 - O_0 \big)  + \int_0^t e^{ ( t - s ) A } \, \mathbbm{1}_{ \{ \| Y_{ \lf s \rf_{h} } \|_{ H_{ \varrho } } + \| O_{ \lf s \rf_{h} } \|_{ H_{ \varrho } } \leq h^{ - \chi } \}} \, F \big(  Y_{ \lf s \rf_{ h } } \big) \, ds + O_t.
\end{equation}
Then it holds for all $ t \in [0, T]$  that
\begin{align}
\begin{split}
& \| Y_t -O_t \|^2_H + (1-\ctwo) \int_0^t  e^{\int_s^t \phi( O_{\fl{u} } ) \, du } \, \| Y_s -O_s \|^2_{H_{1/2}} \, ds \leq e^{\int_0^t \phi(O_{\fl{s}}) \, ds } \, \| Y_0 -O_0 \|^2_H \\
& +   \int_0^te^{ \int_s^t \phi( O_{\fl{u} } ) \, du}  \left[   \Phi( O_{ \lf s \rf_{h } }) + \tfrac{  \cthree \,    (1+ \sqrt{\cthree})^{2 + \cfour} \,   h^{ 2\min \{(1 + \cfour /2)( 1-\alpha - \rho - (1+ \cfour/2)\chi) , \varrho - \rho - (1+\cfour/2) \chi ,1-\alpha - \rho -( 1 + \cfour ) \chi \}}  }{(1/2-\ctwo/2)\, ( 1-\alpha - \rho)^{2 + \cfour}}\right] ds.
\end{split}
\end{align}
\end{prop}
\begin{proof}[Proof of Proposition \ref{prop:priori_bound}.]
	
Throughout this proof let $\Omega_t \subseteq \Omega $, $ t \in [0,T] $, be the sets with the property that for all $ t \in [0,T]$ it holds that $ \Omega_t = \big\{ \| Y_{ \lf t \rf_{h} } \|_{ H_{ \varrho } } + \| O_{ \lf t \rf_{h} } \|_{ H_{ \varrho } }\leq h^{ - \chi } \big\}$ and let $\bar{Y} \in \mathbb{M}( [0,T], H )$ be the function with the property that for all $ t \in [0,T] $ it holds that $\bar{Y}_t = Y_t - O_t$. Next note that for all $t \in [0,T]$ it holds that 
\begin{align}
  \bar{Y}_t = e^{ t A } \bar{Y}_0 
  + \int_0^t e^{  ( t - s ) A} \mathbbm{1}_{\Omega_s} F(Y_{ \lf s \rf_{h } }) \, ds.
\end{align}
Lemma \ref{lem:YX_in_DA} hence proves that $\bar{Y}$ has continuous sample paths and that for all $ t \in [0,T] $ it holds that $\bar{Y}_t = \bar{Y}_0 + \int_0^t A \bar{Y}_s + \mathbbm{1}_{\Omega_s} F\big(Y_{ \lf s \rf_{h} }\big)\, ds$. The fundamental theorem of calculus therefore ensures for all $ t \in [0,T] $ that
\begin{equation}
\label{eq:fund_theorem}
\begin{split}
 & e^{- \int_0^t  \phi( O_{\fl{s} } ) \, ds} \, \| \bar{Y}_t \|^2_H \\
 & =\|\bar{Y}_0 \|^2_H +  \int_0^t 2 \, e^{- \int_0^s \phi( O_{\fl{u} } ) \, du} \,  \big< \bar{Y}_s, A \bar{Y}_s + \mathbbm{1}_{\Omega_s}
 F\big( Y_{ \lf s \rf_{h} }\big) \big>_H \, ds - \int_0^t  \phi( O_{\fl{s} } ) \, e^{- \int_0^s \phi( O_{\fl{u} } ) \, du} \, \| \bar{Y}_s \|^2_H\, ds \\
 & =\| \bar{Y}_0 \|^2_H + 2 \int_0^t e^{- \int_0^s \phi( O_{\fl{u} } ) \, du} \,  \big< \bar{Y}_s, A \bar{Y}_s + \mathbbm{1}_{ \Omega_s} F\big( \bar{Y}_s +
O_{ \lf s \rf_{h} } \big) \big>_H \, ds \\
& + 2 \int_0^t e^{- \int_0^s \phi( O_{\fl{u} } ) \, du} \, \mathbbm{1}_{ \Omega_s} \big< \bar{Y}_s, F\big(Y_{ \lf s \rf_{h} }\big)- F\big( \bar{Y}_s +
O_{\lf s \rf_{h} } \big) \big>_H \, ds - \int_0^t \phi( O_{\fl{s} } ) \, e^{- \int_0^s \phi( O_{\fl{u} } ) \, du}  \, \| \bar{Y}_s \|^2_H \, ds.
\end{split}
\end{equation}
Next observe that for all $ s \in [0, T]$ it holds that
\begin{align}
\begin{split}
 \big< \bar{Y}_s , A \bar{Y}_s \big>_H + \mathbbm{1}_{\Omega_s}  \tfrac{(1-\varphi)}{2} \|\bar{Y}_s \|^2_{H_{1/2}} & \leq \big< \bar{Y}_s , A \bar{Y}_s \big>_H +  \tfrac{(1-\varphi)}{2} \|\bar{Y}_s \|^2_{H_{1/2}}  = \big< \bar{Y}_s , A \bar{Y}_s \big>_H  - \tfrac{(1-\varphi)}{2} \big< \bar{Y}_s , A \bar{Y}_s \big>_H \\
& = \big< \bar{Y}_s , \big[1-\tfrac{(1-\varphi)}{2}  \big] A \bar{Y}_s \big>_H = \big< \bar{Y}_s , \tfrac{(1+\varphi)}{2} A \bar{Y}_s \big>_H  .
\end{split}
\end{align}
This, \eqref{eq:fund_theorem}, and the 
Cauchy-Schwartz inequality show  for all  $ t \in [0,T] $ that
\begin{equation}
\begin{split}
&  e^{- \int_0^t \phi( O_{\fl{s} } ) \, ds} \, \| \bar{Y}_t \|^2_H  = \| \bar{Y}_0 \|^2_H + 2 \int_0^t e^{- \int_0^s \phi( O_{\fl{u} } ) \, du}  \,  \big< \bar{Y}_s, A \bar{Y}_s + \mathbbm{1}_{ \Omega_s} F(\bar{Y}_s + O_{ \lf s \rf_{h } }) \big>_H \, ds\\
 & + 2 \int_0^t e^{- \int_0^s \phi( O_{\fl{u} } ) \, du}  \, \mathbbm{1}_{\Omega_s}  \big<( - A )^{ 1 / 2 } \bar{Y}_s ,  ( - A )^{ - 1 / 2 }\big[F( Y_{ \lf s \rf_{h } } )- F( \bar{Y}_s + O_{ \lf s \rf_{h } })\big] \big>_H \, ds \\
 &-   \int_0^t \phi( O_{\fl{s} } ) \, e^{- \int_0^s \phi( O_{\fl{u} } ) \, du}  \, \| \bar{Y}_s \|^2_H \, ds \\
& \leq \| \bar{Y}_0 \|^2_H + 2\int_0^t  e^{- \int_0^s \phi( O_{\fl{u} } ) \, du} \, \big<\bar{Y}_s , \tfrac{(1+\ctwo)}{2}  A \bar{Y}_s + \mathbbm{1}_{\Omega_s} F\big( \bar{Y}_s + O_{ \lf s \rf_{h} } \big) \big>_H \, ds \\ 
& + 2 \int_0^t  e^{- \int_0^s \phi( O_{\fl{u} } ) \, du} \, \mathbbm{1}_{\Omega_s} \Bigl[  \|\bar{Y}_s 
\|_{ H_{ 1 / 2 } } \|F( Y_{ \lf s \rf_{h } })-F( \bar{Y}_s +O_{ \lf s \rf_{h } }) \|_{ H_{ - 1 / 2 } } - \tfrac{(1- \ctwo)}{2} \| \bar{Y}_s \|_{ H_{ 1 / 2 } }^2 \Bigr]\, ds \\
&-\int_0^t \phi( O_{\fl{s} } ) \, e^{- \int_0^s \phi( O_{\fl{u} } ) \, du}  \, \| \bar{Y}_s \|^2_H \, ds.
\end{split}
\end{equation}
The fact that $\forall \, a , b \in \R , \varepsilon \in (0,\infty) \colon a b \leq \varepsilon  a^2 + \frac{ b^2 }{ 4 \varepsilon }$ therefore proves for all $ t \in [0, T]$ that
\begin{equation}
\label{eq:estimate0}
\begin{split}
 e^{- \int_0^t \phi( O_{\fl{s} } ) \, ds} \, \| \bar{Y}_t \|^2_H   & \leq \| \bar{Y}_0 \|^2_H +  \int_0^t e^{- \int_0^s \phi( O_{\fl{u} } ) \, du} \,  \big< \bar{Y}_s , (1+\ctwo)  A \bar{Y}_s +  2 \mathbbm{1}_{\Omega_s}  F(\bar{Y}_s +O_{ \lf s \rf_{h } })\big>_H \, ds\\
 &  +\tfrac{ 1 }{( 1- \ctwo ) }\int_0^t e^{- \int_0^s \phi( O_{\fl{u} } ) \, du}  \,  \mathbbm{1}_{\Omega_s} \|F( Y_{ \lf s \rf_{h } })-F( \bar{Y}_s + O_{ \lf s \rf_{h } })\|_{ H_{ - 1 / 2 } }^2 \, ds \\
 & -  \int_0^t \phi( O_{\fl{s} } ) \, e^{- \int_0^s \phi( O_{\fl{u} } ) \, du}  \, \| \bar{Y}_s \|^2_H\, ds.
\end{split}
\end{equation}
In the next step we use the fact that  $ \forall \, v, w \in H_1=H \colon \| F(v) - F(w) \|_{H_{ - 1 / 2 }}^2\leq \cthree \max\{ 1, \|v\|_{ H_{ \varrho } }^{\cfour} \} \|v-w\|_{ H_{ \rho } }^2 + \cthree \, \|v-w\|^{2 + \cfour}_{ H_{ \rho } }$  to obtain for all $v, w \in H_1=H$,  $ s \in [0, T]$  that
\begin{equation}
\label{eq:estimate1}
\begin{split}
& \mathbbm{1}_{\Omega_s} \|F(Y_{ \lf s \rf_{h } })- F(\bar{Y}_s + O_{ \lf s \rf_{h } }) \|_{ H_{ - 1 / 2 } }^2 \\
 & \leq \mathbbm{1}_{\Omega_s} \cthree    \max\{ 1, \|Y_{ \lf s \rf_{h } }\|_{ H_{ \varrho } }^{\cfour} \} \, \| \bar{Y}_{ \lf s \rf_{h } } - \bar{Y}_s \|_{ H_{ \rho } }^2 +  \mathbbm{1}_{\Omega_s} \cthree   \, \| \bar{Y}_{ \lf s \rf_{h } } - \bar{Y}_s \|_{ H_{ \rho } }^{2 + \cfour}  \\
& \leq \mathbbm{1}_{\Omega_s} \cthree   \, \| \bar{Y}_{ \lf s \rf_{h } } - \bar{Y}_s \|_{ H_{ \rho } }^2 \left( h^{-\cfour \chi} +   \| \bar{Y}_{ \lf s \rf_{h } } - \bar{Y}_s \|_{ H_{ \rho } }^{\cfour} \right).
\end{split}
\end{equation}
Moreover, observe that, e.g., Theorem~4.7.6 in  \cite{Jentzen2015SPDElecturenotes} and, e.g., Lemma 4.7.7 in  \cite{Jentzen2015SPDElecturenotes}  imply that for all $ s \in [0, T]$ it holds that 
\begin{equation}
\label{eq:estimate2}
\begin{split}
&\mathbbm{1}_{\Omega_s}  \|\bar{Y}_{ \lf s \rf_{h } } - \bar{Y}_s \|_{ H_{ \rho } } = \mathbbm{1}_{\Omega_s} \Big\| \left( e^{ ( s - \lf s \rf_{h } ) A }- \operatorname{Id}_H \right) \bar{Y}_{ \lf s \rf_{h } } + \smallint_{\lf s \rf_{h }}^s e^{  ( s - u ) A } \, F( Y_{ \lf s \rf_{h } } ) \, du \Big\|_{ H_{ \rho }} \\ 
& \leq \mathbbm{1}_{\Omega_s}\Big[\left| s - \lf s \rf_{h } \right|^{ \varrho - \rho }\| \bar{Y}_{ \lf s \rf_{h } } \|_{ H_{ \varrho }} + \smallint_{ \lf s \rf_{h }}^s \left( s - u \right)^{ - (\alpha+ \rho)} \| F( Y_{ \lf s \rf_{h } } ) \|_{ H_{ - \alpha }} \, du \Big] \\ 
& \leq \mathbbm{1}_{ \Omega_s}\Big[ \left|s - \lf s \rf_{h } \right|^{ \varrho -\rho }\| \bar{Y}_{ \lf s \rf_{h } } \|_{H_{ \varrho }} + \tfrac{\left| s - \lf s \rf_{h } \right|^{ 1-\alpha - \rho}}{1-\alpha - \rho} \sqrt{\cthree} \max\big\{ 1, \|Y_{ \lf s \rf_{h } }\|_{ H_{ \varrho } }^{1 + \nicefrac{\cfour}{2}} \big\}  \Big]\\
& \leq \tfrac{\mathbbm{1}_{\Omega_s}}{ 1-\alpha - \rho} \left[h^{ \varrho- \rho } \,\| \bar{Y}_{ \lf s \rf_{h } } \|_{ H_{ \varrho }}+ \sqrt{\cthree} \, h^{ 1-\alpha - \rho } \max\{ 1, h^{-\chi - \cfour \chi/2} \}\right] \\
& \leq  \tfrac{\mathbbm{1}_{\Omega_s}}{ 1-\alpha - \rho}  \left[ h^{ \varrho -\rho - \chi }+ \sqrt{\cthree} \, h^{ 1-\alpha - \rho -\chi - \cfour \chi/2} \right] \leq \frac{ (1+ \sqrt{\cthree}) \,  h^{\min \{ \varrho - \rho - \chi ,  1-\alpha - \rho -( 1 + \cfour / 2 ) \chi \}}  }{( 1-\alpha - \rho)}.
\end{split}
\end{equation}
Putting \eqref{eq:estimate2} into \eqref{eq:estimate1} shows for all $ s \in [0, T]$ that
\begin{equation}
\label{eq:estimate3}
\begin{split}
& \mathbbm{1}_{\Omega_s} \| F( Y_{ \lf s \rf_{h } })- F( \bar{Y}_s + O_{ \lf s \rf_{h } }) \|_{ H_{ - 1 / 2 } }^2 \\
& \leq \frac{ \cthree \,    (1+ \sqrt{\cthree})^{2} \,   h^{2\min \{ \varrho - \rho - \chi ,  1-\alpha - \rho -( 1 + \cfour / 2 ) \chi \}}  }{( 1-\alpha - \rho)^2} \left(  h^{-\cfour \chi}+ \frac{ (1+ \sqrt{\cthree})^{\cfour} \, h^{\cfour \min \{ \varrho - \rho - \chi ,  1-\alpha - \rho -( 1 + \cfour / 2 ) \chi \}}  }{( 1-\alpha - \rho)^{\cfour}}\right) \\
& \leq \frac{ 2 \,\cthree \,    (1+ \sqrt{\cthree})^{2 + \cfour} \,   h^{2\min \{ \varrho - \rho - \chi ,  1-\alpha - \rho -( 1 + \cfour / 2 ) \chi \}} \, h^{\cfour \min \{ \varrho - \rho - \chi ,  1-\alpha - \rho -( 1 + \cfour / 2 ) \chi, -\chi \} } }{( 1-\alpha - \rho)^{2+\cfour}}\\
& = \frac{ 2 \, \cthree \,    (1+ \sqrt{\cthree})^{2 + \cfour} \,  h^{2\min \{ \varrho - \rho - \chi ,  1-\alpha - \rho -( 1 + \cfour / 2 ) \chi \}} \, h^{\cfour \min \{   1-\alpha - \rho -( 1 + \cfour / 2 ) \chi, -\chi \} } }{( 1-\alpha - \rho)^{2+\cfour}}\\
& = \frac{ 2 \, \cthree \,    (1+ \sqrt{\cthree})^{2 + \cfour} \,   h^{ 2\min \{(1 + \cfour /2)( 1-\alpha - \rho - (1+ \cfour/2)\chi) , \varrho - \rho - (1+\cfour/2) \chi ,1-\alpha - \rho -( 1 + \cfour ) \chi \}}  }{( 1-\alpha - \rho)^{2 + \cfour}}.
\end{split}
\end{equation}
In the next step we put \eqref{eq:estimate3} into \eqref{eq:estimate0} to obtain that for all $t \in [0, T]$ it holds that
\begin{equation}
\begin{split}
& e^{-\int_0^t  \phi( O_{\fl{s} } ) \, ds } \, \| \bar{Y}_t \|^2_H  \leq \|\bar{Y}_0 \|^2_H  +  \int_0^te^{- \int_0^s \phi( O_{\fl{u} } ) \, du}  \left[ -(1+\varphi) \|\bar{Y}_s\|_{H_{1/2}}^2 + 2   \mathbbm{1}_{\Omega_s} \big<\bar{Y}_s ,   F\big( \bar{Y}_s + O_{ \lf s \rf_{h} } \big) \big>_H  \right] ds \\
&+ \tfrac{  \cthree \,    (1+ \sqrt{\cthree})^{2 + \cfour} \,   h^{ 2\min \{(1 + \cfour /2)( 1-\alpha - \rho - (1+ \cfour/2)\chi) , \varrho - \rho - (1+\cfour/2) \chi ,1-\alpha - \rho -( 1 + \cfour ) \chi \}}   }{(1/2-\ctwo/2)\, ( 1-\alpha - \rho)^{2 + \cfour}} \int_0^te^{- \int_0^s \phi( O_{\fl{u} } ) \, du}  \, ds\\ 
&-    \int_0^t \phi( O_{\fl{s} } ) \, e^{- \int_0^s \phi( O_{\fl{u} } ) \, du}  \, \| \bar{Y}_s \|^2_H\, ds.
\end{split}
\end{equation}
The assumption that $ \forall \, v, w \in H_1 \colon \left< v, F( v + w ) \right>_H \leq \frac{1}{2}\phi(w)  \| v \|^2_H+ \ctwo \| v \|^2_{ H_{ 1 / 2 } }+ \frac{1}{2}\Phi( w )$ hence proves that for all $t \in [0, T]$ it holds that
\begin{equation}
\begin{split}
 & e^{-\int_0^t  \phi( O_{\fl{s} } ) \, ds } \, \| \bar{Y}_t \|^2_H  \leq \|\bar{Y}_0 \|^2_H \\
 & +  \int_0^te^{- \int_0^s \phi( O_{\fl{u} } ) \, du}  \left[ - (1+\ctwo) \, \| \bar{Y}_s \|^2_{H_{1/2}} + \phi( O_{\fl{s} } ) \, \| \bar{Y}_s \|^2_H +2 \ctwo \, \| \bar{Y}_s \|^2_{H_{1/2}} +
 \Phi(O_{ \lf s \rf_{h } }) \right] ds \\
 &+ \tfrac{  \cthree \,    (1+ \sqrt{\cthree})^{2 + \cfour} \,   h^{ 2\min \{(1 + \cfour /2)( 1-\alpha - \rho - (1+ \cfour/2)\chi) , \varrho - \rho - (1+\cfour/2) \chi ,1-\alpha - \rho -( 1 + \cfour ) \chi \}}   }{(1/2-\ctwo/2)\, ( 1-\alpha - \rho)^{2 + \cfour}} \int_0^te^{- \int_0^s \phi( O_{\fl{u} } ) \, du}  \, ds\\ 
&-    \int_0^t \phi( O_{\fl{s} } ) \, e^{- \int_0^s \phi( O_{\fl{u} } ) \, du}  \, \| \bar{Y}_s \|^2_H\, ds.
\end{split}
\end{equation}
Therefore, we obtain for all $ t \in [0, T]$ that
\begin{align}
\begin{split}
 & e^{-\int_0^t  \phi( O_{\fl{s} } ) \, ds } \, \| \bar{Y}_t \|^2_H  \leq \|\bar{Y}_0 \|^2_H +  \int_0^te^{- \int_0^s \phi( O_{\fl{u} } ) \, du}  \left[ \{2\varphi- (1+\varphi)\}\| \bar{Y}_s \|^2_{H_{1/2}} +
 \Phi(O_{ \lf s \rf_{h } }) \right] ds\\
 & + \tfrac{  \cthree \,    (1+ \sqrt{\cthree})^{2 + \cfour} \,   h^{ 2\min \{(1 + \cfour /2)( 1-\alpha - \rho - (1+ \cfour/2)\chi) , \varrho - \rho - (1+\cfour/2) \chi ,1-\alpha - \rho -( 1 + \cfour ) \chi \}}   }{(1/2-\ctwo/2)\, ( 1-\alpha - \rho)^{2 + \cfour}} \int_0^te^{- \int_0^s \phi( O_{\fl{u} } ) \, du}  \, ds.
 \end{split}
\end{align}
This assures for all $ t \in [0, T]$ that
\begin{equation}
\label{eq:estimate00}
\begin{split}
& \| \bar{Y}_t \|^2_H + (1-\ctwo) \int_0^t  e^{\int_s^t \phi( O_{\fl{u} } ) \, du } \, \| \bar{Y}_s \|^2_{H_{1/2}} \, ds \leq e^{\int_0^t \phi( O_{\fl{s} } ) \, ds } \, \| \bar{Y}_0 \|^2_H \\
& +    \int_0^te^{ \int_s^t \phi( O_{\fl{u} } ) \, du}  \left[   \Phi( O_{ \lf s \rf_{h } }) + \tfrac{  \cthree \,    (1+ \sqrt{\cthree})^{2 + \cfour} \,   h^{ 2\min \{(1 + \cfour /2)( 1-\alpha - \rho - (1+ \cfour/2)\chi) , \varrho - \rho - (1+\cfour/2) \chi ,1-\alpha - \rho -( 1 + \cfour ) \chi \}} }{(1/2-\ctwo/2)\, ( 1-\alpha - \rho)^{2 + \cfour}} \right] ds.
\end{split}
\end{equation}
The proof of Proposition~\ref{prop:priori_bound} is thus completed.
\end{proof}


\begin{cor}[A priori moment bounds]
\label{cor:bounds}
Let $ ( H, \left< \cdot , \cdot \right>_H, \left\| \cdot \right\|_H ) $ 
be a separable $\mathbb{R}$-Hilbert space, 
let $ ( \Omega, \mathcal{F}, \P ) $ be a probability space, 
let $ A \in L(H) $ be a diagonal linear operator with $\sup\!\big( \sigma_p(A) \big) < 0$ 
(see, e.g., Definition 3.4.5 in  \cite{Jentzen2015SPDElecturenotes}), 
let $ ( H_r, \left< \cdot , \cdot \right>_{ H_r }, \left\| \cdot \right\|_{ H_r } ) $, 
$ r \in \R $, be a family of interpolation spaces associated to $ - A $ 
(see, e.g., Definition~3.5.25 in \cite{Jentzen2015SPDElecturenotes}), 
let $O \colon [0,T] \times \Omega \to H $ be a stochastic process, 
and let 
$Y \in \mathbb{M}([0, T] \times \Omega, H) $,
$F \in \mathcal{C}(H, H)$, 
$\phi, \Phi \in \mathcal{M}( \mathcal{B}(H) , \mathcal{B}([0,\infty)))$, 
$ \ctwo \in [0,1)$, $\alpha \in [0,\nicefrac{1}{2}]$, $\rho \in [0,1-\alpha)$, 
$ \varrho \in [\rho, \rho+1]$,  $ \cthree , \cfour \in [0,\infty)$, 
$\chi \in [0, (2-2\alpha-2\rho)/(2+\cfour) ] $, $T \in (0, \infty)$, 
$ h \in (0, 1]$, $p \in [2, \infty)$
satisfy for all $v, w \in H_1 $, $t \in [0,T]$  that $\left< v, F( v + w ) \right>_H \leq  \frac{1}{2}\phi(w) \| v \|^2_H+ \ctwo \| v \|^2_{ H_{ 1 / 2 } }+ \frac{1}{2}\Phi( w )$, $ \left\| F(v) - F(w) \right\|_{H_{ -1/2 }}^2\leq \cthree \max\{ 1, \|v\|_{ H_{ \varrho } }^{\cfour} \} \|v-w\|_{ H_{ \rho } }^2 + \cthree \, \|v-w\|^{2 + \cfour}_{ H_{ \rho } }$, $\|F(v)\|_{H_{ - \alpha}}^2 \leq \cthree \max\{ 1, \|v\|_{ H_{ \varrho } }^{2 + \cfour} \} $, and
\begin{equation}\label{meas:Y}
Y_t =  \int_0^t e^{ ( t - s ) A } \, \mathbbm{1}_{ \{ \| Y_{ \lf s \rf_{h} } \|_{ H_{ \varrho } } + \| O_{ \lf s \rf_{h} } \|_{ H_{ \varrho } } \leq h^{ - \chi } \}} \, F \big(  Y_{ \lf s \rf_{ h } } \big) \, ds + O_t.
\end{equation}
Then it holds that $Y - O \colon [0,T] \times \Omega \to H$ is a stochastic process with continuous sample paths and it holds that
\begin{align}
\begin{split}
& \left\| \sup\nolimits_{t \in [0,T]} \|Y_t - O_t \|_H \right\|_{\mathcal{L}^p(\P ; \R)} \\
&\leq  \sqrt{ \int_0^T \left\|e^{ \int_s^T \phi( O_{\fl{u} } ) \, du}  \left[   \Phi( O_{ \lf s \rf_{h } }) + \tfrac{  \cthree \,    (1+ \sqrt{\cthree})^{2 + \cfour} \,   h^{ \min \{ 2\varrho, 2- 2\alpha  -\cfour  \chi \}  - 2\rho - (2+\cfour) \chi}  }{(1/2-\ctwo/2)\, ( 1-\alpha - \rho)^{2 + \cfour}}\right]  \right\|_{\mathcal{L}^{p/2}(\P ; \R)} ds } \\
& \leq
\sqrt{
	\left[
	1
	+
\tfrac{  \cthree \,    (1+ \sqrt{\cthree})^{2 + \cfour} \,   h^{ \min \{ 2\varrho, 2-2\alpha  -\cfour  \chi \}  - 2\rho - (2+\cfour) \chi}  }{(1/2-\ctwo/2)\, ( 1-\alpha - \rho)^{2 + \cfour}}
	\right]
	\int_0^T
	\left|
	\E\!\left[
	e^{ \nicefrac{ p }{ 2 } \int_s^T \phi( O_{\fl{u} } ) \, du }
	\max\!\big\{
	1 , | \Phi( O_{ \lf s \rf_{h } } ) |^{ \nicefrac{ p }{ 2 } }
	\big\} \right]
	\right|^{ 2 / p }
	ds
}.
\end{split}
\end{align}
\end{cor}

\begin{proof}[Proof of Corollary \ref{cor:bounds}.]
Note that the assumption that $ O \colon [0,T] \times \Omega \to H $ 
is a stochastic process and Lemma~\ref{lemma:meas} yield that 
$ Y \colon [0,T] \times \Omega \to H $ is also a stochastic process.
Combining Proposition~\ref{prop:priori_bound} with the 
assumption that $\chi \in [0, (2 -2\alpha-2\rho)/(2+\cfour) ] $ hence yields that 
\begin{equation}\label{cor:mY}
\begin{split}
&\sup_{t \in [0, T]} \| Y_t - O_t \|^2_H \\ 
& \leq  \sup_{t \in [0, T]}  
\bigg(
  \smallint_0^te^{ \int_s^t \phi(O_{\fl{u}}) \, du}  \left[   \Phi( O_{ \lf s \rf_{h } }) + \tfrac{  \cthree \,    (1+ \sqrt{\cthree})^{2 + \cfour} \,   h^{ 2\min \{(1 + \cfour /2)( 1-\alpha - \rho - (1+ \cfour/2)\chi) , \varrho - \rho - (1+\cfour/2) \chi ,1-\alpha - \rho -( 1 + \cfour ) \chi \}} }{(1/2-\ctwo/2)\, ( 1-\alpha - \rho)^{2 + \cfour}} \right] ds 
\bigg)
\\
& =  \sup_{t \in [0, T]} 
\bigg( 
  \int_0^te^{ \int_s^t \phi(\mathcal{O}_{\fl{u}}) \, du}  \left[   \Phi( O_{ \lf s \rf_{h } }) + \tfrac{  \cthree \,    (1+ \sqrt{\cthree})^{2 + \cfour} \,   h^{ \min \{ 2\varrho, 2-2\alpha  -\cfour  \chi \}  - 2\rho - (2+\cfour) \chi} }{(1/2-\ctwo/2)\, ( 1-\alpha - \rho)^{2 + \cfour}} \right] ds 
\bigg)
\\
& = \int_0^T e^{ \int_s^T \phi(O_{\fl{u}}) \, du}  \left[   \Phi( O_{ \lf s \rf_{h } }) + \tfrac{  \cthree \,    (1+ \sqrt{\cthree})^{2 + \cfour} \,   h^{ \min \{ 2\varrho, 2-2\alpha  -\cfour  \chi \}  - 2\rho - (2+\cfour) \chi} }{(1/2-\ctwo/2)\, ( 1-\alpha - \rho)^{2 + \cfour}} \right] ds.
\end{split}
\end{equation}
Moreover, 
the fact that 
$ Y \colon [0,T] \times \Omega \to H $ is a stochastic process, 
the assumption that 
$ O \colon [0,T] \times \Omega \to H $ is a stochastic process, 
\eqref{meas:Y},  
and Lemma~\ref{lem:YX_in_DA} 
prove that $Y - O \colon [0,T] \times \Omega \to H$ 
is a stochastic process with continuous sample paths. 
Hence, we obtain that
$ 
  \big(
    \Omega \ni \omega \mapsto
    \sup\nolimits_{ t \in [0,T] } 
    \| Y_t( \omega ) - O_t( \omega ) \|_H 
    \in
    \R
  \big)
  \in \mathcal{M}(\mathcal{F}, \mathcal{B}(\R))$. This, \eqref{cor:mY}, Minkowski's integral inequality, and the assumption that $p \geq 2$  show that 
\begin{align}\label{2cor:mY}
\begin{split}
& \left\| \sup\nolimits_{t \in [0,T]} \|Y_t - O_t \|_H \right\|_{\mathcal{L}^p(\P ; \R)} \\
&\leq  \sqrt{ \int_0^T \left\|e^{ \int_s^T \phi(O_{\fl{u}}) \, du}  \left[   \Phi( O_{ \lf s \rf_{h } }) + \tfrac{  \cthree \,    (1+ \sqrt{\cthree})^{2 + \cfour} \,   h^{ \min \{ 2\varrho, 2- 2\alpha  -\cfour  \chi \}  - 2\rho - (2+\cfour) \chi}  }{(1/2-\ctwo/2)\, ( 1-\alpha - \rho)^{2 + \cfour}}\right]  \right\|_{\mathcal{L}^{p/2}(\P ; \R)} ds }.
\end{split}
\end{align}
The proof of Corollary~\ref{cor:bounds}
is thus completed.
\end{proof}

\section{Pathwise convergence}
\label{sec:pathwise_error}

\subsection{Setting}\label{setting:pathwise}

Let  $( V, \left\| \cdot \right\|_V)$ and 
$(W, \left\| \cdot \right\|_W)$ be $\R$-Banach spaces 
and let $T, \chi \in (0, \infty)$, $\Upsilon \in \R$, 
$\alpha \in [0, 1)$,  $  (P_n)_{n \in \mathbb{N}} \in \mathbb{M}(\N, L(V))$, 
$(h_n)_{n \in \N} \in \mathbb{M}(\N, (0, \infty) )$, $F \in \mathcal{C} (V, W)$, 
$ \Psi \in \mathbb{M}( [0, \infty], [0, \infty])$, 
$X, O \in \mathcal{C} ([0, T] , V )$, $ (\Y^{n})_{n \in \N}$, 
$(\mathcal{O}^n)_{n \in \N} \in \mathbb{M}( \N, \mathbb{M}( [0, T], V))$, 
$S \in \mathcal{M}\big( \mathcal{B}( (0, T] ) , \mathcal{B}(L(W, V))\big)$ satisfy for all $ r \in [0, \infty]$, $t \in [0,T]$, $n \in \mathbb{N} $ that $\limsup_{m \to \infty} h_m=0$, $\Upsilon  = \sup_{t \in (0, T]} ( t^{\alpha} \| S_t \|_{L(W,V)} )$, $\Psi(r) =  \sup \!\big( \{0 \} \cup  \big\{ \frac{\|F(v)-F(w)\|_W}{\|v-w\|_V}   \colon v, w \in V,  v\neq w, \max \{ \|v\|_V, \|w\|_V \} \leq r \big\} \big)$, $\Psi( [0,\infty) ) \subseteq \R$, $X_t=\int_0^t S_{t-s} \, F(X_s) \, ds + O_t$, and
\begin{align}
\begin{split}
 \Y_t^{n} = \int_0^t P_n \, S_{  t - s  } \, \mathbbm{1}_{[0, |h_n|^{-\chi}] } \big( \| \Y_{ \lf s \rf_{h_n} }^{n} \|_{V} + \| \mathcal{O}_{ \lf s \rf_{h_n} }^{n} \|_{V}\big) \, F \big(  \Y_{ \lf s \rf_{ h_n } }^{n} \big) \, ds +  \mathcal{O}_t^{n}.
 \end{split}
\end{align}

\subsection{Auxiliary results}

\begin{lemma}\label{lemma:pathwise_bound}
Assume the setting in Section \ref{setting:pathwise} and let $n \in \mathbb{N}$, $t \in [0, T]$. Then
\begin{align}
\begin{split}
& \|X_t - \Y_t^{n} \|_V \leq   \left\| O_t - \mathcal{O}_t^{n} \right\|_V + \frac{  \, \Upsilon \, t^{1-\alpha}}{(1-\alpha)}  \Psi\! \left( \sup\nolimits_{s \in [0, t]} \|X_s\|_V\right)  \left[\sup\nolimits_{s \in [0, t]} \| X_s - X_{\lf s \rf_{h_n}} \|_V \right] \\
& +  \frac{ \Upsilon \, |h_n|^{\chi} \, t^{1-\alpha}    }{ (1-\alpha)} \left( \sup\nolimits_{s \in [0, t]} \left\| F(X_s) \right\|_W \right)\left( \sup\nolimits_{s \in [0, t]} \left[ \left\| X_s \right\|_V + \left\|\mathcal{O}_s^{n} \right\|_V \right] \right) \\
& + \left( \sup\nolimits_{s \in [0, t]} \left\| F(X_s) \right\|_W \right) \int_0^t \left\|  ( \Id_V - P_n )  S_s   \right\|_{L(W, V)}  ds  \\
& +  \Upsilon \left[\left\|P_n \right\|_{L(V)}  \Psi\big( \sup\nolimits_{s \in [0, \max\{0, \lceil t \rceil_{h_n}-h_n \}]}\max \! \big\{\|X_s \|_V ,  \|\Y_s^n\|_V \big\} \big)  +   |h_n|^{\chi} \left( \sup\nolimits_{s \in [0, t]} \left\| F(X_s) \right\|_W \right) \right]\\
& \cdot \int_0^t (t-s)^{-\alpha} \,  \| X_{ \lf s \rf_{h_n} } -  \Y_{ \lf s \rf_{h_n} }^{n}  \|_V \, ds.
\end{split}
\end{align}
\end{lemma}
\begin{proof}[Proof of Lemma \ref{lemma:pathwise_bound}.] 
Observe that the triangle inequality proves that
\begin{align}\label{eq:triangle}
\begin{split}
&\|X_t - \Y_t^{n} \|_{V}  \leq \left\| O_t - \mathcal{O}_t^{n} \right\|_V + \int_0^t \big \| S_{t-s} \big[ F(X_s)- F\big(X_{ \lf s \rf_{h_n} }\big)\big] \big\|_V \, ds \\
& +\int_0^t \big\| S_{t-s} \big[  F\big(X_{ \lf s \rf_{h_n} }\big) - \mathbbm{1}_{[0, |h_n|^{-\chi}] } \big( \| \Y_{ \lf s \rf_{h_n} }^{n} \|_{V} + \| \mathcal{O}_{ \lf s \rf_{h_n} }^{n} \|_{V} \big) F\big(X_{ \lf s \rf_{h_n} }\big)\big]  \big\|_V  \,ds \\
& + \int_0^t \big\|\! \left[ S_{t-s} - P_n \, S_{t-s} \right]   \mathbbm{1}_{[0, |h_n|^{-\chi}] } \big( \| \Y_{ \lf s \rf_{h_n} }^{n} \|_{V} + \| \mathcal{O}_{ \lf s \rf_{h_n} }^{n} \|_{V} \big)  F\big(X_{ \lf s \rf_{h_n} }\big)  \big\|_V \, ds \\
& + \int_0^t \big\| P_n \, S_{t-s} \,\mathbbm{1}_{[0, |h_n|^{-\chi}] } \big( \| \Y_{ \lf s \rf_{h_n} }^{n} \|_{V} + \| \mathcal{O}_{ \lf s \rf_{h_n} }^{n} \|_{V} \big) \big[ F\big(X_{ \lf s \rf_{h_n} }\big) -  F\big(\Y_{ \lf s \rf_{h_n}}^{n}\big) \big]  \big\|_V \, ds .
\end{split}
\end{align}
Next note that
\begin{align}\label{eq:path_1}
\begin{split}
\int_0^t \left \| S_{t-s} \big[ F(X_s)- F\big(X_{ \lf s \rf_{h_n} }\big)\big] \right\|_V ds  & \leq \int_0^t \left \|  S_{t-s} \right\|_{L(W,V)}  \big\|  F(X_s) -  F\big(X_{ \lf s \rf_{h_n} }\big)  \big\|_{W} \, ds \\
& \leq \Upsilon \, \Psi \! \left(\sup\nolimits_{s \in [0, t]} \|X_s\|_V\right)  \int_0^t (t-s)^{-\alpha} \, \| X_s -X_{\lf s \rf_{h_n}} \|_V \, ds \\
& \leq \frac{ \Upsilon  \, t^{1-\alpha}}{(1-\alpha)}  \Psi \! \left(\sup\nolimits_{s \in [0, t]} \|X_s\|_V\right)  \left[\sup\nolimits_{s \in [0, t]} \| X_s - X_{\lf s \rf_{h_n}} \|_V \right].
\end{split}
\end{align}
Moreover, observe that
\begin{align}\label{eq:path_2}
\begin{split}
& \int_0^t \big\| S_{t-s} \big[  F(X_{ \lf s \rf_{h_n} }) - \mathbbm{1}_{[0, |h_n|^{-\chi}] } \big( \| \Y_{ \lf s \rf_{h_n} }^{n} \|_{V} + \| \mathcal{O}_{ \lf s \rf_{h_n} }^{n} \|_{V}\big) \, F\big(X_{ \lf s \rf_{h_n} }\big)\big]  \big\|_V \, ds \\
& = \int_0^t \big\| S_{t-s} \,\mathbbm{1}_{( |h_n|^{-\chi}, \infty) } \big( \| \Y_{ \lf s \rf_{h_n} }^{n} \|_{V} + \| \mathcal{O}_{ \lf s \rf_{h_n} }^{n} \|_{V} \big) \, F\big(X_{ \lf s \rf_{h_n} }\big)  \big\|_V \, ds \\
& \leq \int_0^t  \left\|S_{t-s} \right\|_{L(W, V)} \mathbbm{1}_{( |h_n|^{-\chi}, \infty) } \big( \| \Y_{ \lf s \rf_{h_n} }^{n} \|_{V} + \| \mathcal{O}_{ \lf s \rf_{h_n} }^{n} \|_{V} \big) \, \big\| F\big(X_{\lf s \rf_{h_n}}\big) \big\|_W \, ds \\
& \leq \Upsilon \left( \sup\nolimits_{s \in [0, t]} \left\| F(X_s) \right\|_W \right) \int_0^t (t-s)^{-\alpha} \, \big(  \| \Y_{ \lf s \rf_{h_n} }^{n} \|_V + \| \mathcal{O}_{ \lf s \rf_{h_n} }^{n} \|_{V} \big) \, | h_n|^{ \chi }\, ds \\
& \leq  | h_n|^{ \chi } \, \Upsilon  \left( \sup\nolimits_{s \in [0, t]} \left\| F(X_s) \right\|_W \right) \\
& \quad \cdot \!\int_0^t (t-s)^{-\alpha} \,  \big( \| X_{ \lf s \rf_{h_n} } -  \Y_{ \lf s \rf_{h_n} }^{n}  \|_V  + \| X_{ \lf s \rf_{h_n} }\|_V  + \| \mathcal{O}_{ \lf s \rf_{h_n} }^n|_V \big) \, ds  \\
& \leq  \Upsilon \, |h_n|^{\chi} \left( \sup\nolimits_{s \in [0, t]} \left\| F(X_s) \right\|_W \right) \int_0^t (t-s)^{-\alpha} \, \| X_{ \lf s \rf_{h_n} } -  \Y_{ \lf s \rf_{h_n} }^{n}  \|_V \, ds \\
& \quad +  \frac{ \Upsilon \, |h_n|^{\chi} \, t^{1-\alpha}    }{ (1-\alpha)} \left( \sup\nolimits_{s \in [0, t]} \left\| F(X_s) \right\|_W \right)\left( \sup\nolimits_{s \in [0, t]} \left[ \left\| X_s \right\|_V +  \left\| \mathcal{O}_s^n \right\|_V \right]\right)
\end{split}
\end{align}
and 
\begin{align}\label{eq:path_3}
\begin{split}
& \int_0^t \big\| \!\left[ S_{t-s} - P_n \, S_{t-s} \right]   \mathbbm{1}_{[0, |h_n|^{-\chi}] } \big( \| \Y_{ \lf s \rf_{h_n} }^{n} \|_{V} + \| \mathcal{O}_{ \lf s \rf_{h_n} }^{n} \|_{V} \big) \, F\big(X_{ \lf s \rf_{h_n} }\big)  \big\|_V \, ds  \\
& \leq \int_0^t \left\|  ( \Id_V - P_n ) S_{t-s}  \right\|_{L(W, V)} \|  F(X_{ \lf s \rf_{h_n} }) \|_W \, ds \\
& \leq  \left( \sup\nolimits_{s \in [0, t]} \left\| F(X_s) \right\|_W \right) \int_0^t \left\|  ( \Id_V - P_n ) S_{t-s}   \right\|_{L(W, V)}   ds\\
& = \left( \sup\nolimits_{s \in [0, t]} \left\| F(X_s) \right\|_W \right) \int_0^t \left\|  ( \Id_V - P_n ) S_s   \right\|_{L(W, V)}   ds.
\end{split}
\end{align}
Furthermore, note that
\begin{align}\label{eq:path_4}
\begin{split}
& \int_0^t \big\| P_n \, S_{t-s} \,\mathbbm{1}_{[0, |h_n|^{-\chi}] } \big( \| \Y_{ \lf s \rf_{h_n} }^{n} \|_{V} + \| \mathcal{O}_{ \lf s \rf_{h_n} }^{n} \|_{V}\big) \big[ F\big(X_{ \lf s \rf_{h_n} }\big) -  F\big(\Y_{ \lf s \rf_{h_n}}^{n}\big) \big]  \big\|_V \, ds \\
& \leq \int_0^t  \left\|P_n \right\|_{L(V)} \left\| S_{t-s} \, \right\|_{L(W,V)} \big\| F\big(X_{ \lf s \rf_{h_n} }\big) -   F\big(\Y_{ \lf s \rf_{h_n}}^{n }\big)  \big\|_W \, ds \\
& \leq \Upsilon \left\|P_n \right\|_{L(V)} \int_0^t (t-s)^{-\alpha} \, \big\|X_{ \lf s \rf_{h_n} }- \Y_{ \lf s \rf_{h_n} }^{n }  \big\|_V  \Psi\big( \!\max \! \big\{\|X_{ \lf s \rf_{h_n} }\|_V , \|\Y_{ \lf s \rf_{h_n} }^n\|_V \big\} \big) \,  ds \\
& \leq  \Upsilon \left\|P_n \right\|_{L(V)} \Psi\big(  \sup\nolimits_{s \in [0, \max\{0, \lceil t \rceil_{h_n}-h_n \}] }\max \! \big\{\|X_s \|_V , \|\Y_s^n \|_V \big\} \big)  \int_0^t (t-s)^{-\alpha} \, \|X_{ \lf s \rf_{h_n} }- \Y_{ \lf s \rf_{h_n} }^{n }  \|_V \, ds.
\end{split}
\end{align}
Combining \eqref{eq:triangle}--\eqref{eq:path_4}  completes the proof of Lemma \ref{lemma:pathwise_bound}.
\end{proof}


\begin{cor}\label{cor:pathwise}
Assume the setting in Section \ref{setting:pathwise}, let $n \in \N$, and assume that $\sup\nolimits_{t \in [0,T]}  \|\mathcal{O}_t^{n} \|_V < \infty$. Then  it holds for all  $t \in [0,T]$ that $\sup\nolimits_{s \in [0, \max\{0, \lceil t \rceil_{h_n}-h_n \}]}\max \!\big\{\|X_{s}\|_V , \|\Y_s^n\|_V \big\}< \infty$ and
\begin{align}\label{eq:cor:pathwise}
\begin{split}
& \sup\nolimits_{s \in [0,t]} \|X_s - \Y_s^{n} \|_V  \\
& \leq \bigg[ \sup\nolimits_{s \in [0,T]}  \| O_s - \mathcal{O}_s^{n} \|_V + \frac{  \, \Upsilon \, T^{1-\alpha}}{(1-\alpha)}  \Psi \big( \sup\nolimits_{s \in [0, T]} \|X_s\|_V\big)  \left[\sup\nolimits_{s \in [0, T]}\| X_s - X_{\lf s \rf_{h_n}} \|_V \right] \\
& +  \frac{ \Upsilon \, |h_n|^{\chi} \, T^{1-\alpha}    }{ (1-\alpha)} \left( \sup\nolimits_{s \in [0, T]} \| F(X_s) \|_W \right)\left( \sup\nolimits_{s \in [0, T]} \left[ \left\| X_s \right\|_V + \left\|\mathcal{O}_s^{n} \right\|_V \right] \right) \\
& + \left( \sup\nolimits_{s \in [0, T]} \left\| F(X_s) \right\|_W \right) \smallint_0^T \left\|  ( \Id_V - P_n ) S_s   \right\|_{L(W, V)}  ds  \bigg] \cdot \mathrm{E}_{1-\alpha} \Big[ t \,
\big| \Gamma (1-\alpha) \, \Upsilon \big|^{\nicefrac{1}{(1-\alpha)}} \big[ \left\|P_n \right\|_{L(V)}  \\
& \cdot   \Psi\big( \sup\nolimits_{s \in [0, \max\{0, \lceil t \rceil_{h_n}-h_n \}]}\max \!\big\{\|X_{s}\|_V , \|\Y_s^n\|_V \big\} \big)  +   |h_n|^{\chi} \left( \sup\nolimits_{s \in [0, T]} \left\| F(X_s) \right\|_W \right)  \big]^{\nicefrac{1}{(1-\alpha)}} \Big]< \infty.
\end{split}
\end{align}
\end{cor}
\begin{proof}[Proof of Corollary \ref{cor:pathwise}.]
Note that Lemma \ref{lemma:pathwise_bound} implies for all $t \in [0,T]$  that
\begin{align}\label{eq:cor:pathwise1}
\begin{split}
& \sup\nolimits_{s \in [0,t]} \|X_s - \Y_s^{n} \|_V\\
&  \leq \sup\nolimits_{s \in [0,T]}  \| O_s - \mathcal{O}_s^{n} \|_V + \frac{  \Upsilon \, T^{1-\alpha}}{(1-\alpha)}  \Psi \big( \sup\nolimits_{s \in [0, T]} \|X_s\|_V\big) \! \left[\sup\nolimits_{s \in [0, T]} \| X_s - X_{\lf s \rf_{h_n}} \|_V \right] \\
& +  \frac{ \Upsilon \, |h_n|^{\chi} \, T^{1-\alpha}    }{ (1-\alpha)} \left( \sup\nolimits_{s \in [0, T]} \left\| F(X_s) \right\|_W \right)\left( \sup\nolimits_{s \in [0, T]} \left[ \left\| X_s \right\|_V + \left\|\mathcal{O}_s^{n} \right\|_V \right] \right) \\
&  + \left( \sup\nolimits_{s \in [0, T]} \left\| F(X_s) \right\|_W \right)  \smallint\nolimits_0^T \left\|  ( \Id_V - P_n ) S_s   \right\|_{L(W, V)}  ds   \\
& + \Upsilon  \left[  \left\|P_n \right\|_{L(V)}  \Psi\big( \sup\nolimits_{s \in [0, \max\{0, \lceil t \rceil_{h_n}-h_n \}]}\max \! \big\{\|X_{s}\|_V , \|\Y_s^n\|_V \big\} \big)  +  |h_n|^{\chi} \left( \sup\nolimits_{s \in [0, T]} \left\| F(X_s) \right\|_W \right) \right]\\
&  \cdot \smallint_0^t (t-s)^{-\alpha} \big(\sup\nolimits_{u \in [0, s]}  \| X_{u} -  \Y_u^{n} \|_V \big) \, ds.
\end{split}
\end{align}
Moreover, note that the assumption that $\sup\nolimits_{t \in [0,T]}  \| \mathcal{O}_t^{n} \|_V < \infty$ and the assumption that $O \in \mathcal{C}([0,T], V)$ imply that
\begin{align}
\begin{split}
\sup\nolimits_{t \in [0,T]}  \|O_t- \mathcal{O}_t^{n} \|_V \leq \sup\nolimits_{t \in [0,T]}  \| O_t  \|_V +\sup\nolimits_{t \in [0,T]}  \| \mathcal{O}_t^{n}\|_V  < \infty.
\end{split}
\end{align}
This yields that
\begin{align}
\begin{split}
 &\sup\nolimits_{t \in [0,T]} \|\Y_t^{n}\|_V \\
 &\leq \| P_n \|_{L(V)}  \big(\max\nolimits_{s \in \{ 0, h_n, 2h_n, \ldots\}\cap [0,T]} \| F (  \Y_{s}^{n} ) \|_W \big) \! \left[\sup\nolimits_{t \in [0,T]} \smallint\nolimits_0^t  \| S_{  t - s  } \|_{L(W,V)} \, ds \right]+ \sup_{t \in [0,T] } \| \mathcal{O}_t^{n} \|_W \\
& \leq \Upsilon\, \| P_n \|_{L(V)}  \big(\max\nolimits_{s \in \{ 0, h_n, 2h_n, \ldots\}\cap [0,T]} \| F (  \Y_{s}^{n} ) \|_W \big) \left[\sup_{t \in [0,T]}  \smallint_0^t (t-s)^{-\alpha} \, ds \right]+ \sup_{t \in [0,T] } \| \mathcal{O}_t^{n} \|_W \\
& = \frac{T^{1-\alpha} \,  \Upsilon\, \| P_n \|_{L(V)}  \big(\max\nolimits_{s \in \{ 0, h_n,  2h_n, \ldots\}\cap [0,T]} \| F (  \Y_{s}^{n} ) \|_W \big)}{(1-\alpha)}  + \sup_{t \in [0,T] } \| \mathcal{O}_t^{n} \|_W < \infty.
\end{split}
\end{align}
The assumption that $X \in \mathcal{C}([0,T], V)$ hence yields that
\begin{align}
\begin{split}
\sup\nolimits_{t \in [0,T]}  \| X_t - \Y_t^{n} \|_V \leq \sup\nolimits_{t \in [0,T]}  \| X_t\|_V + \sup\nolimits_{t \in [0,T]} \|\Y_t^{n} \|_V < \infty.
\end{split}
\end{align}
Moreover, note that
\begin{align}\label{eq:cor:pathwise4}
\begin{split}
& \smallint_0^T \left\|  ( \Id_V - P_n ) S_{s}   \right\|_{L(W, V)}  ds \leq \big( \|\Id_V\|_{L(V)} + \|P_n\|_{L(V)} \big) \smallint_0^T \|S_{s} \|_{L(W, V)} \, ds \\
&\leq \Upsilon \, \big( 1 + \|P_n\|_{L(V)} \big) \smallint_0^T s^{-\alpha} \, ds = \frac{ \Upsilon \, ( 1 + \|P_n\|_{L(V)} ) \, T^{1-\alpha}}{(1- \alpha)} < \infty.
\end{split}
\end{align}
In the next step we combine \eqref{eq:cor:pathwise1}-\eqref{eq:cor:pathwise4} with the generalized Gronwall lemma in Chapter~7 in Henry~\cite{h81} (see, e.g., Corollary 1.4.6 in \cite{Jentzen2015SPDElecturenotes}) to obtain that for all $t \in [0,T]$ it holds that
\begin{align}
\begin{split}
& \sup\nolimits_{s \in [0,t]} \|X_s - \Y_s^{n} \|_V  \\
& \leq \bigg[ \sup\nolimits_{s \in [0,T]}  \| O_s - \mathcal{O}_s^{n} \|_V + \frac{  \, \Upsilon \, T^{1-\alpha}}{(1-\alpha)}  \Psi \big( \sup\nolimits_{s \in [0, T]} \|X_s\|_V\big)  \left[\sup\nolimits_{s \in [0, T]} \big\| X_s - X_{\lf s \rf_{h_n}} \big\|_V \right] \\
&  +  \frac{ \Upsilon \, |h_n|^{\chi} \, T^{1-\alpha}    }{ (1-\alpha)} \left( \sup\nolimits_{s \in [0, T]} \left\| F(X_s) \right\|_W \right)\left( \sup\nolimits_{s \in [0, T]} \left[ \left\| X_s \right\|_V + \left\|\mathcal{O}_s^{n} \right\|_V \right] \right) \\
& + \left( \sup\nolimits_{s \in [0, T]} \left\| F(X_s) \right\|_W \right) \smallint_0^T \left\|  ( \Id_V - P_n ) S_{s}   \right\|_{L(W, V)}  ds  \bigg] \cdot \mathrm{E}_{1-\alpha} \Big[ t \,
\big| \Gamma (1-\alpha) \, \Upsilon\big|^{\nicefrac{1}{(1-\alpha)}} \big[ \left\|P_n \right\|_{L(V)}  \\
& \cdot \Psi\big( \sup\nolimits_{s \in [0, \max\{0, \lceil t \rceil_{h_n}-h_n \}]}\max \!\big\{\|X_{s}\|_V , \|\Y_s^n\|_V \big\} \big)  +   |h_n|^{\chi} \left( \sup\nolimits_{s \in [0, T]} \left\| F(X_s) \right\|_W \right)  \big]^{\nicefrac{1}{(1-\alpha)}} \Big] < \infty.
\end{split}
\end{align}
The proof of Corollary \ref{cor:pathwise} is thus completed.
\end{proof}


\subsection{Pathwise convergence}

\begin{prop}\label{prop:pathwise}
Assume the setting in Section \ref{setting:pathwise} and assume that $ \limsup_{n \to \infty} \|P_n\|_{L(V)} < \infty$ and  $\limsup_{n \to \infty} ( \int_0^T \|   ( \Id_V - P_n ) S_s  \|_{L(W, V)}  \, ds + \sup_{s \in [0, T]}   \|O_s - \mathcal{O}_s^n\|_V )=0$. Then 
\begin{enumerate}[(i)]
\item
it holds that $\limsup_{n \to \infty} \sup \nolimits_{s \in [0, T]} \| X_s- \Y_s^n \|_V=0$ and
\item
it holds that there exists a real number $ C \in  (0,\infty) $ such that for all $ n \in \N $ it holds that
\begin{equation}\label{prop:pahthwise:eq2}
\begin{split}
& \sup_{ s \in [0,T] } \| X_s - \mathcal{X}^n_s \|_V \\
&\leq C \left[ \sup_{s \in [0, T]}\|O_s - \mathcal{O}_s^n\|_V  + |h_n|^{\chi} + \sup_{s \in [0, T]} \| X_s - X_{\lf s \rf_{h_n}}\|_V  + \smallint_0^T \|   ( \Id_V - P_n ) S_{s}   \|_{L(W, V)} \,  ds \right].
\end{split}
\end{equation}
\end{enumerate}
\end{prop}
\begin{proof}[Proof of Proposition \ref{prop:pathwise}.]
Note that the assumption that $\limsup_{n \to \infty} \sup\nolimits_{s \in [0,T]}  \| O_s - \mathcal{O}_s^{n} \|_V =0$ and the assumption that $O \in \mathcal{C}([0,T], V)$ imply that
\begin{align}
\begin{split}\label{prop:pathwise:O}
\limsup\nolimits_{n \to \infty} \sup\nolimits_{s \in [0,T]}  \|\mathcal{O}_s^{n} \|_V \leq \sup\nolimits_{s \in [0,T]}  \| O_s  \|_V +\limsup\nolimits_{n \to \infty} \sup\nolimits_{s \in [0,T]}  \| O_s - \mathcal{O}_s^{n}\|_V  < \infty.
\end{split}
\end{align}
This and the assumption  that $\limsup_{n \to \infty} ( \int_0^T \|   ( \Id_V - P_n ) S_s  \|_{L(W, V)}  \, ds + \sup_{s \in [0, T]}   \|O_s - \mathcal{O}_s^n\|_V )=0$ yield that
\begin{align}\label{pathwise:limit}
\begin{split}
&\limsup_{n \to \infty} \bigg[ \sup\nolimits_{s \in [0,T]}  \| O_s - \mathcal{O}_s^{n} \|_V + \frac{\Upsilon \, T^{1-\alpha}}{(1-\alpha)}  \Psi \big( \sup\nolimits_{s \in [0, T]} \|X_s\|_V\big)  \left[\sup\nolimits_{s \in [0, T]} \| X_s - X_{\lf s \rf_{h_n}} \|_V \right] \\
&  +  \frac{ \Upsilon \, |h_n|^{\chi} \, T^{1-\alpha}    }{ (1-\alpha)} \left( \sup\nolimits_{s \in [0, T]} \left\| F(X_s) \right\|_W \right)\left( \sup\nolimits_{s \in [0, T]} \left[ \left\| X_s \right\|_V + \left\|\mathcal{O}_s^{n} \right\|_V \right] \right) \\
& + \left( \sup\nolimits_{s \in [0, T]} \left\| F(X_s) \right\|_W \right) \smallint_0^T \left\|  ( \Id_V - P_n )  S_s   \right\|_{L(W, V)}  ds  \bigg] = 0.
\end{split}
\end{align}
Combining this with \eqref{prop:pathwise:O}  and the fact that $\sup_{ n \in \N } \big[ \| P_n \|_{ L(V) } + | h_n |^{ \chi } \big] < \infty$ ensures that there exists a natural number $N \in \N$ such that
\begin{align}\label{prop:pathwise:sup}
\sup\nolimits_{n \in \{N, N+1, \ldots\}}  \sup\nolimits_{s \in [0,T]}  \|\mathcal{O}_s^{n} \|_V < \infty
\end{align}
and
\begin{align}\label{prop:pathwise:1}
\begin{split}
& \sup_{n \in \{N, N+1, \ldots\}}  \Bigg( \bigg[\sup\nolimits_{s \in [0,T]}  \| O_s - \mathcal{O}_s^{n} \|_V + \frac{  \Upsilon \, T^{1-\alpha}}{(1-\alpha)}  \Psi \big( \sup\nolimits_{s \in [0, T]} \|X_s\|_V\big)  \left[\sup\nolimits_{s \in [0, T]} \| X_s - X_{\lf s \rf_{h_n}} \|_V \right] \\
&  +  \frac{ \Upsilon \, |h_n|^{\chi} \, T^{1-\alpha}    }{ (1-\alpha)} \left( \sup\nolimits_{s \in [0, T]} \left\| F(X_s) \right\|_W \right)\left( \sup\nolimits_{s \in [0, T]} \left[ \left\| X_s \right\|_V + \left\|\mathcal{O}_s^{n} \right\|_V \right] \right) \\
& + \left( \sup\nolimits_{s \in [0, T]} \left\| F(X_s) \right\|_W \right) \smallint_0^T \left\|  ( \Id_V - P_n ) S_s   \right\|_{L(W, V)}  ds  \bigg] \cdot \mathrm{E}_{1-\alpha} \Big[ T \, \big|\Gamma (1-\alpha) \,  \Upsilon\big|^{\nicefrac{1}{(1-\alpha)}} \big[  \left\|P_n \right\|_{L(V)}\\
& \cdot   \Psi\big( \sup\nolimits_{s \in [0,T]} \|X_s\|_V + 1  \big)  +   |h_n|^{\chi} \left( \sup\nolimits_{s \in [0, T]} \left\| F(X_s) \right\|_W \right)  \big]^{\nicefrac{1}{(1-\alpha)}} \Big]\Bigg) \leq 1 < \infty.
\end{split}
\end{align}	
Moreover, observe that the triangle inequality shows for all $t \in [0, T]$, $n \in \N$ that
\begin{align}\label{prop:pathwise:max}
\begin{split}
&\sup\nolimits_{s \in [0, t]} \max \!\big\{\|X_{s}\|_V , \|\Y_s^n\|_V \big\} \leq \sup\nolimits_{s \in [0, t]} \max \!\big\{\|X_{s}\|_V , \|X_{s}\|_V + \|X_s -\Y_s^n\|_V  \big\} \\
&= \sup\nolimits_{s \in [0, t]}  \big( \|X_{s}\|_V+ \|X_s - \Y_s^n\|_V \big) \leq  \sup\nolimits_{s \in [0, T]} \|X_{s}\|_V  + \sup\nolimits_{s \in [0, t]} \|X_s - \Y_s^n\|_V \\
& \leq \sup\nolimits_{s \in [0, T]}   \|X_{s}\|_V+ 2 \sup\nolimits_{s \in [0, t]} \max \!\big\{\|X_{s}\|_V , \|\Y_s^n\|_V \big\}.
\end{split}
\end{align}	
Combining Corollary \ref{cor:pathwise} with \eqref{prop:pathwise:sup} and the fact that $\Psi$ is non-decreasing hence proves for all $n \in \{N, N+1, \ldots \}$, $ k \in \N_0 \cap (-\infty, \nicefrac{T}{h_n}-1]$ that  $\sup\nolimits_{s \in [0, k h_n]} \|X_s -\Y_s^n\|_V < \infty$ and
\begin{align}\label{eq:prop:pathwise}
\begin{split}
& \sup\nolimits_{s \in [0, (k+1)h_n]} \|X_s - \Y_s^{n} \|_V  \\
& \leq \bigg[ \sup\nolimits_{s \in [0,T]}  \| O_s - \mathcal{O}_s^{n} \|_V + \frac{   \Upsilon \, T^{1-\alpha}}{(1-\alpha)}  \Psi \big( \sup\nolimits_{s \in [0, T]} \|X_s\|_V\big)  \left[\sup\nolimits_{s \in [0, T]} \| X_s - X_{\lf s \rf_{h_n}} \|_V \right] \\
&  +  \frac{ \Upsilon \, |h_n|^{\chi} \, T^{1-\alpha}    }{ (1-\alpha)} \left( \sup\nolimits_{s \in [0, T]} \left\| F(X_s) \right\|_W \right)\left( \sup\nolimits_{s \in [0, T]} \left[ \left\| X_s \right\|_V + \left\|\mathcal{O}_s^{n} \right\|_V \right] \right) \\
& + \left( \sup\nolimits_{s \in [0, T]} \left\| F(X_s) \right\|_W \right) \smallint_0^T \left\|  ( \Id_V - P_n ) S_{s}   \right\|_{L(W, V)}  ds  \bigg] \cdot \mathrm{E}_{1-\alpha} \Big[ T \,
\big|\Gamma (1-\alpha) \, \Upsilon\big|^{\nicefrac{1}{(1-\alpha)}} \big[  \left\|P_n \right\|_{L(V)}\\
& \cdot   \Psi\big( \sup\nolimits_{s \in [0,T]} \|X_s\|_V + \sup\nolimits_{s \in [0, k h_n]} \|X_s -\Y_s^n\|_V  \big)  +   |h_n|^{\chi} \left( \sup\nolimits_{s \in [0, T]} \left\| F(X_s) \right\|_W \right)  \big]^{\nicefrac{1}{(1-\alpha)}} \Big].
\end{split}
\end{align}
Next let $ n \in \{N, N+1, \ldots\}$. We then claim that for all $ k \in \N_0 \cap [0, \nicefrac{T}{h_n}]$ it holds that
\begin{align}\label{ind:gronwall}
 \sup\nolimits_{s \in [0, kh_n]}\|X_s - \Y_s^{n}\|_V \leq 1.
\end{align}
We prove \eqref{ind:gronwall} by induction on $ k \in \N_0 \cap [0, \nicefrac{T}{h_n}]$. Combining \eqref{prop:pathwise:1} and the fact that $n \in \{N, N+1, \ldots\}$ with the fact that $ \forall \, x \in [0, \infty) \colon \mathrm{E}_{1-\alpha}[x] \geq 1$ shows  that
\begin{align}
\begin{split}
&\sup\nolimits_{s \in \{0\}} \|X_s - \Y_s^{n}\|_V= \|X_0 - \Y_0^{n}\|_V= \|O_0 - \mathcal{O}_0^{n}\|_V  \leq \sup\nolimits_{s \in [0,T]}  \| O_s - \mathcal{O}_s^{n} \|_V \leq 1.
\end{split}
\end{align}
This proves  \eqref{ind:gronwall} in the base case $k=0$. The induction step $\N_0 \cap (-\infty, \nicefrac{T}{h_n}-1] \ni k \to k+1 \in \N \cap (-\infty, \nicefrac{T}{h_n}]$ is an immediate consequence of \eqref{prop:pathwise:1}, \eqref{eq:prop:pathwise}, and the induction hypothesis. Induction hence proves \eqref{ind:gronwall}. Inequality \eqref{ind:gronwall}, in particular, shows  that for all $n \in \{N, N+1, \ldots\}$ it holds that 
\begin{align}\label{prop:pathwise:ceil}
 \sup\nolimits_{s \in [0, \max\{0, \lceil T \rceil_{h_n}-h_n \}]}  \|X_s - \Y_s^n \|_V \leq 1.
\end{align}
In the next step we combine \eqref{prop:pathwise:ceil} and the fact that $ \forall \, n \in \{N, N+1, \ldots\} \colon \sup_{s \in [0,T]} \|\mathcal{O}_s^n\|_V < \infty$ with Corollary \ref{cor:pathwise} and \eqref{prop:pathwise:max} to obtain that for all $n \in \{N, N+1, \ldots\}$ it holds that
\begin{align}\label{prop:pathwise:T}
\begin{split}
& \sup\nolimits_{s \in [0, T]} \|X_s - \Y_s^{n} \|_V  \\
& \leq \bigg[ \sup\nolimits_{s \in [0,T]}  \| O_s - \mathcal{O}_s^{n} \|_V + \frac{   \Upsilon \, T^{1-\alpha}}{(1-\alpha)}  \Psi \big( \sup\nolimits_{s \in [0, T]} \|X_s\|_V\big)  \left[\sup\nolimits_{s \in [0, T]} \| X_s - X_{\lf s \rf_{h_n}} \|_V \right] \\
&  +  \frac{ \Upsilon \, |h_n|^{\chi} \, T^{1-\alpha}    }{ (1-\alpha)} \left( \sup\nolimits_{s \in [0, T]} \left\| F(X_s) \right\|_W \right)\left( \sup\nolimits_{s \in [0, T]} \left[ \left\| X_s \right\|_V + \left\|\mathcal{O}_s^{n} \right\|_V \right] \right) \\
& + \left( \sup\nolimits_{s \in [0, T]} \left\| F(X_s) \right\|_W \right) \smallint_0^T \left\|  ( \Id_V - P_n ) S_{s}   \right\|_{L(W, V)}  ds  \bigg] \cdot \mathrm{E}_{1-\alpha} \Big[ T \,
\big|\Gamma (1-\alpha) \, \Upsilon\big|^{\nicefrac{1}{(1-\alpha)}} \big[  \left\|P_n \right\|_{L(V)}\\
& \cdot   \Psi\big( \sup\nolimits_{s \in [0,T]} \|X_s\|_V + 1  \big)  +   |h_n|^{\chi} \left( \sup\nolimits_{s \in [0, T]} \left\| F(X_s) \right\|_W \right)  \big]^{\nicefrac{1}{(1-\alpha)}} \Big].
\end{split}
\end{align}
This and the fact that $\sup_{ n \in \N } \big[ \| P_n \|_{ L(V) } + | h_n |^{ \chi } \big] < \infty$ imply  \eqref{prop:pahthwise:eq2}. Moreover, \eqref{prop:pathwise:T}, the fact that $\sup_{ n \in \N } \big[ \| P_n \|_{ L(V) } + | h_n |^{ \chi } \big] < \infty$, and \eqref{pathwise:limit} prove $\limsup_{ n \to \infty } \sup_{ s \in [0,T] } \| X_s - \mathcal{X}^N_s \|_V = 0$. The proof of Proposition \ref{prop:pathwise} is thus completed.
\end{proof}

\section{Strong convergence}
\label{sec:strong_error}

In this section we accomplish in Theorem~\ref{theorem:main} strong convergence for our approximation scheme. Before we establish Theorem~\ref{theorem:main}, we present in Lemma~\ref{lem:as}, Lemma~\ref{lem:charac}, Lemma~\ref{lem:outer}, Corollary~\ref{cor:as:outer}, and Proposition~\ref{prop:vitali} a few elementary results on an appropriate convergence concept for random fields. We employ Corollary~\ref{cor:as:outer} and Proposition~\ref{prop:vitali} in the proof of Theorem~\ref{theorem:main}.

\subsection{Weakly uniform convergence in probability}
\label{subsec:uniform}

\begin{lemma}\label{lem:as}
Let $(\Omega, \mathcal{F}, \P)$ be a probability space,  let $\P^* \colon \mathcal{P}(\Omega) \to [0, \infty]$ be the mapping with the property that for all $A \in \mathcal{P}(\Omega)$ it holds that
 $\P^*(A)= \inf\! \left( \left\{  \P(B) \in [0,1] \colon( B \in \mathcal{F} \text{ and } A \subseteq B  ) \right\} \right)$, let $\tilde{\Omega} \in \{ A \in \mathcal{P}(\Omega) \colon \P^*(A)=1\}$, and let $X_n \colon \Omega \to \R \cup \{\infty, -\infty\}$, $n \in \N$, be mappings which satisfy for all $\omega \in \tilde{\Omega}$ that $ \limsup_{n \to \infty} |X_n(\omega)|=0$. Then it holds for all $\varepsilon \in (0, \infty)$ that $\liminf_{n \to \infty} \P^*(|X_n| \leq \varepsilon)=1$.
\end{lemma}
\begin{proof}[Proof of Lemma~\ref{lem:as}.]
Throughout this proof let $Y_n \colon \Omega \to [0, \infty]$, $n \in \N$, be the mappings with the property that for all $n \in \N$ it holds that 
\begin{align}
Y_n= \sup_{m \in \N \cap [n, \infty)} |X_m|.
\end{align}
Note that the fact that $ \forall \, n \in \N \colon Y_{n+1} \leq Y_n$ ensures that for all $ n \in \N$,  $\varepsilon \in (0, \infty)$ it holds that $\{Y_n \leq \varepsilon\} \subseteq \{Y_{n+1} \leq \varepsilon \}$. Proposition~1.5.12 in Bogachev~\cite{Bogachev2007} and the fact that $\P^* \colon \mathcal{P}(\Omega) \to [0, \infty]$ is non-decreasing hence prove for all $ \varepsilon \in (0, \infty)$ that 
\begin{align}\label{eq:lem:as:1}
\P^*(\cup_{n \in \N} \{Y_n \leq \varepsilon\} ) = \liminf\limits_{n \to \infty} \P^*( Y_n \leq \varepsilon) \leq \liminf\limits_{n \to \infty} \P^*( |X_n| \leq \varepsilon).
\end{align}
Moreover, again the fact that $\P^* \colon \mathcal{P}(\Omega) \to [0, \infty]$ is non-decreasing shows that for all $\varepsilon \in (0, \infty)$ it holds that
\begin{align}
\begin{split}
\P^*(\cup_{n \in \N} \{Y_n \leq \varepsilon\} ) & = \P^*\!\big( \{\exists \, n \in \N \colon Y_n \leq \varepsilon \}\big) = \P^*\!\big( \{\exists \,  n \in \N  \colon ( \forall \, m \in \N \cap [n, \infty) \colon |X_m| \leq \varepsilon )\}\big) \\
&\geq \P^*(\tilde{\Omega}) =1.
\end{split}
\end{align}
Combining this with \eqref{eq:lem:as:1}, the fact that $\P^* \colon \mathcal{P}(\Omega) \to [0, \infty]$ is non-decreasing, and the fact that $\P^* |_{\mathcal{F}} = \P$ ensures that for all $\varepsilon \in (0, \infty)$ it holds that
\begin{align}
\begin{split}
1 \leq \P^*(\cup_{n \in \N} \{Y_n \leq \varepsilon\} )   \leq \liminf\limits_{n \to \infty} \P^*( |X_n| \leq \varepsilon) \leq  \P^*(\Omega)= \P(\Omega)=1. 
\end{split}
\end{align}
This completes proof of Lemma~\ref{lem:as}.
\end{proof}


\begin{lemma}\label{lem:charac}                    
Let $I$ be a non-empty set, let $(\Omega, \mathcal{F}, \P)$ be a probability space, let $c \in (0, \infty)$, and let $X^n \colon I \times \Omega \to \R \cup \{\infty, -\infty\}$, $n \in \N$, be random fields. Then the following three statements are equivalent:
\begin{enumerate}[(i)]
     \item It holds for all $\varepsilon \in (0, \infty)$ that 
	$\limsup_{n \to \infty} \sup_{i \in I} \P( |X_i^n| \geq \varepsilon)=0$.
	\item It holds for all $\varepsilon \in (0, \infty)$ that 
	$\liminf_{n \to \infty}\inf_{i \in I} \P( |X_i^n| \leq \varepsilon)=1$.
	\item It holds that $\limsup_{n \to \infty} \sup_{i \in I} \E \big[\! \min\{c, |X_i^n|\}\big]=0$.
\end{enumerate}
\end{lemma}
\begin{proof}[Proof of Lemma~\ref{lem:charac}.]
First, note that Markov's inequality proves for all $\varepsilon \in (0,c)$, $n \in \N$, $i \in I$ that
\begin{align}
\P(|X_i^n| \geq \varepsilon) = \P(\min\{c, |X_i^N|\} \geq \varepsilon) \leq \frac{\E \big[\! \min\{c, |X_i^n|\}\big]}{\varepsilon}.
\end{align}
This shows that $((iii) \Rightarrow (i))$. In the next step observe for all $\varepsilon \in (0, \infty)$ that 
\begin{align}
\begin{split}
\limsup_{n \to \infty} \sup_{i \in I} \P( |X_i^n| \geq \varepsilon) & = \limsup_{n \to \infty} \sup_{i \in I} \big[ 1 - \P( |X_i^n| < \varepsilon) \big] =  \limsup_{n \to \infty}  \left[ 1 - \inf_{i \in I} \P( |X_i^n| < \varepsilon) \right] \\
& = 1- \liminf_{n \to \infty}\inf_{i \in I} \P( |X_i^n| < \varepsilon). 
\end{split}
\end{align}
This ensures that $((i) \Leftrightarrow  (ii))$. It thus remains to prove that $((i) \Rightarrow (iii))$. Note that for all $\varepsilon \in (0, \infty)$ it holds that
\begin{align}
\begin{split}
&\limsup_{n \to \infty} \sup_{i \in I}\E \big[\! \min\{c, |X_i^n|\}\big] \\
& \leq \limsup_{n \to \infty} \sup_{i \in I} \E \big[\mathbbm{1}_{\{|X_i^n| \geq \varepsilon \}} \min\{c, |X_i^n|\}\big] + \limsup_{n \to \infty} \sup_{i \in I} \E \big[\mathbbm{1}_{\{|X_i^n| < \varepsilon \}} \min\{c, |X_i^n|\}\big] \\
& \leq c  \limsup_{n \to \infty} \sup_{i \in I} \P(|X_i^n| \geq \varepsilon) + \varepsilon.
\end{split}
\end{align}
This shows that $((i) \Rightarrow (iii))$. The proof of Lemma~\ref{lem:charac} is thus completed.
\end{proof}


\begin{lemma}\label{lem:outer}
Let $\Omega$ and $I$ be non-empty sets, let $\mu \colon \mathcal{P}(\Omega) \to [0, \infty]$ be a non-decreasing mapping, and let $X^n \colon I \times \Omega \to \R \cup \{\infty, -\infty\}$, $n \in \N$, be mappings. Then it holds for all $\varepsilon \in (0, \infty)$, $n \in \N$ that $\inf_{i \in I} \mu( |X_i^n| \leq \varepsilon) \geq \mu(\sup\nolimits_{i \in I} |X_i^n| \leq \varepsilon)$ and $ \liminf_{m \to \infty} \inf_{i \in I} \mu( |X_i^m| \leq \varepsilon) \geq \liminf_{m \to \infty} \mu(\sup\nolimits_{i \in I} |X_i^m| \leq \varepsilon)$.
\end{lemma}
\begin{proof}[Proof of Lemma~\ref{lem:outer}.]
Note that the fact that $\mu \colon \mathcal{P}(\Omega) \to [0, \infty]$ is non-decreasing ensures that for all $n \in \N$, $j \in I$, $\varepsilon \in (0, \infty)$ it holds that
\begin{align}
\mu( |X_j^n| \leq \varepsilon) \geq \mu(\sup\nolimits_{i \in I} |X_i^n| \leq \varepsilon).
\end{align}
This yields for all $n \in \N$, $\varepsilon \in (0, \infty)$ that $\inf_{i \in I} \mu( |X_i^n| \leq \varepsilon) \geq \mu(\sup\nolimits_{i \in I} |X_i^n| \leq \varepsilon)$.
This completes the proof of Lemma~\ref{lem:outer}.
\end{proof}

Informally speaking, the following corollary, Corollary~\ref{cor:as:outer},
shows that convergence uniformly in an index set $I$ on a measurable set of probability $1$
implies convergence in probability uniformly in the index set.
This statement is nontrivial since arbitrary suprema over random variables are, in general, 
not random variables.
\begin{cor}\label{cor:as:outer}
Let $(\Omega, \mathcal{F}, \P)$ be a probability space, let $\tilde{\Omega} \in \{ A \in \mathcal{F} \colon \P(A)=1\}$, let $I$ be a non-empty set, and let $X^n \colon I \times \Omega \to \R \cup \{\infty, -\infty\}$, $n \in \N$, be random fields which satisfy for all $\omega \in \tilde{\Omega}$ that $\limsup_{n \to \infty} \sup_{i \in I} |X_i^n(\omega)|=0$. Then it holds for all $\varepsilon \in (0, \infty)$ that $\limsup_{n \to \infty} \sup_{i \in I} \P(|X_i^n| \geq \varepsilon)=0$.
\end{cor}
\begin{proof}[Proof of Corollary~\ref{cor:as:outer}.]
Throughout this proof let $\P^* \colon \mathcal{P}(\Omega) \to [0, \infty]$ be the mapping with the property that for all $A \in \mathcal{P}(\Omega)$ it holds that $\P^*(A)= \inf\! \left( \left\{  \P(B) \in [0,1] \colon( B \in \mathcal{F} \text{ and } A \subseteq B  ) \right\}  \right)$ and let $Y_n \colon \Omega \to \R \cup \{\infty, -\infty \}$, $n \in \N$, be the mappings with the property that for all $n \in \N$, $\omega \in \Omega$ it holds that $ Y_n(\omega)= \sup_{i \in I} |X_i^n(\omega)|$. Next note that $\P^*(\tilde{\Omega})=1$. Combining Lemma~\ref{lem:as} with the fact that $ \forall \, \omega \in \tilde{\Omega} \colon  \limsup_{n \to \infty} |Y_n(\omega)|=0$  hence proves for all $\varepsilon \in (0, \infty)$ that $\liminf_{n \to \infty} \P^*(|Y_n| \leq \varepsilon)=1$. This implies  for all $\varepsilon \in (0, \infty)$ that $\liminf_{n \to \infty} \P^*(\sup_{i \in I} |X_i^n| \leq \varepsilon)=1$.  The fact that $\P^*|_{\mathcal{F}} = \P$ and Lemma~\ref{lem:outer} therefore prove that
\begin{align}
\liminf_{n \to \infty} \inf_{i \in I} \P( |X_i^n| \leq \varepsilon) = \liminf_{n \to \infty} \inf_{i \in I} \P^*( |X_i^n| \leq \varepsilon) \geq \liminf_{n \to \infty} \P^*(\sup\nolimits_{i \in I} |X_i^n| \leq \varepsilon)=1.
\end{align}
Hence, it holds for all $\varepsilon \in (0, \infty)$ that $\liminf_{n \to \infty} \inf_{i \in I} \P( |X_i^n| \leq \varepsilon)=1$. Combining this with Lemma~\ref{lem:charac} shows that for all $ \varepsilon \in (0, \infty)$ it holds that $\limsup_{n \to \infty} \sup_{i \in I} \P( |X_i^n| \geq \varepsilon)=0$. The proof of Corollary~\ref{cor:as:outer} is thus completed.
\end{proof}


Informally speaking, the following proposition, Proposition~\ref{prop:vitali},
proves for all $p\in(0,\infty)$ that convergence in probability uniformly in an index set $I$
together with uniform moment bounds of the approximations implies for every $q\in(0,p)$
$L^q$-convergence uniformly in $I$.
In applications to stochastic processes the index set $I$ can be a time interval.
\begin{prop}\label{prop:vitali}
Let $I$ be a non-empty set, let $(\Omega, \mathcal{F}, \P)$ be a probability space, let $p \in (0, \infty)$, let $(V, \left\| \cdot \right\|_V)$ be a separable normed $\R$-vector space, and let $X^n \colon I \times \Omega \to V$, $n \in \N_0$, be random fields which satisfy for all $\varepsilon \in (0, \infty)$ that $\limsup_{n \to \infty} \sup_{i \in I} \E \big[\|X_i^n\|_V^p\big] < \infty$ and $\limsup_{n \to \infty} \sup_{i \in I} \P(\|X_i^0- X_i^n \|_V \geq \varepsilon)=0$. Then it holds for all $q \in (0, p)$ that  $\limsup_{n \to \infty} \sup_{i \in I} \E\big[ \|X_i^0 - X_i^n \|_V^q\big]=0$ and $\sup_{i \in I} \E\big[\|X_i^0\|_V^p\big]< \infty$.
\end{prop}
\begin{proof}[Proof of Proposition~\ref{prop:vitali}.]
First, observe that, e.g., Lemma~3.10 in~\cite{Hutzenthaler2015}, the assumption that $\forall \, \varepsilon \in (0, \infty) \colon \limsup_{n \to \infty} \sup_{i \in I} \P(\|X_i^0- X_i^n \|_V \geq \varepsilon)=0$,  and the assumption that $\limsup_{n \to \infty} \sup_{i \in I}$ $ \E \big[\|X_i^n\|_V^p\big] < \infty$  yield that
\begin{align}\label{eq:bound}
\sup_{i \in I} \E\big[\|X_i^0\|_V^p\big] \leq \sup_{i \in I} \liminf_{n \to \infty} \E\big[\|X_i^n\|_V^p\big]  \leq \limsup_{n \to \infty} \sup_{i \in I}  \E \big[\|X_i^n\|_V^p\big] < \infty.
\end{align}
Next note that H\"olders inequality ensures for all $q \in (0, p)$, $n \in \N$ that
\begin{align}
\begin{split}
&\sup_{i \in I} \E\big[\|X_i^0 -X_i^n\|_V^q\big]  = \sup_{i \in I} \bigg( \E\!\left[ \mathbbm{1}_{\{\|X_i^0 -X_i^n \|_V \geq 1\}}\|X_i^0 -X_i^n\|_V^q\right] + \E\!\left[ \mathbbm{1}_{\{\|X_i^0 -X_i^n \|_V < 1\}}\|X_i^0 -X_i^n\|_V^q\right] \bigg) \\
& \leq \sup_{i \in I} \Big( \big| \P(\|X_i^0 -X_i^n \|_V \geq 1)\big|^{\frac{p-q}{q}} \big( \E\big[ \|X_i^0 -X_i^n\|_V^p\big] \big)^{\frac{q}{p}} \Big) + \sup_{i \in I} \E \big[ \!\min\{1, \|X_i^0 -X_i^n\|_V^q \}\big].
\end{split}
\end{align}
The fact that $\forall \, a, b \in \R \colon |a+b|^p \leq 2^p(|a|^p + |b|^p)$ together with the triangle inequality hence shows for all $q \in (0, p)$, $n \in \N$ that
\begin{align}\label{eq:holder}
\begin{split}
&\sup_{i \in I} \E\big[\|X_i^0 -X_i^n\|_V^q\big]  \leq 2^q \sup_{i \in I} \Big( \big| \P(\|X_i^0 -X_i^n \|_V \geq 1)\big|^{\frac{p-q}{q}} \big( \E\big[ \|X_i^0 \|_V^p\big] \big)^{\frac{q}{p}} \Big)\\
&+ 2^q \sup_{i \in I} \Big( \big| \P(\|X_i^0 -X_i^n \|_V \geq 1)\big|^{\frac{p-q}{q}} \big( \E\!\left[ \|X_i^n \|_V^p\right] \big)^{\frac{q}{p}} \Big) + \sup_{i \in I} \E \big[ \!\min\{1, \|X_i^0 -X_i^n\|_V^q \}\big].
\end{split}
\end{align}
Moreover, observe that  Lemma~\ref{lem:charac} and the assumption that $ \forall \, \varepsilon \in (0, \infty) \colon \limsup_{n \to \infty} \sup_{i \in I} \P(\|X_i^0- X_i^n \|_V \geq \varepsilon)=0$  prove that  for all $q \in (0, p)$ it holds that
\begin{align}
\limsup_{n \to \infty} \sup_{i \in I} \E \big[ \!\min\{1, \|X_i^0 -X_i^n\|_V^q \}\big]=0.
\end{align}
This, \eqref{eq:holder}, \eqref{eq:bound}, the fact that $\limsup_{n \to \infty} \sup_{i \in I} \P(\|X_i^0- X_i^n \|_V \geq 1)=0$, and the assumption that $\limsup_{n \to \infty} \sup_{i \in I} \E \big[\|X_i^n\|_V^p\big] < \infty$  yield that for all $q \in (0, p)$ it holds that
\begin{align}
\limsup_{n \to \infty} \sup_{i \in I} \E\big[\|X_i^0 -X_i^n\|_V^q\big] = 0.
\end{align}
Combining this with \eqref{eq:bound} completes the proof of Proposition~\ref{prop:vitali}.
\end{proof}

\subsection{Main result of this article}
\label{subsec:main}
\begin{theorem}\label{theorem:main}
Let $ ( H, \left< \cdot , \cdot \right>_H, \left\| \cdot \right\|_H ) $ be a separable $\mathbb{R}$-Hilbert space,  let $ ( \Omega, \mathcal{F}, \P ) $ be a probability space,  let $ \mathbb{H} \subseteq H$ 
be a non-empty orthonormal basis of $ H $, 
let $\lambda \colon \mathbb{H}\to (0, \infty)$ be a function with the property 
that $\inf_{b \in \mathbb{H}} \lambda_b > 0$, 
let $ A \colon D(A) \subseteq H \to H $ be the linear operator 
such that $ D(A) = \{ v \in H \colon \sum_{b \in \mathbb{H}} | \lambda_b \langle b , v \rangle_H |^2 < \infty \} $ 
and such that for all $ v \in D(A) $ it holds that 
$ A v = \sum_{b \in \mathbb{H}} - \lambda_b \langle b , v \rangle_H b$, 
let $ ( H_r, \left< \cdot , \cdot \right>_{ H_r }, \left\| \cdot \right\|_{ H_r } ) $, 
$ r \in \R $, be a family of interpolation spaces associated to $ - A $ 
(see, e.g., Definition 3.5.25 in  \cite{Jentzen2015SPDElecturenotes}), 
let   $ \alpha \in [ 0, \nicefrac{1}{2}]$, $\ctwo \in [0,1)$, 
$\rho \in [0, 1- \alpha)$, $\varrho \in (\rho, 1-\alpha)$, 
$ T, \theta, \vartheta \in (0,\infty)$, 
$ 
  \chi \in 
  ( 
    0
    ,  
    \nicefrac{ 
      ( 1 - \alpha - \rho ) 
    }{
      ( 1 + 2 \cfour ) 
    }
  ]
  \cap 
  (
    0
    ,
    \nicefrac{
      ( \varrho - \rho ) 
    }{
      ( 1 + \cfour) 
    }
  ] 
$, 
$p \in [2, \infty)$, $\varXi \in \{ B \in \mathcal{F} \colon \P(B)=1 \}$, $\tilde{\Omega} \in \mathcal{P}(\Omega)$, $F \in \mathcal{C}(H_{\rho}, H_{-\alpha} )$, 
$ \phi, \Phi \in \mathcal{M}( \mathcal{B}( H_1 ) , \mathcal{B}( [0,\infty) ) ) $, 
$ ( \mathbb{H}_n )_{ n \in \N } \in \mathbbm{M}( \N, \mathcal{P}_0( \mathbb{H} ) )$, 
$ ( P_n )_{ n \in \N } \in \mathbb{M}( \N, L( H ) )$, $(h_n)_{n \in \N} \in \mathbb{M}(\N, (0, \infty) )$, 
$(\mathbb{X}^n)_{n \in \N}  \in \mathbb{M}(\N, \mathbb{M}([0,T]\times \Omega, H_{\varrho}))$, 
let  $ \mathcal{X}^n \colon [0,T] \times \Omega \to H_{ \varrho } $, $ n \in \N $, and $ \mathcal{O}^n \colon [0,T] \times \Omega \to H_{ \varrho } $, $ n \in \N $, 
be stochastic processes, 
let $ X, O \colon  [0, T] \times \Omega \to H_{ \varrho} $ be stochastic processes 
with continuous sample paths, and assume for all $ v, w \in H_1 $, $ t \in [0,T] $, 
$ n \in \N $, $\omega \in \varXi$ that $ \left< v, F( v + w ) \right>_H \leq \phi( w ) \| v \|^2_H + \ctwo \| v \|^2_{ H_{ 1 / 2 } } + \Phi( w ) $, 
$ \left\| F(v) - F(w) \right\|_{ H_{ - \alpha } } \leq \theta \, ( 1 + \| v \|_{ H_{ \rho } }^{ \vartheta } 
+ \|w\|_{H_{\rho}}^{\vartheta}) \, \|v-w\|_{H_{\rho}} $, 
$ 
  \liminf_{ m \to \infty } 
  \inf( \{\lambda_b \colon b \in \mathbb{H} \backslash \mathbb{H}_m \} \cup \{\infty\}  ) = \infty 
$, 
$
  P_n(v) = 
  \sum_{ b \in \mathbb{H}_n } \langle b, v \rangle_H b
$, 
$
  \limsup_{ m \to \infty} h_m =0 
$,
$ 
  \mathcal{O}^n( [0,T] \times \Omega )
  \subseteq
  P_n( H ) 
$,
   $\tilde{\Omega}= \{ \forall \, s \in [0,T] \colon  X_s = \int_0^s e^{ ( s - u ) A } \, F( X_u ) \, du + O_s\} \cap \{ \forall \, m \in \N, s \in  [0, T] \colon \mathbb{X}_{\lf s \rf_{h_m}}^m = \mathcal{X}_{\lf s \rf_{h_m}}^m \} \cap  \varXi$,  $ 
   \limsup_{m\to \infty} \sup_{s \in [0, T]} \| O_s(\omega) - \mathcal{O}_s^m(\omega) \|_{H_{\varrho}} =0
   $,
$ 
  \limsup_{ m \to \infty}  
  \sup_{ s \in [0,T]} 
  \E\big[
    \| \mathcal{O}_s^m \|_H^p 
    +
    \int_0^T 
    \exp\!\big( 
      \int_r^T p \, \phi( \mathcal{O}_{ \lf u \rf_{h_m} }^m ) \, du 
    \big)  
    \max\{
      1, | \Phi( \mathcal{O}_{ \lf r \rf_{h_m} }^m) |^{ p / 2 } 
    \} \, dr 
  \big] < \infty
$, $ \mathbb{X}_t^n  =  \smallint\nolimits_0^t P_n \,  e^{  ( t - s ) A } \, \mathbbm{1}_{ \{ \| \Y_{ \lf s \rf_{h_n} }^n \|_{ H_{ \varrho } } +  \| \mathcal{O}_{ \lf s \rf_{h_n} }^n \|_{ H_{ \varrho } }  \leq | h_n|^{ - \chi } \}} \, F (  \Y_{ \lf s \rf_{ h_n } }^n ) \, ds  +\mathcal{O}_t^n$, and
$ 
 \P\big( X_t = \int_0^t e^{ ( t - s ) A } \, F( X_s ) \, ds + O_t \big)=\P (\mathbb{X}_t^n  = \mathcal{X}_t^n )=1$.
Then
\begin{enumerate}[(i)]
	\item it holds that $\tilde{\Omega} \in \mathcal{F}$ and $\P(\tilde{\Omega})=1$,
	\item it holds for all $ \omega \in \tilde{\Omega}$ that $\limsup_{n \to \infty} \sup_{t \in [0,T]} \| X_t(\omega)- \mathbb{X}_t^{n}(\omega) \|_{ H_{ \varrho } } = 0$,
	\item it holds for every $n \in \N$ that $\mathbb{X}^n - \mathcal{O}^n \colon [0, T] \times \Omega \to H_{\varrho}$  is a stochastic process with continuous sample paths, 
	\item it holds that $
	\limsup_{ n \to \infty }
	\sup_{ t \in [0,T] }
	\E\big[ 
	\| X_t \|^p_H
	+
	\| \mathcal{X}^n_t \|_H^p
	+
	 \sup_{ s \in [0,T] }
	\|
	\mathbb{X}^n_s - \mathcal{O}^n_s
	\|_H^p
	\big]
	< \infty
	$, and
	\item it holds 
	for all 
	$ q \in (0, p) $
	that
	$
	\limsup_{ n \to \infty } 
	\sup_{ t \in [0,T] } 
	\E\big[
	\| X_t - \Y_t^n \|_H^q 
	\big] 
	= 0
	$.
\end{enumerate}
\end{theorem}

\begin{proof}[Proof of Theorem~\ref{theorem:main}.]
Throughout this proof let $\tilde{X}, \tilde{O} \colon [0,T] \times \Omega \to H_{\varrho}$ be the mappings with the property that for all $t \in [0, T]$ it holds that $\tilde{X}_t= \mathbbm{1}_{\tilde{\Omega}} \, X_t$ and $\tilde{O}_t= \mathbbm{1}_{\tilde{\Omega}} \, O_t - \mathbbm{1}_{\Omega \backslash \tilde{\Omega}} \, \int_0^t e^{(t-s)A} \, F(0) \, ds$, let $(\tilde{\Y}^n)_{n \in \N}$, $(\tilde{\mathcal{O}}^n)_{n \in \N} \in \mathbb{M}(\N, \mathbb{M}([0,T]\times \Omega, H_{\varrho}))$ be the mappings with the property that for all $n \in \N$, $t \in [0, T]$ it holds that $\tilde{\mathcal{O}}_t^n = \mathbbm{1}_{\tilde{\Omega}} \, \mathcal{O}_t^n  + \mathbbm{1}_{\Omega \backslash \tilde{\Omega}} \,P_n \, \tilde{O}_t$ and  
\begin{align}
\tilde{\Y}_t^n  =  \smallint_0^t P_n \,  e^{  ( t - s ) A } \, \mathbbm{1}_{ \{ \|\tilde{\Y}_{ \lf s \rf_{h_n} }^n \|_{ H_{ \varrho } } +  \| \tilde{\mathcal{O}}_{ \lf s \rf_{h_n} }^n \|_{ H_{ \varrho } }  \leq | h_n|^{ - \chi } \}} \, F (  \tilde{\Y}_{ \lf s \rf_{ h_n } }^n ) \, ds  +\tilde{\mathcal{O}}_t^n,
\end{align} 
let
$
\tilde{\vartheta}, \tilde{\theta} \in (0,\infty)
$ 
be the real numbers given by $ \tilde{\vartheta} = 2 \vartheta $ 
and 
\begin{multline}
\tilde{\theta} = 
\max\!\bigg\{ 
3 \, \theta^2  
\bigg[
\sup_{ u \in H_{ - \alpha } \backslash\{ 0 \} 
}
\tfrac{ 
	\| u \|_{ H_{ - 1 / 2 } }^2 
}{
\| u \|_{ H_{ - \alpha } }^2 
} 
\bigg] 
\bigg[
1 + 
\sup_{ u \in H_{ \varrho } \backslash \{ 0 \} 
}
\tfrac{
	\| u \|_{ H_{ \rho } }^{ 2 \vartheta } 
}{
\| u \|_{H_{\varrho}}^{2 \vartheta}} 
\bigg]
\big(1+2^{\max\{2\vartheta-1, 0\}}\big) 
,  
\\
\big(
8 \theta^2 + 
2 \| F(0) \|_{ H_{ - \alpha } }^2 
\big) 
\max\!\bigg\{ 
1, 
\sup_{
	u \in H_{ \varrho } \backslash \{ 0 \} 
}
\tfrac{ 
	\| u \|_{ H_{ \rho } }^{ 2 + 2 \vartheta } 
}{
\| u \|_{ H_{ \varrho } }^{ 2 + 2 \vartheta } 
} 
\bigg\} 
\! \bigg\},
\end{multline}
and let 
$ 
\tilde{\phi}_n \colon 
P_n( H ) \to [0,\infty) 
$,
$ n \in \N $,
and 
$ 
\tilde{\Phi}_n \colon
P_n( H ) \to [0,\infty) 
$,
$ n \in \N $,
be the mappings with the property that
for all 
$ n \in \N $,
$ v \in P_n( H ) $ 
it holds that
$
\tilde{\phi}_n( v ) = 2 \cdot \phi( v )
$
and 
$
\tilde{\Phi}_n( v ) = 2 \cdot \Phi( v )
$. Next observe that the assumption that $X, O \colon [0, T] \times \Omega \to H_{\varrho}$ are stochastic processes with continuous sample paths and the assumption that $\forall \, t \in [0, T] \colon  \P\big(X_t = \int_0^t e^{ ( t - s ) A } \, F( X_s ) \, ds + O_t\big)=1$ ensure that $\{ \forall \, t \in [0,T] \colon  X_t = \int_0^t e^{ ( t - s ) A } \, F( X_s ) \, ds + O_t\} \in \mathcal{F}$ and $\P\big(\forall \, t \in [0,T] \colon  X_t = \int_0^t e^{ ( t - s ) A } \, F( X_s ) \, ds + O_t\big)=1$. The assumption that $\forall \, n\in \N$,  $t \in [0, T] \colon \P(\mathbb{X}_t^n = \mathcal{X}_t^n)=1$ and the assumption that $\P(\varXi)=1$ hence yield that $\tilde{\Omega} \in \mathcal{F}$ and $\P(\tilde{\Omega})=1$. This establishes (i). In the next step we observe that, e.g., Theorem~4.7.6 
in \cite{Jentzen2015SPDElecturenotes} proves that
for all 
$ n \in \N $, $ t \in [0, T] $, $ \varepsilon \in (0, 1 - \varrho - \alpha) $ 
it holds that
\begin{align}
\label{eq:sup_s_alpha_finite}
  \sup\nolimits_{ s \in (0, T] } 
  \big( 
    s^{ ( \varrho + \alpha ) } 
    \| e^{ s A } \|_{ L( H_{ - \alpha }, H_{ \varrho } ) } 
  \big) 
  = 
  \sup\nolimits_{ s \in (0, T] }   
  \|
    ( - s A )^{ ( \varrho + \alpha ) } 
    \, e^{ s A } 
  \|_{ L(H) } 
  \leq 1 
  < \infty
\end{align}
and
\begin{align}
\begin{split}
& 
  \int_0^t \|   
    ( \Id_H - P_n ) 
    e^{ s A }   
  \|_{
    L( H_{ - \alpha } , H_{ \varrho } ) 
  }  
  \, ds 
\leq 
  \int_0^t 
  \|
    \Id_H - P_n 
  \|_{
    L( H_{ \varrho + \varepsilon } , H_{ \varrho } ) 
  } 
  \,
  \|  
    e^{ s A } 
  \|_{
    L( H_{ - \alpha }, H_{ \varrho + \varepsilon } ) 
  } \, ds  
\\ & = 
  \| 
    ( - A )^{ - \varepsilon }
    ( \Id_H - P_n )
  \|_{
    L( H )
  }
  \int_0^t  
  \|
    ( - A )^{ ( \varrho + \varepsilon + \alpha ) }  
    \,
    e^{ s A } 
  \|_{ L(H) } 
  \, ds  
\\ & \leq 
  \| 
    ( - A )^{ - 1 }
    ( \Id_H - P_n )
  \|_{
    L( H )
  }^{ \varepsilon }
  \int_0^t s^{ - ( \varrho + \varepsilon + \alpha ) } \, ds 
=  
  \frac{ 
    \| 
      ( - A )^{ - 1 }
      ( \Id_H - P_n )
    \|_{
      L( H )
    }^{ \varepsilon }
    \, 
    t^{ ( 1 - \varrho - \varepsilon - \alpha ) }
  }{
    ( 1 - \varrho - \alpha - \varepsilon ) 
  }
  .
\end{split}
\end{align}
This together with the assumption that 
$
  \liminf_{ n \to \infty} \inf( \{\lambda_b \colon b \in \mathbb{H} \backslash \mathbb{H}_n \} \cup \{\infty\} ) = \infty
$ 
ensures that
\begin{align}\label{eq:sup_Id_H_0} 
  \limsup_{ n \to \infty } 
  \sup_{ t \in [0, T] } 
  \left( 
    \int_0^t 
    \| 
      ( \Id_H - P_n ) 
      e^{ s A }
    \|_{ L( H_{ - \alpha } , H_{ \varrho } ) } 
  \, ds 
  \right) = 0 
  .
\end{align} 
In addition, the assumption that $ \forall \, \omega \in \varXi \colon \limsup_{n\to \infty} \sup_{t \in [0, T]} \| O_t(\omega) - \mathcal{O}_t^n(\omega) \|_{H_{\varrho}} = 0$ and the fact that $\tilde{\Omega} \subseteq \varXi$ imply that $\limsup_{n\to \infty} \sup_{t \in [0, T]} \| \tilde{O}_t - \tilde{\mathcal{O}}_t^n \|_{H_{\varrho}} = 0$.
Combining this with \eqref{eq:sup_s_alpha_finite}, \eqref{eq:sup_Id_H_0}, the fact that $\forall \, t \in [0,T] \colon  \tilde{X}_t = \int_0^t e^{ ( t - s ) A } \, F( \tilde{X}_s ) \, ds + \tilde{O}_t$,
and the fact that 
$
  \limsup_{ n \to \infty } 
  \| P_n \|_{ L( H_{ \varrho } ) } = 1 < \infty
$ 
allows us to apply
Proposition~\ref{prop:pathwise} 
(with $ V = H_{ \varrho } $, 
$ W = H_{ - \alpha } $, $ T = T $, 
$ \chi = \chi $, $ \alpha = \varrho + \alpha $, 
$ P_n = ( H_{ \varrho } \ni v \mapsto P_n( v ) \in H_{ \varrho } ) $, 
$ h_n = h_n $, 
$ F = F|_{ H_{ \varrho } } \in \mathcal{C} (H_{\varrho}, H_{-\alpha}) $,  
$ S_t = ( H_{ - \alpha } \ni v \mapsto e^{ t A } v \in H_{ \varrho } ) $ 
for 
$ t \in (0,T] $, 
$ n \in \N $ 
in the notation of Proposition \ref{prop:pathwise}) 
to obtain that 
$
  \limsup_{n \to \infty} \sup_{t \in [0, T]} \| \tilde{X}_t- \tilde{\Y}_t^n \|_{H_\varrho}=0
$. This together with the fact that $\forall \, \omega \in \tilde{\Omega}$, $t \in [0, T]$, $n \in \N \colon \tilde{X}_t(\omega)- \tilde{\Y}_t^n(\omega) = X_t(\omega) - \mathbb{X}_t^n(\omega)$ proves (ii). In the next step note that Lemma~\ref{lem:YX_in_DA} yields for every $n \in \N$ that $\mathbb{X}^n- \mathcal{O}^n \colon [0, T] \times \Omega \to H_{\varrho}$ is a stochastic process with continuous sample paths. This establishes (iii).
Next observe that Lemma~\ref{lemma:cond} implies 
that for all $v, w \in H_{\varrho}$ it holds that 
\begin{align}
\label{eq:F_condition_proof_1}
  \| 
    F(v) - F(w) 
  \|_{ 
    H_{ - 1 / 2 } 
  }^2 
  \leq 
  \tilde{\theta} \max\{1, \|v\|_{H_{\varrho}}^{\tilde{\vartheta}}\}\|v-w\|_{H_{\rho}}^2 + \tilde{\theta} \, \|v-w\|_{H_{\rho}}^{2+\tilde{\vartheta}}
\end{align}
and 
\begin{equation}
\label{eq:F_condition_proof_2}
  \|F(v)\|_{H_{-\alpha}}^2 \leq \tilde{\theta} \max\{ 1, \|v\|_{H_{\varrho}}^{2 + \tilde{\vartheta}} 
  \}
  .
\end{equation}
In addition, observe that the assumption that
$
  \chi \in
  ( 
    0
    ,  
    \nicefrac{ 
      ( 1 - \alpha - \rho ) 
    }{
      ( 1 + 2 \cfour ) 
    }
  ]
$,
in particular,
assures that
$
  \chi 
  \in 
  \big[ 
    0,
    ( 2 - 2 \alpha - 2 \rho ) / ( 2 + \tilde{\cfour} )
  \big]
$.
Combining this with 
\eqref{eq:F_condition_proof_1}
and \eqref{eq:F_condition_proof_2}
enables us to apply 
Corollary~\ref{cor:bounds}
(with 
$ H = P_n( H ) $, 
$ (\Omega, \mathcal{F}, \P) = (\Omega, \mathcal{F}, \P) $, 
$ A = ( P_n( H ) \ni v \mapsto A v \in P_n( H ) ) \in L( P_n( H ) ) $, 
$ 
  O = ( [0,T] \times \Omega \ni (t,\omega) \mapsto \tilde{\mathcal{O}}^n_t( \omega ) \in P_n( H ) ) 
$, 
$ Y = ( [0,T] \times \Omega \ni (t,\omega) \mapsto \tilde{\mathcal{X}}^n_t( \omega ) \in P_n( H ) ) $, 
$ 
  F = 
  ( P_n( H ) \ni v \mapsto P_n F( v ) \in P_n(H) ) 
  \in 
  \mathcal{C}( P_n( H ), P_n( H ) ) 
$,  
$ \phi = \tilde{\phi}_n $, 
$ \Phi = \tilde{\Phi}_n $, 
$ \ctwo= \ctwo$, $\alpha = \alpha$, $\rho= \rho$, $ \varrho = \varrho$,  
$ \cthree= \tilde{\theta}$, $\cfour = \tilde{\vartheta}$, $\chi = \chi $, 
$T =T$, $ h = h_n$, $p= p$ for $ n \in \{ m \in \N \colon h_m \leq 1 \}$ 
in the notation of Corollary \ref{cor:bounds}) 
to obtain that for all $ n \in \N  $
with $ h_n \leq 1 $ it holds that
\begin{equation}
\label{eq:cor_application}
\begin{split}
& 
  \big\| 
    \sup\nolimits_{ t \in [0,T] } 
    \| \tilde{\mathcal{X}}^n_t - \tilde{\mathcal{O}}^n_t \|_H 
  \big\|_{ \mathcal{L}^p( \P ; \R ) } 
\\
& 
  \leq
  \sqrt{
  \left[
    1
    +
    \tfrac{  
      \tilde{\cthree} 
      \,
      ( 1 + | \tilde{\cthree} |^{ 1 / 2 } )^{ 2 + \tilde{\cfour} } 
      \,   
      | h_n |^{ \min\{ 2 \varrho, 2 - 2 \alpha - \tilde{\cfour} \chi \} 
      - 2 \rho 
      - ( 2 + \tilde{\cfour} ) \chi }  
    }{
      ( 1 / 2 - \ctwo / 2 ) \, 
      ( 1 - \alpha - \rho )^{ 2 + \tilde{\cfour} }
    }
  \right]
  \smallint_0^T
    \left|  
      \E\!\left[
        e^{ \frac{ p }{ 2 } \int_s^T \tilde{\phi}_n( \tilde{\mathcal{O}}^n_{ \lf u \rf_{ h_n } } ) \, du }
        \max\!\big\{
          1 , | \tilde{\Phi}_n( \tilde{\mathcal{O}}^n_{ \lf s \rf_{ h_n } } ) |^{ \nicefrac{ p }{ 2 } }  
        \big\} \right]
      \right|^{ \frac{ 2 }{ p } }
    ds
  }
\\
& 
  \leq
  \sqrt{
  2
  \left[
    1
    +
    \tfrac{  
      \tilde{\cthree} 
      \,
      ( 1 + | \tilde{\cthree} |^{ 1 / 2 } )^{ 2 + 2 \cfour } 
      \,   
      | h_n |^{ 
        2 [ \min\{ \varrho, 1 - \alpha - \cfour \chi \} 
        - \rho 
        - ( 1 + \cfour ) \chi 
        ]
      }  
    }{
      ( 1 / 2 - \ctwo / 2 ) \, 
      ( 1 - \alpha - \rho )^{ 2 + 2 \cfour }
    }
  \right]
  \smallint_0^T
    \left|  
      \E\!\left[
        e^{ p \int_s^T \phi( \tilde{\mathcal{O}}^n_{ \lf u \rf_{ h_n } } ) \, du }
        \max\!\big\{
          1 , | \Phi( \tilde{\mathcal{O}}^n_{ \lf s \rf_{ h_n } } ) |^{ \nicefrac{ p }{ 2 } }  
        \big\} \right]
      \right|^{ \frac{ 2 }{ p } }
    ds
  }
  .
\end{split}
\end{equation}
Next observe that the assumption that 
$
  \chi \in 
  ( 
    0
    ,  
    \nicefrac{ 
      ( 1 - \alpha - \rho ) 
    }{
      ( 1 + 2 \cfour ) 
    }
  ]
  \cap 
  (
    0
    ,
    \nicefrac{
      ( \varrho - \rho ) 
    }{
      ( 1 + \cfour) 
    }
  ] 
$
ensures that
\begin{equation}
  ( 1 - \alpha - \rho ) 
  -
  ( 1 + 2 \cfour ) 
  \chi
  \geq 0
\qquad
  \text{and}
\qquad  
  ( \varrho - \rho ) 
  -
  ( 1 + \cfour) 
  \chi
  \geq 0
  .
\end{equation}
This, in turn, proves that
\begin{equation}
  \min\!\big\{ 
      \varrho 
    ,
      1 - \alpha - \vartheta \chi
  \big\}
  -
  \rho 
  -
  ( 1 + \cfour ) 
  \chi
  =
  \min\!\big\{ 
    ( \varrho - \rho ) 
    -
    ( 1 + \cfour) 
    \chi
  ,
    ( 1 - \alpha - \rho ) 
    -
    ( 1 + 2 \cfour ) 
    \chi
  \big\}
  \geq 0
  .
\end{equation}
Hence, we obtain for all $ n \in \N $ 
with $ h_n \leq 1 $ that
$
      | h_n |^{ 
        2 [ \min\{ \varrho, 1 - \alpha - \cfour \chi \} 
        - \rho 
        - ( 1 + \cfour ) \chi 
        ]
      }  
      \leq 1
$.
Combining this with \eqref{eq:cor_application}
proves that
for all $ n \in \N $
with $ h_n \leq 1 $ it holds that
\begin{equation}
\begin{split}
& 
  \big\| 
    \sup\nolimits_{ t \in [0,T] } 
    \| \tilde{\mathcal{X}}^n_t - \tilde{\mathcal{O}}^n_t \|_H 
  \big\|_{ \mathcal{L}^p( \P ; \R ) } 
\\
& 
  \leq
  \sqrt{
  2
  \left[
    1
    +
    \tfrac{  
      \tilde{\cthree} 
      \,
      ( 1 + | \tilde{\cthree} |^{ 1 / 2 } )^{ 2 + 2 \cfour } 
    }{
      ( 1 / 2 - \ctwo / 2 ) \, 
      ( 1 - \alpha - \rho )^{ 2 + 2 \cfour }
    }
  \right]
  \int_0^T
    \left|  
      \E\!\left[
        e^{ p \int_s^T \phi( \tilde{\mathcal{O}}^n_{ \lf u \rf_{ h_n } } ) \, du }
        \max\!\big\{
          1 , | \Phi( \tilde{\mathcal{O}}^n_{ \lf s \rf_{ h_n } } ) |^{ \nicefrac{ p }{ 2 } }  
        \big\} \right]
      \right|^{ 2 / p }
    ds
  }
  .
\end{split}
\end{equation}
The fact that $\forall \, \omega \in \tilde{\Omega},$ $t\in [0,T]$, $n \in \N \colon  \tilde{\mathcal{X}}^n_t(\omega) = \mathbb{X}^n_t(\omega), \tilde{\mathcal{O}}^n_t(\omega)  =   \mathcal{O}_t^n(\omega)$, (i), and (iii) hence yield that for all $ n \in \N $
with $ h_n \leq 1 $ it holds that 
\begin{equation}
\begin{split}
& 
\left\| 
\sup\nolimits_{ t \in [0,T] } 
\| \mathbb{X}^n_t - \mathcal{O}^n_t \|_H 
\right\|_{ \mathcal{L}^p( \P ; \R ) } 
\\
& 
\leq
\sqrt{
	2
	\left[
	1
	+
	\tfrac{  
		\tilde{\cthree} 
		\,
		( 1 + | \tilde{\cthree} |^{ 1 / 2 } )^{ 2 + 2 \cfour } 
	}{
	( 1 / 2 - \ctwo / 2 ) \, 
	( 1 - \alpha - \rho )^{ 2 + 2 \cfour }
}
\right]
\int_0^T
\left|  
\E\!\left[
e^{ p \int_s^T \phi( \mathcal{O}^n_{ \lf u \rf_{ h_n } } ) \, du }
\max\!\big\{
1 , | \Phi( \mathcal{O}^n_{ \lf s \rf_{ h_n } } ) |^{ \nicefrac{ p }{ 2 } }  
\big\} \right]
\right|^{ 2 / p }
ds
} \\
& \leq
\sqrt{
	2
	\left[
	1
	+
	\tfrac{  
		\tilde{\cthree} 
		\,
		( 1 + | \tilde{\cthree} |^{ 1 / 2 } )^{ 2 + 2 \cfour } 
	}{
	( 1 / 2 - \ctwo / 2 ) \, 
	( 1 - \alpha - \rho )^{ 2 + 2 \cfour }
}
\right] \left( T + 
\E\!\left[ \int_0^T
e^{ p \int_s^T \phi( \mathcal{O}^n_{ \lf u \rf_{ h_n } } ) \, du }
\max\!\big\{
1 , | \Phi( \mathcal{O}^n_{ \lf s \rf_{ h_n } } ) |^{ \nicefrac{ p }{ 2 } }  
\big\} \, ds \right]
 \right) }
.
\end{split}
\end{equation}
Combining this with the assumption that $ \limsup_{ n \to \infty} h_n = 0 $ implies that
\begin{equation}
\label{eq:a_priori_difference}
  \limsup\nolimits_{ n \to \infty }
  \big\| 
    \sup\nolimits_{ t \in [0,T] } 
    \| \mathbb{X}^n_t - \mathcal{O}^n_t \|_H 
  \big\|_{ \mathcal{L}^p( \P ; \R ) } 
  < \infty
  .
\end{equation}
This, 
the assumption that 
$
  \limsup_{ n \to \infty } 
  \sup_{ s \in [0,T] } 
  \E\big[
    \| \mathcal{O}_s^n \|_H^p 
  \big] < \infty
$,
and the triangle inequality
assure that
\begin{equation}
\label{eq:a_priori_X} 
  \limsup\nolimits_{ n \to \infty }
  \sup\nolimits_{ t \in [0,T] } 
  \E\big[ 
    \| \mathbb{X}^n_t \|_H^p
  \big]
  < \infty
  .
\end{equation}
Next note that (ii) and the fact $H_{\varrho} \subseteq H$ continuously ensure that for all $\omega \in \tilde{\Omega}$ it holds that $\limsup_{n \to \infty} \sup_{t \in [0,T]} \|X_t(\omega) - \mathbb{X}_t^n(\omega)\|_H =0$. Combining this with (i) allows us to apply Corollary~\ref{cor:as:outer} to obtain that for all $\varepsilon \in (0, \infty)$ it holds that $\limsup_{n \to \infty} \sup_{t \in [0, T]} \P(\|X_t - \mathbb{X}_t^n\|_H \geq \varepsilon)=0$. Proposition~\ref{prop:vitali} together with \eqref{eq:a_priori_X} hence ensures that for all $q \in (0, p)$ it holds that 
$
  \sup_{ t \in [0,T] }
  \E\big[ 
    \| X_t \|^p_H
  \big]
  < \infty
$
and
$
  \limsup_{ n \to \infty } 
  \sup_{ t \in [0,T] } 
  \E\big[
    \| X_t - \mathbb{X}_t^n \|_H^q 
  \big] 
  = 0
$.
Combining this with 
\eqref{eq:a_priori_difference},
\eqref{eq:a_priori_X}, and the assumption that $ \forall \, n \in \N$, $t \in [0,T] \colon \P(\mathcal{X}_t^n= \mathbb{X}_t^n)=1$ establishes (iv) and (v). The proof of Theorem~\ref{theorem:main} is thus completed.
\end{proof}

\section{Stochastic Kuramoto-Sivashinsky equations}
\label{sec:KS}

In this section we establish a few elementary results which, in particular, demonstrate that Theorem~\ref{theorem:main} can be applied to the stochastic K-S equation~\eqref{eq:intro_SPDE}.

\subsection{Setting}\label{setting:kura}
Let $(H, \langle \cdot , \cdot \rangle_H, \left\| \cdot \right\|_H) = (L^2(\lambda_{(0,1)}; \R), \langle \cdot , \cdot \rangle_{L^2(\lambda_{(0,1)}; \R)}, \left\| \cdot \right\|_{L^2(\lambda_{(0,1)}; \R)} )$, $ \beta \in (\nicefrac{1}{8}, \nicefrac{1}{2})$,  $T, \vt \in (0, \infty)$, $\kappa \in \R$, $\varrho \in (\nicefrac{1}{16}, \nicefrac{\beta}{2}) $, $ \chi \in    (0, \nicefrac{\varrho}{2}- \nicefrac{1}{32}]$, $\xi \in H_{1/4} $,  $(e_k)_{k \in \mathbb{Z}} \in \mathbb{M}(\mathbb{Z}, H)$, $(\lambda_k)_{k \in \mathbb{Z}}$, $(b_k)_{k \in \mathbb{Z}}$, $(\tilde{b}_k)_{k \in \mathbb{Z}} \in \mathbb{M}(\mathbb{Z}, \R)$ satisfy for all  $n \in \N$, $k \in \mathbb{Z}$ that $e_0= [(1)_{x \in (0,1)}]_{\lambda_{(0,1)}, \mathcal{B}(\R)}$, $ e_n = [(\sqrt{2}\cos(2 n \pi x))_{x \in (0,1)}]_{\lambda_{(0,1)} , \mathcal{B}(\R)}$, $ e_{-n} = [(\sqrt{2}\sin(2 n \pi x))_{x \in (0,1)}]_{\lambda_{(0,1)} , \mathcal{B}(\R)}$, $ \lambda_k = 16 k^4 \pi^4 - 4k^2 \pi^2 +\vt $, and $\sum_{m \in \mathbb{Z}} (|b_{m}|^2 + |\tilde{b}_{m}|^2 ) \, |m|^{4\beta -4}< \infty$,   let $ A \colon D(A) \subseteq H \to H $ be the linear operator such that  $ D(A) = \{ v \in H \colon \sum_{k \in \mathbb{Z}} |\lambda_k \langle e_k , v \rangle_H |^2 < \infty \}$ and such that for all  $ v \in D(A)$ it holds that $A v= \sum_{k \in \mathbb{Z}} - \lambda_k \langle e_k , v \rangle_H e_k$,  let $ ( H_r, \left< \cdot , \cdot \right>_{ H_r}, \left\| \cdot \right\|_{ H_r } ) $, $ r \in \R $, be a family of interpolation spaces associated to $ - A $ (see, e.g., Definition~3.5.25 in  \cite{Jentzen2015SPDElecturenotes}), let $(\Omega, \mathcal{F}, \P)$ be a probability space with a normal filtration $(\mathcal{F}_{t})_{t \in [0,T]}$,  let  $F \in \mathbb{M}(H_{1/16}, H_{-1/4})$, $B \in \mathbb{M}(H, H_{-1})$,  $(h_n)_{n \in \N} \in \mathbb{M}(\N, (0, \infty))$, $(P_n)_{n \in \N} \in \mathbb{M}(\N, L(H))$ satisfy for all $v \in H_{1/16}$, $n \in \N$ that $F(v)= \vt v - \frac{\kappa}{2}(v^2)'$, $B  v= \sum_{k \in \mathbb{Z}} (b_k \langle e_k, v \rangle_H + \tilde{b}_k \langle e_{-k}, v \rangle_H) \, e_k$, $\limsup_{k \to \infty} h_k=0$, $P_n(v)=\sum_{k =-n}^{n} \langle e_k, v \rangle_H e_k$, let $(\und{\cdot}) \colon\ \{ [v]_{\lambda_{(0,1)}, \mathcal{B}(\R)} \in L^0(\lambda_{(0,1)}; \R) \colon v \in \mathcal{C}( (0,1), \R ) \} \to \mathcal{C}((0,1), \R)$ be the function with the property that for all $v \in \mathcal{C}( (0,1), \R ) $ it holds that $\und{[v]_{\lambda_{(0,1)}, \mathcal{B}(\R)}}=v$, let $(W_t)_{t \in [0, T]}$ be an $\Id_H$-cylindrical $(\mathcal{F}_{t})_{t \in [0,T]}$-Wiener process, let $\Y^n \colon [0, T] \times \Omega \to P_n(H)$, $n \in \N$, and $ \mathcal{O}^n \colon [0, T] \times \Omega \to P_n(H)$, $ n \in \N$, be  stochastic processes, and assume  for all $n \in \N$, $t \in [0,T]$  that $\left[\mathcal{O}_t^n \right]_{\P, \mathcal{B}(H)} = \int_0^t P_n \, e^{(t-s)A}  \,B \, dW_s$ and 
\begin{equation}\label{eq:kura_X}
\P \Big( \Y_t^n = P_n \, e^{ t A } \, \xi + \smallint_0^t P_n \,  e^{  ( t - s ) A } \, \mathbbm{1}_{ \{ \| \Y_{ \lf s \rf_{h_n} }^n \|_{ H_{\varrho} } + \| \mathcal{O}_{ \lf s \rf_{h_n} }^n +P_n  e^{ \lf s \rf_{ h_n } A } \xi \|_{ H_{\varrho} } \leq | h_n|^{ - \chi } \}} \, F \big(  \Y_{ \lf s \rf_{ h_n } }^n \big) \, ds + \mathcal{O}_t^n \Big)=1. 
\end{equation}

\subsection{Properties of the nonlinearity}

In Lemma~\ref{coer:kura} and Lemma~\ref{nonlin:kura} below we demonstrate that the function $F$ in Section~\ref{setting:kura} fulfills the hypotheses of Theorem~\ref{theorem:main} above. Our proofs of Lemma~\ref{coer:kura} and Lemma~\ref{nonlin:kura} use the following well-known lemma.

\begin{lemma}\label{kura_norms}
Assume the setting in Section~\ref{setting:kura} and let $v \in H_{1}$. Then $\|v'\|_H \leq 2^{\nicefrac{1}{4}}\|v\|_{H_{1/4}}$ and $ \|v\|_{H_{1/4}}^2 \leq \|v\|_H \|v\|_{H_{1/2}}$.
\end{lemma}	
\begin{proof}[Proof of Lemma~\ref{kura_norms}.]
Note that Parseval's identity and integration by parts prove that 
\begin{align}\label{kura_norms_1}
\begin{split}
&\|v'\|_H^2  = \sum_{k \in \mathbb{Z}} \left|\langle  e_k, v'  \rangle _H\right|^2 = \sum_{k \in \mathbb{Z}} \left|\int_0^1 \und{e_k}(x) \,  (\und{v})'(x) \,  dx \right|^2 = \sum_{k \in \mathbb{Z}} \left|\int_0^1  (\und{e_k})'(x) \,  \und{v}(x)  \, dx \right|^2 \\
&= \sum_{k \in \mathbb{Z}} \left|2 k \pi \int_0^1  \und{e_{-k}}(x) \, \und{v}(x) \, dx \right|^2  = \sum_{k \in \mathbb{Z}} 4 k^2 \pi^2 \left| \langle  e_{-k}, v \rangle_H \right|^2 = \sum_{k \in \mathbb{Z}} 4 k^2 \pi^2  \left| \langle   e_{k}, v \rangle_H \right|^2 \\
&\leq  \sum_{k \in \mathbb{Z}} \sqrt{2} \left(16 k^4 \pi^4 - 4k^2 \pi^2 +\vt\right)^{\nicefrac{1}{2}}  \left| \langle   e_{k}, v \rangle_H \right|^2 = \sqrt{2} \sum_{k \in \mathbb{Z}} |\lambda_k|^{\nicefrac{1}{2}}  \left| \langle   e_{k}, v \rangle_H \right|^2 = \sqrt{2} \, \|v\|_{H_{1/4}}^2.
\end{split}
\end{align}
Moreover, H\"older's inequality shows that
\begin{align}
\begin{split}
\|v\|_{H_{1/4}}^2 = \sum_{k \in \mathbb{Z}} |\lambda_k|^{\nicefrac{1}{2}}  \left| \langle  e_{k}, v \rangle_H \right|^2  \leq \sqrt{\smallsum_{k \in \mathbb{Z}}   \left| \langle e_{k}, v \rangle_H \right|^2} \sqrt{\smallsum_{k \in \mathbb{Z}} |\lambda_k|  \left| \langle  e_{k}, v \rangle_H \right|^2} = \|v\|_H \|v\|_{H_{1/ 2}} .
\end{split}
\end{align}
Combining this and \eqref{kura_norms_1} completes the proof of Lemma~\ref{kura_norms}.
\end{proof}

The next simple lemma is a slight modification of Lemma~5.7 in Bl\"omker \& Jentzen~\cite{BloemkerJentzen2013}.
\begin{lemma}
\label{coer:kura}
Assume the setting in Section~\ref{setting:kura} and let $v, w \in H_1$, $\varepsilon \in (0, \infty)$. Then
\begin{align}
\left< v, F( v + w ) \right>_H \leq \| v \|_H^2  \left(  \tfrac{3\vt}{2} + \tfrac{\kappa^4}{16} [1 + \nicefrac{1}{\varepsilon}]^2 + \varepsilon \! \left[\sup\nolimits_{x \in (0,1)} |\und{w}(x)|^2 \right]\right) + \tfrac{1}{2} \|v\|_{H_{1/2}}^2 + \tfrac{\vt}{2} \|w \|_H^2 +  \tfrac{1}{4} \| w^2 \|_H^2.
\end{align}
\begin{proof}[Proof of Lemma~\ref{coer:kura}.]
Note that integration by parts yields that
\begin{align}
\begin{split}
\left< v, F( v + w ) \right>_H   &= \left< v, \vt (v+w) - \tfrac{\kappa}{2}((v+w)^2)' \right>_H \\
& = \vt \|v\|_H^2 + \vt \! \left< v, w \right>_H - \frac{\kappa}{2} \int_0^1 \und{v}(x) \big([\und{v}+\und{w}]^2\big)'(x) \, dx \\
& = \vt \|v\|_H^2 + \vt \! \left< v, w \right>_H +\frac{\kappa}{2} \int_0^1 (\und{v})'(x) \, [\und{v}(x)+ \und{w}(x)]^2 \, dx \\
&= \vt \|v\|_H^2 + \vt \! \left< v, w \right>_H + \kappa \int_0^1 (\und{v})'(x) \, \und{v}(x) \, \und{w}(x) \, dx + \frac{\kappa}{2} \int_0^1 (\und{v})'(x) \, |\und{w}(x)|^2 \, dx.
\end{split}
\end{align}
The Cauchy-Schwartz inequality and the fact that $\forall \, x, y, r \in (0, \infty) \colon xy \leq \frac{x^2}{2r} + \frac{r y^2}{2}$ therefore prove that
\begin{align}\label{kur_coer}
\begin{split}
\left< v, F( v + w ) \right>_H &\leq \vt \|v\|_H^2 + \vt \|v\|_H \|w\|_H+  |\kappa| \|v'\|_H \| v \, w \|_H + \tfrac{|\kappa|}{2} \|v'\|_H \|w^2\|_H \\
& \leq \vt \|v\|_H^2 + \tfrac{\vt}{2} \|v \|_H^2 + \tfrac{\vt}{2} \|w \|_H^2 +  \tfrac{\kappa^2}{4\varepsilon} \|v'\|_H^2 + \varepsilon \| v \, w \|_H^2 + \tfrac{\kappa^2}{4}\|v'\|_H^2 + \tfrac{1}{4} \| w^2 \|_H^2 \\
& \leq \tfrac{3\vt}{2} \|v \|_H^2 + \tfrac{\vt}{2} \|w \|_H^2+  \big[\tfrac{\kappa^2}{4 \varepsilon} + \tfrac{\kappa^2}{4}\big]\|v'\|_H^2 + \varepsilon \| v \|_H^2 \! \left[\sup\nolimits_{x \in (0,1)} |\und{w}(x)|^2 \right] + \tfrac{1}{4} \| w^2 \|_H^2. 
\end{split}
\end{align}
Lemma~\ref{kura_norms} and again the fact that $\forall \, x, y, r \in (0, \infty) \colon xy \leq \frac{x^2}{2r} + \frac{r y^2}{2}$ hence show that 
\begin{align}
\begin{split}
&\left< v, F( v + w ) \right>_H   \\
& \leq \tfrac{3\vt}{2} \|v \|_H^2 + \tfrac{\vt}{2} \|w \|_H^2+ \sqrt{2} \kappa^2\big[\tfrac{1}{4 \varepsilon} + \tfrac{1}{4}\big] \|v\|_H \|v\|_{H_{1/2}}+ \varepsilon \| v \|_H^2 \!  \left[\sup\nolimits_{x \in (0,1)} |\und{w}(x)|^2 \right] + \tfrac{1}{4} \| w^2 \|_H^2 \\
& \leq  \tfrac{3\vt}{2} \|v \|_H^2 + \tfrac{\vt}{2} \|w \|_H^2+   \kappa^4 \big[\tfrac{1}{4 \varepsilon} + \tfrac{1}{4}\big]^2 \|v\|_H^2 + \tfrac{1}{2} \|v\|_{H_{1/2}}^2+ \varepsilon \| v \|_H^2 \! \left[\sup\nolimits_{x \in (0,1)} |\und{w}(x)|^2 \right] + \tfrac{1}{4} \| w^2 \|_H^2 \\
& = \| v \|_H^2  \left(  \tfrac{3\vt}{2} + \tfrac{\kappa^4}{16} [1 + \nicefrac{1}{\varepsilon}]^2 + \varepsilon  \! \left[\sup\nolimits_{x \in (0,1)} |\und{w}(x)|^2 \right] \right) + \tfrac{1}{2} \|v\|_{H_{1/2}}^2 + \tfrac{\vt}{2} \|w \|_H^2 +  \tfrac{1}{4} \| w^2 \|_H^2 . 
\end{split}
\end{align}
The proof of Lemma~\ref{coer:kura} is thus completed.
\end{proof}

\end{lemma}

\begin{lemma}\label{nonlin:kura}
Assume the setting in Section~\ref{setting:kura} and let $v, w \in H_{\nicefrac{1}{16}}$. Then
\begin{align}
\begin{split}
 &\|F(v) - F(w) \|_{H_{-1 /4}} \\
 & \leq  \left( \vt^{\nicefrac{11}{16}} + |\kappa| \left[ \sup_{u \in  H_{ 1/16 }\backslash \{0\}} \frac{\|u\|^2_{L^4(\lambda_{(0,1)}; \R)}}{\|u\|^2_{H_{1/16}}} \right] \right) \big( 1+ \left\| v \right\|_{ H_{1/16 }} +  \left\| w \right\|_{ H_{1/16 }} \big) \left\| v - w \right\|_{ H_{1/16}} < \infty.
\end{split}
\end{align}	
\end{lemma}
\begin{proof}[Proof of Lemma~\ref{nonlin:kura}.]
First, note that
\begin{align}\label{v-w:kura}
\begin{split}
\|v-w \|_{H_{-1 /4}} &\leq \|(-A)^{\nicefrac{-5}{16}} \|_{L(H)} \|v-w \|_{H_{1 /16}} = \vt^{\nicefrac{-5}{16}} \|v-w \|_{H_{1 /16}} \\
& \leq \vt^{\nicefrac{-5}{16}} \big( 1+ \left\| v \right\|_{ H_{1/16 }} +  \left\| w \right\|_{ H_{1/16 }} \big) \|v-w \|_{H_{1 /16}}.
\end{split}
\end{align}	
Next observe that for all $u \in H$ it holds that 
\begin{align}
(-A)^{-\nicefrac{1}{4}}(u')= \big((-A)^{-\nicefrac{1}{4}}u\big)'.
\end{align}
This and Lemma~\ref{kura_norms} prove that
\begin{align}
\begin{split}
&\|(v^2)' - (w^2)' \|_{H_{-1 /4}} = \|(-A)^{\nicefrac{-1}{4}}((v^2)' - (w^2)') \|_{H} = \|((-A)^{\nicefrac{-1}{4}} [v^2 - w^2])' \|_{H} \\
& \leq 2^{\nicefrac{1}{4}} \|(-A)^{\nicefrac{-1}{4}} [v^2 - w^2] \|_{H_{1/4}} 
 = 2^{\nicefrac{1}{4}} \| v^2 - w^2 \|_H  \leq 2^{\nicefrac{1}{4}} \|v+w \|_{L^4(\lambda_{(0,1)}; \R)} \|v -w \|_{L^4(\lambda_{(0,1)}; \R)} \\
& \leq 2^{\nicefrac{1}{4}} \left[\sup_{u \in  H_{ 1/16 } \backslash \{0\}} \frac{\|u\|_{L^4(\lambda_{(0,1)}; \R)}}{\|u\|_{H_{1/16}}} \right]^2 \! \! \left\| v + w \right\|_{ H_{1/16 }} \left\| v - w \right\|_{ H_{1/16 }} \\
& \leq 2 \left[\sup_{u \in  H_{ 1/16 }\backslash \{0\}} \frac{\|u\|^2_{L^4(\lambda_{(0,1)}; \R)}}{\|u\|^2_{H_{1/16}}} \right] \! \big( 1+ \left\| v \right\|_{ H_{1/16 }} +  \left\| w \right\|_{ H_{1/16 }} \big) \left\| v - w \right\|_{ H_{1/16 }}.
\end{split}
\end{align}
This together with \eqref{v-w:kura} shows that
\begin{align}\label{v-w:kura:2}
\begin{split}
& \|F(v) - F(w) \|_{H_{-1 /4}}  = \left\| \vt (v-w) - \tfrac{\kappa}{2}\left((v^2)' - (w^2)'\right) \right\|_{H_{-1 /4}} \\
& \leq \vt \|v-w \|_{H_{-1 /4}} + \tfrac{|\kappa|}{2} \|(v^2)' - (w^2)' \|_{H_{-1 /4}} \\
& \leq \left( \vt^{\nicefrac{11}{16}} + |\kappa| \left[\sup_{u \in  H_{ 1/16 }\backslash \{0\}} \frac{\|u\|^2_{L^4(\lambda_{(0,1)}; \R)}}{\|u\|^2_{H_{1/16}}} \right]\right) \big( 1+ \left\| v \right\|_{ H_{1/16 }} +  \left\| w \right\|_{ H_{1/16 }} \big) \left\| v - w \right\|_{ H_{1/16}}.
\end{split}
\end{align}
Next observe that the Sobolev embedding theorem ensures that
\begin{align}
\sup_{u \in  H_{ 1/16 }\backslash \{0\}} \frac{\|u\|^2_{L^4(\lambda_{(0,1)}; \R)}}{\|u\|^2_{H_{1/16}}} < \infty.
\end{align}
Combining this with \eqref{v-w:kura:2} completes the proof of Lemma~\ref{nonlin:kura}.
\end{proof}


\subsection{Fernique's theorem}

\begin{lemma}\label{lemma:fernique}
Let $(V, \left\| \cdot \right\|_V)$ be a separable $\R$-Banach space, let $(\Omega, \mathcal{F}, \P)$ be a probability space, let $\mathcal{X} \colon \Omega \to V$ be a mapping which satisfies that for every $\varphi \in V'$ it holds that $\varphi \circ \mathcal{X} \colon \Omega \to \R $ is a
centered Gaussian random variable, and let $r \in (0, \infty)$ satisfy that $\P(\|\mathcal{X}\|_V^2 > r) \leq \nicefrac{1}{10}$. Then
\begin{align}\label{eq:fernique}
\begin{split}
\E \! \left[ \exp\! \left( \frac{\|\mathcal{X}\|_V^2}{18 r}\right)\right] \leq \sqrt{e} + \sum_{k=0}^{\infty} \bigg[\frac{e}{3}\bigg]^{2^k} < 13.
\end{split}
\end{align}
\end{lemma}
\begin{proof}[Proof of Lemma \ref{lemma:fernique}.]
Note that \eqref{eq:fernique} is an immediate consequence of the fact that $\P(\|\mathcal{X}\|_V^2 \leq r) \geq \nicefrac{9}{10}$ and of Fernique's theorem (see, e.g., Theorem~8.2.1 in Stroock~\cite{Stroock2010}). The proof of Lemma~\ref{lemma:fernique} is thus completed.
\end{proof}

\subsection{Properties of the stochastic convolution process}

\begin{lemma}\label{lem:o_expansion}
Assume the setting in Section~\ref{setting:kura} and let $(a_k)_{k \in \Z} \in \mathbb{M}(\Z, \R)$, $S \in \mathcal{P}_0(\Z)$. Then
\begin{align}
\E \! \left[ \left| \smallsum\limits_{k \in S} a_k \int_0^t e^{-\lambda_k(t-s)} \, \big( b_k \, d \! \left<  e_k, W_s \right>_H + \tilde{b}_k \, d \! \left<  e_{-k}, W_s\right>_H \big) \right|^2 \right] \leq \sum_{k \in S }  \frac{|a_k \, b_k|^2 + |a_{k} \, \tilde{b}_{k} |^2}{ \lambda_k}.
\end{align}
\end{lemma}
\begin{proof}[Proof of Lemma~\ref{lem:o_expansion}.]
Throughout this proof let $\mathbb{S} \in  \mathcal{P}_0(\Z)$ be the set given by $\mathbb{S}= \{-k  \colon  k \in S\}$. Next note that It\^o's isometry proves that
\begin{align}
\begin{split}
&\E \! \left[ \left| \smallsum\limits_{k \in S} a_k \int_0^t e^{-\lambda_k(t-s)} \, \big( b_k \, d \! \left<  e_k, W_s \right>_H + \tilde{b}_k \, d \! \left<  e_{-k}, W_s\right>_H \big) \right|^2 \right]\\
& = \E \! \left[ \left| \smallsum\limits_{k \in S} a_k \,b_k \int_0^t e^{-\lambda_k(t-s)}  \, d \! \left<  e_k, W_s \right>_H  + \smallsum\limits_{k \in S} a_k \, \tilde{b}_k  \int_0^t e^{-\lambda_k(t-s)}  \, d \! \left<  e_{-k}, W_s\right>_H \right|^2 \right] \\
& = \E \! \left[ \left| \smallsum\limits_{k \in S} a_k \,b_k  \int_0^t e^{-\lambda_k(t-s)}  \, d \! \left<  e_k, W_s \right>_H  + \smallsum\limits_{k \in \mathbb{S}} a_{-k} \, \tilde{b}_{-k}  \int_0^t e^{-\lambda_{-k}(t-s)}  \, d \! \left<  e_{k}, W_s\right>_H \right|^2 \right] \\
& = \E \Bigg[ \bigg| \smallsum\limits_{k \in S \cap \mathbb{S}} ( a_k \, b_k +  a_{-k} \, \tilde{b}_{-k})  \int_0^t e^{-\lambda_k(t-s)}  \, d \! \left<  e_k, W_s \right>_H + \smallsum\limits_{k \in S \backslash \mathbb{S}} a_k \, b_k \int_0^t e^{-\lambda_k(t-s)}  \, d \! \left<  e_k, W_s \right>_H\\
&+ \smallsum\limits_{k \in \mathbb{S} \backslash S} a_{-k} \, \tilde{b}_{-k} \int_0^t e^{-\lambda_{k}(t-s)}  \, d \! \left<  e_{k}, W_s\right>_H  \bigg|^2 \Bigg]  = \displaystyle\sum_{k \in S \cap \mathbb{S}} |a_k \, b_k +  a_{-k} \, \tilde{b}_{-k}|^2 \int_0^t e^{-2\lambda_k(t-s)} \, ds\\
& + \sum_{k \in S \backslash \mathbb{S}} |a_k \, b_k|^2 \int_0^t e^{-2\lambda_k(t-s)} \, ds  + \sum_{k \in \mathbb{S} \backslash S} |a_{-k} \, \tilde{b}_{-k} |^2 \int_0^t e^{-2\lambda_k(t-s)} \, ds.
\end{split}
\end{align}
The fact that $\forall \, x, y \in \R \colon |x+y|^2 \leq 2 x^2 + 2 y^2$ hence ensures that
\begin{align}
\begin{split}
&\E \! \left[ \left| \smallsum\limits_{k \in S} a_k \int_0^t e^{-\lambda_k(t-s)} \, \big( b_k \, d \! \left<  e_k, W_s \right>_H + \tilde{b}_k \, d \! \left<  e_{-k}, W_s\right>_H \big) \right|^2 \right]\\
& \leq \sum_{k \in S \cap \mathbb{S}}  \frac{|a_k \, b_k +  a_{-k} \, \tilde{b}_{-k}|^2}{2 \lambda_k}  + \sum_{k \in S \backslash \mathbb{S}} \frac{|a_k \, b_k|^2}{2 \lambda_k} + \sum_{k \in \mathbb{S} \backslash S} \frac{|a_{-k} \, \tilde{b}_{-k} |^2}{2 \lambda_k} \\
& \leq \sum_{k \in S \cap \mathbb{S}}  \frac{|a_k \, b_k|^2 +  |a_{-k} \, \tilde{b}_{-k}|^2}{ \lambda_k}  + \sum_{k \in S \backslash \mathbb{S}} \frac{|a_k \, b_k|^2}{\lambda_k} + \sum_{k \in \mathbb{S} \backslash S } \frac{|a_{-k} \, \tilde{b}_{-k} |^2}{ \lambda_k}.
\end{split}
\end{align}
This yields that
\begin{align}
\begin{split}
&\E \! \left[ \left| \smallsum\limits_{k \in S} a_k \int_0^t e^{-\lambda_k(t-s)} \, \big( b_k \, d \! \left<  e_k, W_s \right>_H + \tilde{b}_k \, d \! \left<  e_{-k}, W_s\right>_H \big) \right|^2 \right]\\ 
& \leq \sum_{k \in S \cap \mathbb{S}}  \frac{|a_k \, b_k|^2}{ \lambda_k} +  \sum_{k \in S \cap \mathbb{S}}  \frac{|a_{-k} \, \tilde{b}_{-k}|^2}{ \lambda_k}  + \sum_{k \in S \backslash \mathbb{S}} \frac{|a_k \, b_k|^2}{\lambda_k} + \sum_{k \in \mathbb{S} \backslash S } \frac{|a_{-k} \, \tilde{b}_{-k} |^2}{ \lambda_k}\\
& = \sum_{k \in S \cap \mathbb{S}}  \frac{|a_k \, b_k|^2}{ \lambda_k} +  \sum_{k \in S \cap \mathbb{S}}  \frac{|a_{k} \, \tilde{b}_{k}|^2}{ \lambda_k}  + \sum_{k \in S \backslash \mathbb{S}} \frac{|a_k \, b_k|^2}{\lambda_k} + \sum_{k \in S \backslash \mathbb{S} } \frac{|a_{k} \, \tilde{b}_{k} |^2}{ \lambda_k} = \sum_{k \in S} \frac{|a_k \, b_k|^2 + |a_{k} \, \tilde{b}_{k}|^2}{\lambda_k}.\\
\end{split}
\end{align}
The proof of Lemma~\ref{lem:o_expansion} is thus completed.
\end{proof}

\begin{lemma}\label{kura:phi(w)_finite}
Assume the setting in Section~\ref{setting:kura}, let $p \in (\nicefrac{1}{\beta}, \infty)$, $t \in [0,T]$, $n \in \N$, and let $Y \colon \Omega \to \R$ be a standard normal random variable. Then 
\begin{align}\label{kura:norm:eq}
\begin{split}
& \left(\E \! \left[\sup\nolimits_{x \in (0,1)} | \und{\mathcal{O}_t^n}(x) |^2 \right] \right)^{\!\nicefrac{1}{2}} \leq  \sqrt{10} \,\big( \E \big[|Y|^p \big] \big)^{\nicefrac{1}{p}} \left[ \sum_{k=-n}^{n} \frac{ \max \{|k \pi|^{4 \beta}, 1 \} \big(|b_k|^2 +   | \tilde{b}_{k} |^2 \big)}{\lambda_k}\right]^{\nicefrac{1}{2}} \\
& \quad \cdot \Big[\sup \! \Big(\Big\{ \! \sup\nolimits_{x \in (0,1)} |v(x)| \colon \big[v \in \mathcal{C}((0,1), \R) \text{ and } \|v\|_{\W^{\beta, p}((0,1), \mathbb{R})} \leq 1\big] \Big\}\Big)\Big]   < \infty.
\end{split}
\end{align}
\end{lemma}
\begin{proof}[Proof of Lemma~\ref{kura:phi(w)_finite}.] 
First, note that Jensen's inequality proves that
\begin{align}\label{kura:phi(w)_eq1}
\begin{split}
&\E \! \left[\sup\nolimits_{x \in (0,1)} | \und{\mathcal{O}_t^n}(x) |^2 \right] \\
&\leq \Big[\sup \! \Big(\Big\{ \! \sup\nolimits_{x \in (0,1)} |v(x)| \colon \big[v \in \mathcal{C}((0,1), \R) \text{ and } \|v\|_{\W^{\beta, p}((0,1), \mathbb{R})} \leq 1\big] \Big\}\Big)\Big]^2 \, \E \! \left[\big\| \und{\mathcal{O}_t^n} \big\|_{\W^{\beta, p}((0,1), \mathbb{R})}^2 \right] \\
& \leq \Big[\sup \!\Big(\Big\{ \! \sup\nolimits_{x \in (0,1)} |v(x)| \colon \big[v \in \mathcal{C}((0,1), \R) \text{ and } \|v\|_{\W^{\beta, p}((0,1), \mathbb{R})} \leq 1\big] \Big\}\Big)\Big]^2  \left( \E \! \left[\big\| \und{\mathcal{O}_t^n} \big\|_{\W^{\beta, p}((0,1), \mathbb{R})}^p \right]\right)^{\!\nicefrac{2}{p}}.
\end{split}
\end{align}
Moreover, observe that
\begin{align}\label{kura:phi(w)_all}
\begin{split}
&\E \! \left[\big\| \und{\mathcal{O}_t^n} \big\|_{\W^{\beta, p}((0,1), \mathbb{R})}^p \right] = \E \! \left[ \int_0^1 |\und{\mathcal{O}_t^n} (x)|^p \, dx + \int_0^1 \int_0^1 \frac{|\und{\mathcal{O}_t^n} (x)- \und{\mathcal{O}_t^n}(y)|^p}{|x-y|^{1+ \beta p}} \, dx \, dy\right]\\
& = \E \big[|Y|^p \big] \int_0^1 \left( \E \! \left[ |\und{\mathcal{O}_t^n} (x)|^2 \right] \right)^{\nicefrac{p}{2}} dx  + \E \big[|Y|^p \big] \int_0^1 \int_0^1 \frac{\left( \E \! \left[ |\und{\mathcal{O}_t^n} (x) - \und{\mathcal{O}_t^n}(y) |^2 \right] \right)^{\nicefrac{p}{2}}}{|x-y|^{1+ \beta p}}\,  dx \, dy.
\end{split}
\end{align}
Next note that Lemma~\ref{lem:o_expansion} ensures that for all $ x \in (0,1)$ it holds that
\begin{align}
\begin{split}
 & \E \! \left[ |\und{\mathcal{O}_t^n}(x)|^2 \right] = \E \! \left[ \left| \smallsum\limits_{k=-n}^{n} \und{e_k} (x) \int_0^t e^{-\lambda_k(t-s)} \, \big( b_k \, d \! \left<  e_k, W_s \right>_H + \tilde{b}_k \, d \! \left<  e_{-k}, W_s\right>_H \big) \right|^2 \right] \\
 & \leq \sum_{k=-n}^{n} \frac{|\und{e_k}(x)|^2 \,  \big[ |b_k|^2 +   |\tilde{b}_{k} |^2 \big]}{ \lambda_k} \leq 2  \sum_{k=-n}^{n} \frac{ |b_k|^2 +   |\tilde{b}_{k} |^2}{ \lambda_k}.
\end{split}
\end{align}
This yields that 
\begin{align}\label{kura:phi(w)_1}
\int_0^1 \left( \E \! \left[ |\und{\mathcal{O}_t^n}(x)|^2 \right] \right)^{\nicefrac{p}{2}} dx \leq 2^{\nicefrac{p}{2}} \left[\sum_{k=-n}^{n} \frac{|b_k|^2 +|\tilde{b}_k|^2}{\lambda_k}  \right]^{\nicefrac{p}{2}}.
\end{align}
Moreover, Lemma~\ref{lem:o_expansion}  proves that for all  $ x, y \in (0,1)$ it holds that
\begin{align}\label{eq:o_diff}
\begin{split}
 &\E \! \left[ |\und{\mathcal{O}_t^n} (x) - \und{\mathcal{O}_t^n}(y) |^2 \right] \\
 &=  \E \! \left[ \left| \smallsum\limits_{k=-n}^{n} \left[\und{e_k}(x) - \und{e_k}(y)\right] \int_0^t e^{-\lambda_k(t-s)} \, \big( b_k \, d \! \left<  e_k, W_s \right>_H + \tilde{b}_k \, d \! \left<  e_{-k}, W_s\right>_H \big) \right|^2 \right] \\
 &\leq  \sum_{k=-n}^{n} \frac{|\und{e_k}(x) - \und{e_k}(y)|^2 \, |b_k|^2 + |\und{e_k}(x) - \und{e_k}(y)|^2 \,  | \tilde{b}_{k} |^2 }{ \lambda_k}.
\end{split}
\end{align}
In addition, the assumption that $\beta < \nicefrac{1}{2}$ and  the fact that $\forall \, x, y \in \R \colon \max\{|\sin(x)-\sin(y)|, |\cos(x)- \cos(y)|\} \leq |x-y|$ ensure for all $x, y \in (0,1)$, $k \in \Z$ that
\begin{align}
\begin{split}
&|\und{e_k}(x) - \und{e_k}(y)|^2  \leq 2 \max\!\big\{|\sin(2k \pi x) - \sin (2k\pi y) |^2, |\cos(2k \pi x) - \cos (2k\pi y) |^2  \big\} \\
& \leq  2^{3-4\beta}\max\!\big\{|\sin(2k \pi x) - \sin (2k\pi y) |^{4\beta}, |\cos(2k \pi x) - \cos (2k\pi y) |^{4\beta}  \big\} \leq 2^3 \, |k \pi|^{4 \beta} \, |x-y|^{4 \beta} .
\end{split}
\end{align}
Combining this with \eqref{eq:o_diff} proves for all $x, y \in (0, 1)$ that
\begin{align}
\E \! \left[ |\und{\mathcal{O}_t^n} (x) - \und{\mathcal{O}_t^n}(y) |^2 \right] \leq 2^3 \, |x-y|^{4 \beta}  \sum_{k=-n}^{n} \frac{ |k \pi|^{4 \beta} \, \big(|b_k|^2 +   | \tilde{b}_{k} |^2 \big)}{ \lambda_k}.
\end{align}
This and the assumption that $\beta p > 1$ yield that
\begin{align}\label{kura:phi(w)_2}
\begin{split}
 &\int_0^1 \int_0^1 \frac{\left( \E \! \left[ |\und{\mathcal{O}_t^n} (x) - \und{\mathcal{O}_t^n}(y) |^2 \right] \right)^{\nicefrac{p}{2}}}{|x-y|^{1+ \beta p}} \,  dx \, dy \\
 &\leq 2^{\nicefrac{3p}{2}} \! \left[ \sum_{k=-n}^{n} \frac{|k \pi|^{4 \beta} \,  \big(|b_k|^2 + | \tilde{b}_{k} |^2 \big) }{\lambda_k}\right]^{\nicefrac{p}{2}} \int_0^1 \int_0^1 |x-y|^{\beta p -1} \, dx \, dy  \\
 & \leq 2^{\nicefrac{3p}{2}} \! \left[ \sum_{k=-n}^{n}\frac{|k \pi|^{4 \beta} \,  \big(|b_k|^2 +  | \tilde{b}_{k} |^2 \big)}{\lambda_k}\right]^{\nicefrac{p}{2}}.
\end{split}
\end{align}
Combining \eqref{kura:phi(w)_all}, \eqref{kura:phi(w)_1}, and the fact that $\forall \, x, y \in \R \colon  |x+y|^{\nicefrac{2}{p}} \leq |x|^{\nicefrac{2}{p}} + |y|^{\nicefrac{2}{p}}$ hence shows that
\begin{align}\label{kura:phi(w)_last}
\begin{split}
 &\left( \E \! \left[\big\| \und{\mathcal{O}_t^n} \big\|_{\W^{\beta,p}((0,1), \mathbb{R})}^p \right]\right)^{\nicefrac{1}{p}} \\
  & \leq \big( \E \big[|Y|^p \big] \big)^{\nicefrac{1}{p}} \left\{ 2^{\nicefrac{p}{2}} \! \left[ \sum_{k=-n}^{n} \frac{ |b_k|^2 +   | \tilde{b}_{k} |^2 }{\lambda_k}\right]^{\nicefrac{p}{2}} +  2^{\nicefrac{3p}{2}}\! \left[ \sum_{k=-n}^{n} \frac{|k \pi|^{4 \beta} \big(|b_k|^2 +  | \tilde{b}_{k} |^2 \big) }{\lambda_k}\right]^{\nicefrac{p}{2}} \right\}^{\!\nicefrac{1}{p}} \\
 & \leq \sqrt{10} \,\big( \E \big[|Y|^p \big] \big)^{\nicefrac{1}{p}} \left[ \sum_{k=-n}^{n} \frac{ \max \{|k \pi|^{4 \beta}, 1 \} \big(|b_k|^2 +   | \tilde{b}_{k} |^2 \big)}{\lambda_k}\right]^{\nicefrac{1}{2}}.
\end{split}
\end{align}
Next observe that  the Sobolev embedding theorem and the assumption that $\beta p > 1$ ensure that
\begin{align}
\sup \! \Big(\Big\{ \! \sup\nolimits_{x \in (0,1)} |v(x)| \colon \big[v \in \mathcal{C}((0,1), \R) \text{ and } \|v\|_{\W^{\beta, p}((0,1), \mathbb{R})} \leq 1\big] \Big\}\Big) < \infty.
\end{align}
Combining this with \eqref{kura:phi(w)_eq1} and \eqref{kura:phi(w)_last}  establishes \eqref{kura:norm:eq}. The proof of Lemma~\ref{kura:phi(w)_finite} is thus completed.
\end{proof}


\begin{lemma}\label{lem:exp}
Let $a\in (0, \infty)$, $x, r \in [0, \infty)$. Then $x^r \leq a^{-r} (\lf r \rf_1 +1)! \, e^{ax}$.
\end{lemma}
\begin{proof}[Proof of Lemma~\ref{lem:exp}.]
Note that
\begin{align}
e^{ax} \geq 1+ \frac{|ax|^{\lf r \rf_1 +1}}{(\lf r \rf_1 +1)!} \geq \frac{1 +|ax|^{\lf r \rf_1 +1}}{(\lf r \rf_1 +1)!} \geq \frac{|ax|^r}{(\lf r \rf_1 +1)!}.
\end{align}
The proof of Lemma~\ref{lem:exp} is thus completed. 
\end{proof}

\begin{cor}\label{kura:cor}

Assume the setting in Section~\ref{setting:kura} and let $ \phi , \Phi \in \mathcal{M}( \mathcal{B}(H_{1}) , \mathcal{B}([0,\infty)))$,  $p \in (\nicefrac{1}{\beta}, \infty)$,  $\varepsilon \in (0, \infty)$ satisfy for all $ v \in H_{1}$ that  $\phi(v) = \frac{3\vt}{2} + \tfrac{\kappa^4}{16} [1 + \nicefrac{1}{\varepsilon}]^2 + \varepsilon \!   \left[\sup\nolimits_{x \in (0,1)} |\und{v}(x)|^2 \right]$, $\Phi(v)= \frac{\vt}{2} \|v \|_H^2 +  \frac{1}{4} \| v^2 \|_H^2$, and
\begin{align}\label{kura:vareps}
\begin{split}
\varepsilon &\leq  \tfrac{1}{7200  p^3 T} \left[ \max \! \left\{ 1, \smallsum\limits_{k \in \mathbb{Z}} \tfrac{  \max \{|k \pi|^{4 \beta}, 1 \} (|b_k|^2 +  | \tilde{b}_{k} |^2 )}{\lambda_k} \right\} \right]^{-1} \\
& \quad  \cdot \Big[\sup \!\Big(\Big\{ \sup\nolimits_{x \in (0,1)} |v(x)| \colon \big[v \in \mathcal{C}((0,1), \R) \text{ and } \|v\|_{\W^{\beta, p}((0,1), \mathbb{R})} \leq 1\big] \Big\}\Big)\Big]^{-2}.
\end{split}
\end{align}
Then it holds that $\limsup_{m \to \infty}  \sup_{s \in [0,T]} \E \! \left[\|  \mathcal{O}_s^m  + P_m  e^{sA}  \xi  \|_H^p \right] < \infty$ and 
 \begin{align}
 \limsup_{m \to \infty}   \E \! \left[ \int_0^T \exp \!\left( \smallint_r^T p \, \phi\big(\mathcal{O}_{\lf u \rf_{h_m}}^m + P_m  e^{\lf u \rf_{h_m} A} \xi\big) \, du \right) \max \! \left\{1, \big| \Phi\big( \mathcal{O}_{ \lf r \rf_{h_m } }^m + P_m  e^{\lf r \rf_{h_m} A} \xi\big)\big|^{\nicefrac{p}{2}} \right\} dr \right] < \infty. 
 \end{align}
\end{cor}
\begin{proof}[Proof of Corollary~\ref{kura:cor}.]
First, note that Markov's inequality, e.g., Lemma~4.7 in~\cite{hj11}, Lemma~\ref{kura:phi(w)_finite}, and \eqref{kura:vareps} imply  for all $m \in \N$, $t \in [0,T]$  that
\begin{align}
\begin{split}
& \P \! \left( \sup\nolimits_{x \in (0,1)} | \und{\mathcal{O}_t^m}(x) |^2 \geq \frac{1}{72 p  T  \varepsilon} \right) \\
&\leq 72  p T  \varepsilon \,  \E \! \left[\sup\nolimits_{x \in (0,1)} | \und{\mathcal{O}_t^m}(x) |^2  \right]    \\
& \leq 720   p^3  T  \varepsilon \Big[\sup \! \Big(\Big\{ \! \sup\nolimits_{x \in (0,1)} |v(x)| \colon \big[v \in \mathcal{C}((0,1), \R) \text{ and } \|v\|_{\W^{\beta, p}((0,1), \mathbb{R})} \leq 1\big] \Big\}\Big)\Big]^2 \\
& \quad \cdot \left[ \smallsum\limits_{k =-m}^m \tfrac{ \max \{|k \pi|^{4 \beta}, 1 \} (|b_k|^2 +   | \tilde{b}_{k} |^2 )}{\lambda_k}\right] \leq \tfrac{1}{10}.
\end{split}
\end{align}
Lemma~\ref{lemma:fernique} hence shows that for all $m \in \N$, $t \in [0, T]$ it holds that
\begin{align}\label{kura:prop:phi(w)}
\E \! \left[ \exp \! \left(4 p T \varepsilon  \Big\{ \! \sup\nolimits_{x \in (0,1)} | \und{\mathcal{O}_t^m}(x) |^2\Big\} \right)\right] \leq 13.
\end{align}
Moreover, H\"olders inequality ensures for all $ r \in [0,T]$, $m \in \N$ that
\begin{align}\label{kura:cor:2}
\begin{split}
& \E \!\left[ \exp \! \left( \smallint_r^T p \, \phi\big(\mathcal{O}_{\lf u \rf_{h_m}}^m + P_m  e^{\lf u \rf_{h_m} A} \xi\big) \, du \right) \max \! \left\{1, \big| \Phi\big( \mathcal{O}_{ \lf r \rf_{h_m } }^m + P_m  e^{\lf r \rf_{h_m} A} \xi\big)\big|^{\nicefrac{p}{2}} \right\}\right] \\
&\leq \E \!\left[ \exp \! \left( \smallint_0^T p \, \phi\big(\mathcal{O}_{\lf u \rf_{h_m}}^m + P_m  e^{\lf u \rf_{h_m} A} \xi\big) \, du \right) \max \! \left\{1, \big| \Phi\big( \mathcal{O}_{ \lf r \rf_{h_m } }^m + P_m  e^{\lf r \rf_{h_m} A} \xi\big) \big|^{\nicefrac{p}{2}} \right\}\right] \\
& \leq \sqrt{\E \!\left[ \exp \!\left( \smallint_0^T  2p \,\phi\big(\mathcal{O}_{\lf u \rf_{h_m}}^m +  P_m  e^{\lf u \rf_{h_m} A} \xi\big) \, du \right) \right] \E \! \left[1+ \big| \Phi\big( \mathcal{O}_{ \lf r \rf_{h_m } }^m + P_m  e^{\lf r \rf_{h_m} A} \xi\big) \big|^p \right]}.
\end{split}
\end{align}
Next note that the fact that $ \forall \, x, y \in \R \colon (x+y)^2 \leq 2 x^2 + 2y^2$ yields that for all $m \in \N$ it holds that
\begin{align}\label{eq:kura:phi}
\begin{split}
& \E \!\left[ \exp \! \left( \smallint_0^T  2p \,\phi\big(\mathcal{O}_{\lf u \rf_{h_m}}^m +  P_m  e^{\lf u \rf_{h_m} A} \xi\big) \, du \right) \right] \\
&= \E \!\left[ \exp \! \left( \smallint_0^T  3p\vt +  \tfrac{p \kappa^4}{8} [1 + \nicefrac{1}{\varepsilon}]^2 +  2p\varepsilon  \Big\{ \!\sup\nolimits_{x \in (0,1)} \big|\und{\mathcal{O}_{\lf u \rf_{h_m}}^m +  P_m e^{\lf u \rf_{h_m} A} \xi}(x)\big|^2 \Big\} \, du \right) \right] \\
& \leq \exp \!\left(   3pT\vt + \tfrac{p T \kappa^4}{8} [1 + \nicefrac{1}{\varepsilon}]^2 + 4p\varepsilon \smallint_0^T  \Big\{ \!\sup\nolimits_{x \in (0,1)} |\und{ P_m  e^{\lf u \rf_{h_m} A} \xi}(x)|^2 \Big\} \, du \right) \\
& \quad \cdot \E \!\left[ \exp \! \left( \smallint_0^T    4p\varepsilon  \Big\{ \!\sup\nolimits_{x \in (0,1)} \big|\und{\mathcal{O}_{\lf u \rf_{h_m}}^m} \!(x)\big|^2 \Big\} \, du \right) \right].
\end{split}
\end{align}
In addition, observe that the triangle inequality and the fact that $ \forall \, x, y \in \R$, $a \in [1, \infty) \colon |x+y|^a \leq 2^{a-1} |x|^a + 2^{a-1} |y|^a$ show for all $r \in [0, T]$, $m \in \N$ that
\begin{align}
\begin{split}
& \E \! \left[ \big| \Phi\big( \mathcal{O}_{ \lf r \rf_{h_m } }^m + P_m  e^{\lf r \rf_{h_m} A} \xi\big) \big|^p \right]\\
&=  \E  \bigg[ \left| \tfrac{\vt}{2} \big\| \mathcal{O}_{ \lf r \rf_{h_m } }^m + P_m  e^{\lf r \rf_{h_m} A} \xi \big\|_H^2 + \tfrac{1}{4} \big\| \big(\mathcal{O}_{ \lf r \rf_{h_m } }^m + P_m  e^{\lf r \rf_{h_m} A} \xi\big)^2  \big\|_H^2 \right|^p \bigg] \\
& \leq  \E \! \left[ \tfrac{\vt^p}{2} \big\| \mathcal{O}_{ \lf r \rf_{h_m } }^m +  P_m  e^{\lf r \rf_{h_m} A} \xi \big\|^{2p}_H + \tfrac{1}{2^{p+1}} \big\|\big( \mathcal{O}_{ \lf r \rf_{h_m } }^m + P_m  e^{\lf r \rf_{h_m} A} \xi \big)^2 \big\|^{2p}_H \right] \\
& \leq  \E \! \left[ \tfrac{\vt^p}{2}  \Big\{ \!\sup\nolimits_{x \in (0,1)} \big|\und{\mathcal{O}_{\lf r \rf_{h_m}}^m +  P_m  e^{\lf r \rf_{h_m} A} \xi}(x)\big|^{2p} \Big\} + \tfrac{1}{2^{p+1}} \Big\{\!\sup\nolimits_{x \in (0,1)} \big|\und{\mathcal{O}_{\lf r \rf_{h_m}}^m +  P_m  e^{\lf r \rf_{h_m} A} \xi}(x)\big|^{4p} \Big\} \right] \\
& \leq    2^{2p-2}\vt^p \Big\{\!\sup\nolimits_{x \in (0,1)} |\und{ P_m  e^{\lf r \rf_{h_m} A} \xi}(x)|^{2p} \Big\} + 2^{3p-2} \Big\{\!\sup\nolimits_{x \in (0,1)} |\und{ P_m  e^{\lf r \rf_{h_m} A} \xi}(x)|^{4p} \Big\}  \\
& +  \E\! \left[ 2^{2p-2}\vt^p \Big\{\!\sup\nolimits_{x \in (0,1)} \big|\und{\mathcal{O}_{\lf r \rf_{h_m}}^m} \!(x)\big|^{2p} \Big\} + 2^{3p-2} \Big\{\!\sup\nolimits_{x \in (0,1)} \big|\und{\mathcal{O}_{\lf r \rf_{h_m}}^m }\!(x)\big|^{4p} \Big\} \right].
\end{split}
\end{align}
Lemma~\ref{lem:exp} hence proves for all $r \in [0, T]$, $m \in \N$ that
\begin{align}\label{eq:kura:Phi}
\begin{split}
& \E \! \left[ \big| \Phi\big( \mathcal{O}_{ \lf r \rf_{h_m } }^m + P_m  e^{\lf r \rf_{h_m} A} \xi\big) \big|^p \right] \\
&\leq     2^{2p-2}\vt^p \Big\{\!\sup\nolimits_{x \in (0,1)} |\und{ P_m  e^{\lf r \rf_{h_m} A} \xi}(x)|^{2p} \Big\}  + 2^{3p-2} \Big\{\!\sup\nolimits_{x \in (0,1)} |\und{ P_m  e^{\lf r \rf_{h_m} A} \xi}(x)|^{4p} \Big\} \\
& +  \E\! \left[ \left( \tfrac{2^{2p-2}\vt^p (\lf p \rf_1 +1)!}{ |4pT\varepsilon|^p} + \tfrac{2^{3p-2} (\lf 2p \rf_1 +1)! }{|4pT\varepsilon|^{2p}} \right)  \exp \! \left( 4pT\varepsilon \Big\{\!\sup\nolimits_{x \in (0,1)} |\und{\mathcal{O}_{\lf r \rf_{h_m}}^m }\!(x)|^{2} \Big\} \right)  \right].
\end{split}
\end{align}
Combining this with \eqref{kura:cor:2}, \eqref{eq:kura:phi}, and the fact that $\forall \, x, y \in [0, \infty) \colon \sqrt{x+y} \leq \sqrt{x} + \sqrt{y}$ ensures that for all $m \in \N$ it holds that
 \begin{align}\label{eq:kura:phi:Phi}
 \begin{split}
 & \E \!\left[ \int_0^T \exp \!\left( \smallint_r^T p \, \phi\big(\mathcal{O}_{\lf u \rf_{h_m}}^m + P_m  e^{\lf u \rf_{h_m} A} \xi\big) \, du \right) \max \! \left\{1, \big| \Phi\big( \mathcal{O}_{ \lf r \rf_{h_m } }^m + P_m  e^{\lf r \rf_{h_m} A} \xi\big)\big|^{\nicefrac{p}{2}} \right\} dr\right] \\
 & \leq   \exp \!\left(   \tfrac{3pT\vt}{2} + \tfrac{p T \kappa^4}{16} [1 + \nicefrac{1}{\varepsilon}]^2 + 2p\varepsilon \smallint_0^T \Big[\sup\nolimits_{x \in (0,1)} \big|\und{ P_m  e^{\lf u \rf_{h_m} A} \xi}(x)\big|^2 \Big] \, du \right) \\
 &  \cdot \sqrt{\E \!\left[ \exp \!\left( \smallint_0^T    4p\varepsilon \, \Big\{\!\sup\nolimits_{x \in (0,1)} \big|\und{\mathcal{O}_{\lf u \rf_{h_m}}^m}(x)\big|^2 \Big\} \, du \right)  \right]}\\
 & \cdot \Bigg( \int_0^T  1+ 2^{p-1}\vt^{\nicefrac{p}{2}} \Big\{ \!\sup\nolimits_{x \in (0,1)} |\und{ P_m  e^{\lf r \rf_{h_m} A} \xi}(x)|^{p} \Big \} + 2^{2p-1} \Big\{\!\sup\nolimits_{x \in (0,1)} |\und{ P_m  e^{\lf r \rf_{h_m} A} \xi}(x)|^{2p} \Big\} \, dr  \\
 & + \int_0^T  \sqrt{ \left( \tfrac{\vt^p (\lf p \rf_1 +1)!}{ 4|pT\varepsilon|^p} + \tfrac{ (\lf 2p \rf_1 +1)! }{2^{p+2} |pT\varepsilon|^{2p}} \right) \E\! \left[ \exp \! \left( 4pT\varepsilon  \Big\{ \!\sup\nolimits_{x \in (0,1)} \big|\und{\mathcal{O}_{\lf r \rf_{h_m}}^m }(x)\big|^{2} \Big\}  \right)  \right] } \, dr \Bigg).
 \end{split}
 \end{align}
Next note that, e.g., Lemma~2.22 in Cox et al.~\cite{CoxHutzenthalerJentzen2014}  and  \eqref{kura:prop:phi(w)} show that for all $m \in \N$ it holds that
\begin{equation}\label{kura:prop:jen}
\begin{split}
& \E \! \left[ \exp \! \left(  \smallint_0^T 4 p \varepsilon  \Big\{\!\sup\nolimits_{x \in (0,1)} \big|\und{\mathcal{O}_{\lf u \rf_{h_m}}^m} \! (x)\big|^2 \Big\} \,  du \right)\right] \\\
& \leq \frac{1}{T}  \int_0^T \E\! \left[ \exp \! \left(  4  p T \varepsilon \Big\{\!\sup\nolimits_{x \in (0,1)} \big|\und{\mathcal{O}_{\lf u \rf_{h_m}}^m}\!(x)\big|^2 \Big\} \right)\right] du \leq 13 .
\end{split}
\end{equation} 
Combining  \eqref{eq:kura:phi:Phi} with \eqref{kura:prop:phi(w)}  hence shows for all $m \in \N$ that
\begin{align}\label{eq:kura:phi:finite}
\begin{split}
& \E \!\left[ \int_0^T \exp \!\left( \smallint_r^T p \, \phi\big(\mathcal{O}_{\lf u \rf_{h_m}}^m + P_m  e^{\lf u \rf_{h_m} A} \xi\big) \, du \right) \max \!\left\{1, \big| \Phi\big( \mathcal{O}_{ \lf r \rf_{h_m } }^m + P_m  e^{\lf r \rf_{h_m} A} \xi\big)\big|^{\nicefrac{p}{2}} \right\} dr\right] \\
& \leq \sqrt{13} \exp \! \left(   \tfrac{3pT\vt}{2} + \tfrac{p T \kappa^4}{16} [1 + \nicefrac{1}{\varepsilon}]^2 + 2p\varepsilon \smallint_0^T \Big\{\!\sup\nolimits_{x \in (0,1)} |\und{ P_m  e^{\lf u \rf_{h_m} A} \xi}(x)|^2 \Big\} \, du \right) \\
& \cdot \bigg( T +  \int_0^T 2^{p-1}\vt^{\nicefrac{p}{2}} \Big\{\!\sup\nolimits_{x \in (0,1)} |\und{ P_m  e^{\lf r \rf_{h_m} A} \xi} (x)|^p \Big\} + 2^{2p-1} \Big\{\!\sup\nolimits_{x \in (0,1)} |\und{ P_m  e^{\lf r \rf_{h_m} A} \xi}(x)|^{2p} \Big\} \, dr \\
&   + T\sqrt{13\left( \tfrac{\vt^p (\lf p \rf_1 +1)!}{4 |pT\varepsilon|^p} + \tfrac{ (\lf 2p \rf_1 +1)! }{2^{p+2} |pT\varepsilon|^{2p}} \right) } \bigg).
\end{split}
\end{align}
Moreover, the Sobolev embedding theorem implies that 
\begin{align}
\sup \! \Big(\Big\{\! \sup\nolimits_{x \in (0,1)} |\und{v}(x)| \colon \big[ v \in H_{1/4} \text{ and } \|v\|_{H_{1/4}} \leq 1\big] \Big\}\Big) < \infty.
\end{align}
This yields for all $s \in [0, T]$, $m \in \N$ that
\begin{align}
\begin{split}
&\sup\nolimits_{x \in (0,1)} | \und{P_m  e^{sA} \xi} (x) | \\
&\leq  \Big[\sup \! \Big(\Big\{ \!\sup\nolimits_{x \in (0,1)} |\und{v}(x)| \colon \big[v \in H_{1/4} \text{ and } \|v\|_{H_{1/4}} \leq 1\big] \Big\}\Big) \Big]  \|P_m  e^{sA}  \xi \|_{H_{1/4}}  \\
& \leq \Big[\sup \! \Big(\Big\{ \sup\nolimits_{x \in (0,1)} |\und{v}(x)| \colon \big[v \in H_{1/4} \text{ and } \|v\|_{H_{1/4}} \leq 1\big] \Big\}\Big) \Big] \| \xi \|_{H_{1/4}} < \infty.
\end{split}
\end{align}
Combining this with \eqref{eq:kura:phi:finite} implies that
\begin{align}\label{eq:kura:last}
\limsup_{m \to \infty}   \E \! \left[ \int_0^T \exp \! \left( \smallint_r^T p \, \phi\big(\mathcal{O}_{\lf u \rf_{h_m}}^m + P_m  e^{\lf u \rf_{h_m} A} \xi\big) \, du \right) \max \! \left\{1, \big| \Phi\big( \mathcal{O}_{ \lf r \rf_{h_m } }^m + P_m  e^{\lf r \rf_{h_m} A} \xi\big)\big|^{\nicefrac{p}{2}} \right\} dr \right] < \infty.
\end{align}
In the next step observe that Lemma~\ref{lem:exp} and \eqref{kura:prop:phi(w)} prove that for all $m \in \N$, $s \in [0, T]$ it holds that 
\begin{align}
\begin{split}
\E \big[\|  \mathcal{O}_s^m \|_H^p \big]  \leq \E \Big[\sup\nolimits_{x \in (0,1)} |\und{\mathcal{O}_{s}^m }(x)|^{p} \Big] &\leq \tfrac{(\lf p/2 \rf_1+1)! }{|4pT\varepsilon|^{p/2}} \, \E\! \left[ \exp \! \left( 4pT\varepsilon \Big\{\!\sup\nolimits_{x \in (0,1)} |\und{\mathcal{O}_{s}^m }(x)|^{2} \Big\} \right)\right] \\
& \leq \tfrac{13 \, (\lf p/2 \rf_1+1)! }{2^p |pT\varepsilon|^{p/2}}.
\end{split}
\end{align}
The triangle inequality and the fact that $\forall \, x, y \in \R \colon |x+y|^p \leq 2^{p-1} |x|^p + 2^{p-1} |y|^p$ hence show that
\begin{align}
\begin{split}
\limsup_{m \to \infty}  \sup_{s \in [0,T]} \E \big[\|  \mathcal{O}_s^m  + P_m  e^{sA}  \xi  \|_H^p \big] &\leq  2^{p-1} \limsup_{m \to \infty} \bigg( \sup_{s \in [0,T]}  \E  \big[\|  \mathcal{O}_s^m \|_H^p \big] +  \sup_{s \in [0,T]} \|P_m  e^{sA}  \xi\|_H^p \bigg)\\
& \leq \tfrac{13 \, (\lf p/2 \rf_1+1)! }{2\, |pT\varepsilon|^{p/2}}  + 2^{p-1} \|\xi \|_H^p< \infty.
\end{split}
\end{align}
Combining this with \eqref{eq:kura:last} completes the proof of Corollary~\ref{kura:cor}.
\end{proof}

\begin{lemma}\label{lemma:conv:rate}
Assume the setting in Section~\ref{setting:kura}, let  $p \in [2, \infty)$, $n \in \N$, $\varepsilon \in [0, \nicefrac{\beta}{2} -\varrho)$, let $O \colon  [0, T] \times \Omega \to H_{\varrho}$  be a  stochastic process, and assume for all $t \in [0, T]$ that $ [O_t ]_{\P, \mathcal{B}(H)} =  \int_0^t e^{(t-s)A} \, B \, dW_s$. Then 
\begin{align}
\sup_{t \in [0, T]} \|O_t  - \mathcal{O}_t^n \|_{\mathcal{L}^p(\P; H_{\varrho})} \leq \left[\tfrac{p(p-1) T^{ ( \beta - 2\varrho - 2\varepsilon  ) }}{2( \beta - 2\varrho - 2\varepsilon  )} \right]^{\nicefrac{1}{2}} \|B \|_{HS(H, H_{(\beta-1)/2})} \,  n^{- 4\varepsilon }.
\end{align}
\end{lemma}
\begin{proof}[Proof of Lemma~\ref{lemma:conv:rate}.]
First, observe that the Burkholder-Davis-Gundy type inequality in Lemma~7.7 in Da Prato \& Zabczyk~\cite{dz92} implies that for all  $t \in [0, T]$  it holds that
\begin{align}\label{eq:burk}
\begin{split}
\|O_t - \mathcal{O}_t^n\|_{\mathcal{L}^p(\P; H_{\varrho})}  & = \left\| \smallint_0^t (\mathrm{Id}_H- P_n) \, e^{(t-s)A} \, B \, d W_s \right\|_{L^p(\P; H_{\varrho})} \\
& \leq \left[ \tfrac{p(p-1)}{2} \smallint_0^t \big \|(\mathrm{Id}_H- P_n) \,  e^{(t-s)A} \, B \big\|^2_{HS(H, H_{\varrho})} \, ds \right]^{\nicefrac{1}{2}} \\
& \leq \left[ \tfrac{p(p-1)}{2}  \|B \|^2_{HS(H, H_{(\beta-1)/2})} \smallint_0^t \big \|(\mathrm{Id}_H- P_n) \, e^{(t-s)A} \big\|^2_{L(H_{(\beta-1)/2}, H_{\varrho})} \, ds \right]^{\nicefrac{1}{2}}.
\end{split}
\end{align}
Next note that, e.g., Theorem~4.7.6 
in \cite{Jentzen2015SPDElecturenotes} proves for all  $ t \in [0, T] $ that
\begin{align}
\begin{split}
& \int_0^t \big \|(\mathrm{Id}_H- P_n) \, e^{(t-s)A} \big\|^2_{L(H_{(\beta-1)/2}, H_{\varrho})} \, ds \leq  \int_0^t  \|\mathrm{Id}_H- P_n \|^2_{L(H_{\varrho+\varepsilon}, H_{\varrho})} \, \| e^{(t-s)A} \|^2_{L(H_{(\beta-1)/2}, H_{\varrho+ \varepsilon})} \, ds \\ & = 
\| 
( - A )^{ - \varepsilon }
( \Id_H - P_n )
\|^2_{
	L( H )
}
\int_0^t  
\|
( - A )^{ ( \varrho + \varepsilon + \nicefrac{(1-\beta)}{2} ) }  
\,
e^{ s A } 
\|^2_{ L(H) } 
\, ds  
\\ & \leq 
| \lambda_{n+1}|^{-2 \varepsilon } 
\int_0^t s^{ - 2( \varrho + \varepsilon + \nicefrac{(1-\beta)}{2} ) } \, ds 
\leq   
\frac{ n^{-8 \varepsilon }
	\, 
	t^{ ( \beta - 2\varrho - 2\varepsilon  ) }
}{
( \beta - 2\varrho - 2\varepsilon  ) 
}
.
\end{split}
\end{align}
This together with \eqref{eq:burk} yields that for all  $t \in [0, T]$ it holds that
\begin{align}
\|O_t - \mathcal{O}_t^n\|_{\mathcal{L}^p(\P; H_{\varrho})} \leq \left[\tfrac{p(p-1) T^{ ( \beta - 2\varrho - 2\varepsilon  ) }}{2( \beta - 2\varrho - 2\varepsilon  )} \right]^{\nicefrac{1}{2}}  \|B \|_{HS(H, H_{(\beta-1)/2})} \, 	n^{- 4\varepsilon }.
\end{align}
The proof of Lemma~\ref{lemma:conv:rate} is thus completed.
\end{proof}

\subsection{Strong convergence}

\begin{cor}
\label{kura:cor:last}
Assume the setting in Section~\ref{setting:kura} and let $ X\colon [0, T] \times \Omega \to H_{\varrho}$  be a  stochastic process with continuous sample paths which satisfies for all $t \in [0, T]$  that  $ [X_t ]_{\P, \mathcal{B}(H)} =    [ e^{ t A }   \xi  + \smallint_0^t e^{  ( t - s ) A}  \, F (  X_s ) \, ds ]_{\P, \mathcal{B}(H)} + \int_0^t e^{(t-s)A} \, B \, dW_s$. Then it holds for all $p \in (0, \infty)$ that
\begin{align}
\limsup_{n \to \infty} \sup_{t \in [0,T]} \E \big[ \| X_t -\Y_t^n \|_H^p \big] = 0.
\end{align}
\end{cor}
\begin{proof}[Proof of Corollary~\ref{kura:cor:last}.]
Throughout this proof let $ \phi , \Phi \in \mathcal{M}( \mathcal{B}(H_{1}) , \mathcal{B}([0,\infty)))$,  $\varepsilon \in (0, \nicefrac{\beta}{2} -\varrho)$, $p, \gamma \in (0, \infty)$, $q \in  ( \max\{p, \nicefrac{1}{\beta},\nicefrac{4}{\varepsilon} \}, \infty)$ satisfy for all $ v \in H_{1}$ that  $\phi(v) = \frac{3\vt}{2} + \tfrac{\kappa^4}{16} [1 + \nicefrac{1}{\gamma}]^2 + \gamma \!   \left[\sup\nolimits_{x \in (0,1)} |\und{v}(x)|^2 \right]$, $\Phi(v)= \frac{\vt}{2} \|v \|_H^2 +  \frac{1}{4} \| v^2 \|_H^2$, and
\begin{align}\label{eq:eps:bound}
\begin{split}
\gamma &\leq  \tfrac{1}{7200  q^3 T} \left[ \max\! \left\{ 1, \smallsum\limits_{k \in \mathbb{Z}} \tfrac{  \max \{|k \pi|^{4 \beta}, 1 \} (|b_k|^2 +  | \tilde{b}_{k} |^2 )}{\lambda_k} \right\} \right]^{-1} \\
& \quad  \cdot \Big[\sup \!\Big(\Big\{\! \sup\nolimits_{x \in (0,1)} |v(x)| \colon \big[v \in \mathcal{C}((0,1), \R) \text{ and } \|v\|_{\W^{\beta, q}((0,1), \mathbb{R})} \leq 1\big] \Big\}\Big)\Big]^{-2}.
\end{split}
\end{align}
Next note that Lemma~\ref{coer:kura} implies that for all $v, w \in H_1$ it holds that
\begin{align}\label{kura:cor:coer}
\left< v, F( v + w ) \right>_H \leq \phi(w) \| v \|^2_H+ \tfrac{1}{2} \| v \|^2_{ H_{ 1 / 2 } }+ \Phi( w ).
\end{align}
Moreover, Lemma~\ref{nonlin:kura} proves that for all $v, w \in H_{1}$ it holds that $F \in \mathcal{C}(H_{1/16}, H_{-1/4})$ and
\begin{align}\label{kura:cor:con}
\begin{split}
&\|F(v) - F(w) \|_{H_{-1 /4}} \\
& \leq  \left( \vt^{\nicefrac{11}{16}} + |\kappa| \left[\sup_{u \in  H_{ 1/16 }\backslash \{0\}}  \frac{\|u\|^2_{L^4(\lambda_{(0,1)}; \R)}}{\|u\|^2_{H_{1/16}}} \right]\right) \big( 1+ \left\| v \right\|_{ H_{1/16 }} +  \left\| w \right\|_{ H_{1/16 }} \big) \left\| v - w \right\|_{ H_{1/16}} < \infty.
\end{split}
\end{align} 
In the next step observe that the Burkholder-Davis-Gundy type inequality in Lemma~7.7 in Da Prato \& Zabczyk~\cite{dz92} shows that for all $n \in \N$, $t_1, t_2 \in [0, T]$ with $t_1 < t_2$ it holds that
\begin{align}
\begin{split}
&\left\| \int_0^{t_1}P_n \, e^{(t_1-s)A} \, B \, dW_s - \int_0^{t_2} P_n \, e^{(t_2-s)A} \, B \, dW_s \right\|^2_{L^q(\P; H_{\varrho})} \\
&+ \left\| \int_0^{t_1} e^{(t_1-s)A} \, B \, dW_s - \int_0^{t_2}  e^{(t_2-s)A} \, B \, dW_s \right\|^2_{L^q(\P; H_{\varrho})}  \\
& \leq q(q-1) \int_{t_1}^{t_2} \big \| e^{(t_2-s)A} \, B \big\|^2_{HS(H, H_{\varrho})} \, ds + q (q-1) \int_{0}^{t_1} \big \| \big(e^{(t_1-s)A}- e^{(t_2-s)A}\big) B \big\|^2_{HS(H, H_{\varrho})} \, ds  \\
& \leq q (q-1) \|B \|^2_{HS(H, H_{(\beta-1)/2})} \\
&  \cdot \left[ \int_{t_1}^{t_2} \big \|  e^{(t_2-s)A}  \big\|^2_{L(H_{(\beta-1)/2}, H_{\varrho})} \, ds+ \int_0^{t_1} \big\| e^{(t_1-s)A} \big(\Id_H - e^{(t_2-t_1)A}\big)  \big\|_{L(H_{(\beta-1)/2}, H_{\varrho})}^2 \, ds \right].
\end{split}
\end{align}
Therefore, e.g., Theorem~4.7.6 in \cite{Jentzen2015SPDElecturenotes} and, e.g., Lemma~4.7.7 in \cite{Jentzen2015SPDElecturenotes} imply that for all   $n \in \N$, $t_1, t_2 \in [0, T]$ with $t_1 < t_2$ it holds that
\begin{align}\label{eq:O:Holder}
\begin{split}
&\left\| \int_0^{t_1}P_n \, e^{(t_1-s)A} \, B \, dW_s - \int_0^{t_2} P_n \, e^{(t_2-s)A} \, B \, dW_s \right\|^2_{L^q(\P; H_{\varrho})}  \\
&+ \left\| \int_0^{t_1} e^{(t_1-s)A} \, B \, dW_s - \int_0^{t_2}  e^{(t_2-s)A} \, B \, dW_s \right\|^2_{L^q(\P; H_{\varrho})}  \\
&   \leq q(q-1)  \|B \|^2_{HS(H, H_{(\beta-1)/2})}  \bigg[ \int_{t_1}^{t_2} \big \| (-A)^{(\varrho- \nicefrac{(\beta-1)}{2})} \, e^{(t_2-s)A}  \big\|^2_{L(H)} \, ds \\
& + \int_0^{t_1} \big\| (-A)^{(\varrho- \nicefrac{(\beta-1)}{2} +\varepsilon)} \, e^{(t_1-s)A}  \big\|_{L(H)}^2 \, \big\|(-A)^{-\varepsilon} \big(\Id_H - e^{(t_2-t_1)A}\big)  \big\|_{L(H)}^2 \, ds \bigg]\\
& \leq q (q-1)  \|B \|^2_{HS(H, H_{(\beta-1)/2})} \left[ \int_{t_1}^{t_2} (t_2-s)^{(\beta-1-2\varrho)} \, ds + \int_0^{t_1} (t_1-s)^{(\beta-1-2\varrho-2\varepsilon)} (t_2 -t_1)^{2\varepsilon} \, ds \right]\\
& = q(q-1) \|B \|^2_{HS(H, H_{(\beta-1)/2})} \left[ \frac{(t_2-t_1)^{(\beta- 2 \varrho)}}{(\beta- 2\varrho)} + \frac{|t_1|^{(\beta- 2 \varrho -2\varepsilon)} (t_2-t_1)^{2\varepsilon}}{(\beta- 2\varrho-2\varepsilon)} \right]\\
& \leq \frac{2q(q-1) \|B \|^2_{HS(H, H_{(\beta-1)/2})}  T^{(\beta- 2 \varrho -2\varepsilon)} (t_2-t_1)^{2\varepsilon}}{(\beta- 2\varrho-2\varepsilon)}  .
\end{split}
\end{align}
Combining this with the Kolmogorov-Chentsov theorem and the fact that $ q \varepsilon > 1$ yields that there exist  stochastic processes  $ O \colon [0, T] \times \Omega \to H_{\varrho}$ and  $\tilde{\mathcal{O}}^n \colon [0, T] \times \Omega \to P_n(H)$, $n \in \N$, with continuous sample paths which satisfy for all $n \in \N$, $t \in [0, T]$ that $[O_t]_{\P, \mathcal{B}(H)} =  \int_0^{t}  e^{(t-s)A} \, B \, dW_s$ and $[\tilde{\mathcal{O}}^n_t]_{\P, \mathcal{B}(H)} =  \int_0^{t} P_n \, e^{(t-s)A} \, B \, dW_s$.
Next observe that Lemma~\ref{lemma:conv:rate} proves that for all $n \in \N$ it holds that
\begin{align}\
 \sup_{t \in [0, T]} \|O_t  - \tilde{\mathcal{O}}_t^n \|_{\mathcal{L}^q(\P; H_{\varrho})} \leq \left[\tfrac{q(q-1) T^{ ( \beta - 2\varrho - 2\varepsilon  ) }}{2( \beta - 2\varrho - 2\varepsilon  )} \right]^{\nicefrac{1}{2}} \|B \|_{HS(H, H_{(\beta-1)/2})} \,  n^{- \varepsilon }.
\end{align}
This, the fact that $ O \colon [0, T] \times \Omega \to H_{\varrho}$ and  $\tilde{\mathcal{O}}^n \colon [0, T] \times \Omega \to P_n(H)$, $n \in \N$, are stochastic processes with continuous sample paths, \eqref{eq:O:Holder}, and Corollary~2.11 in Cox et al.~\cite{CoxWelti2016} (with $T=T$, $p=q$, $\beta= \varepsilon$, $\theta^N= \{ \frac{k T}{N} \in [0, \infty) \colon k \in \N_0 \cap [0, N] \}$, $(E, \left\| \cdot \right\|_E)= (H_{\varrho}, \left\| \cdot \right\|_{H_{\varrho}})$, $Y^N= ([0,T] \times \Omega  \ni (t, \omega) \mapsto \tilde{\mathcal{O}}^N_t(\omega) \in H_{\varrho})$, $Y^0= O$, $\alpha=0$, $\varepsilon= \nicefrac{\varepsilon}{2}$ for $N \in \N$ in the notation of Corollary~2.11 in Cox et al.~\cite{CoxWelti2016}) ensure that
\begin{align}
\sup_{n \in \N} \Bigg( n^{(\nicefrac{\varepsilon}{2}- \nicefrac{1}{q})} \bigg\| \sup_{t \in [0, T]} \|O_t - \tilde{\mathcal{O}}_t^n \|_{H_{\varrho}} \bigg\|_{\mathcal{L}^q(\P; \R)} \Bigg)< \infty.
\end{align}
Lemma~3.21 in  \cite{Hutzenthaler2015} (cf., e.g., Theorem~7.12 in Graham \& Talay \cite{Graham2013} and Lemma~2.1 in Kloeden \& Neuenkirch \cite{Kloeden2007}) together with the fact that $\nicefrac{\varepsilon}{2}- \nicefrac{1}{q} > \nicefrac{1}{q}$ hence yields that  
\begin{align}\label{eq:O:conv}
\P \bigg(\limsup_{n \to \infty} \sup_{s \in [0, T]} \| O_s - \tilde{\mathcal{O}}_s^n \|_{H_{\varrho}} =0 \bigg)=1.
\end{align}
In the next step observe that  for all $n \in \N$, $t \in [0, T]$ it holds that 
\begin{align}
\|(\Id_H- P_n) \, e^{tA}  \xi \|_{H_{\varrho}} \leq \|(-A)^{\varrho-\nicefrac{1}{4}} (\Id_H-P_n) \|_{L(H)} \|\xi\|_{H_{1/4}} = |\lambda_{n+1}|^{\varrho-\nicefrac{1}{4}} \|\xi\|_{H_{1/4}} \leq n^{\varrho-\nicefrac{1}{4}} \|\xi\|_{H_{1/4}}.
\end{align}
Combining this with \eqref{eq:O:conv} proves that 
\begin{align}\label{eq:O:as}
\P \bigg( \limsup_{n \to \infty} \sup_{s \in [0, T]} \big\| (O_s + e^{sA}  \xi) - (\tilde{\mathcal{O}}_s^n + P_n  e^{sA}  \xi )\big\|_{H_{\varrho}} =0 \bigg)=1.
\end{align}
Moreover, note that the fact that  $ \forall \, n \in \N$, $t \in [0, T] \colon \P(\mathcal{O}_t^n = \tilde{\mathcal{O}}^n_t )=1$ and \eqref{eq:kura_X} ensure that for all $n \in \N$, $t \in [0, T]$ it holds that 
\begin{equation}\label{eq:X^n}
\P \Big( \Y_t^n = P_n  e^{ t A } \xi + \smallint_0^t P_n \,  e^{  ( t - s ) A } \, \mathbbm{1}_{ \{ \| \Y_{ \lf s \rf_{h_n} }^n \|_{ H_{\varrho} } + \| \tilde{\mathcal{O}}_{ \lf s \rf_{h_n} }^n +P_n  e^{ \lf s \rf_{ h_n } A } \xi \|_{ H_{\varrho} } \leq | h_n|^{ - \chi } \}} \, F \big(  \Y_{ \lf s \rf_{ h_n } }^n \big) \, ds + \tilde{\mathcal{O}}_t^n \Big)=1. 
\end{equation}
In addition, Corollary~\ref{kura:cor}, \eqref{eq:eps:bound}, and again the fact that  $ \forall \, n \in \N$, $t \in [0, T] \colon \P(\mathcal{O}_t^n = \tilde{\mathcal{O}}^n_t )=1$ show that
 \begin{align}\label{eq:strong:limsup}
 \begin{split}
 &\limsup_{m \to \infty}   \E \! \left[ \int_0^T \exp \! \left( \smallint_r^T q \, \phi\big(\tilde{\mathcal{O}}_{\lf u \rf_{h_m}}^m + P_m e^{\lf u \rf_{h_m} A} \xi\big) \, du \right) \max \! \left\{1, \big| \Phi\big( \tilde{\mathcal{O}}_{ \lf r \rf_{h_m } }^m + P_m e^{\lf r \rf_{h_m} A} \xi\big)\big|^{\nicefrac{q}{2}} \right\} dr \right] \\
 &+ \limsup_{m \to \infty}  \sup_{s \in [0,T]} \E  \big[\|  \tilde{\mathcal{O}}_s^m  + P_m e^{sA}  \xi  \|_H^q \big] < \infty.  
 \end{split}
 \end{align}
Combining  \eqref{kura:cor:coer}--\eqref{kura:cor:con}, \eqref{eq:O:as}--\eqref{eq:strong:limsup},  the fact that $p \in (0, q)$, the fact that $ \forall \, t \in [0, T] \colon \P(X_t = \int_0^t e^{(t-s)A} \, F(X_s) \, ds +O_t + e^{tA}  \xi)=1$, and Item~(v) in Theorem~\ref{theorem:main} (with
$\mathbb{H}= \{ e_k \in H \colon k \in \Z \}$, $\alpha= \nicefrac{1}{4}$, $\varphi= \nicefrac{1}{2}$, $\rho= \nicefrac{1}{16}$, $\varrho =\varrho$, $ \theta=  \vt^{\nicefrac{11}{16}} + |\kappa| [\sup_{u \in  H_{ 1/16 }\backslash \{0\}}  \nicefrac{\|u\|^2_{L^4(\lambda_{(0,1)}; \R)}}{\|u\|^2_{H_{1/16}}} ]$, $\vartheta=1$, $p=q$, $F=F$, $\phi=\phi$, $\Phi=\Phi$, $\mathbb{H}_n = \{ e_k \in \mathbb{H} \colon k \in \{-n, 1-n, \ldots , n-1, n\}\}$, $h_n=h_n$,  $\mathcal{X}^n = ([0, T] \times \Omega \ni (\omega, t) \mapsto \mathcal{X}_t^n(\omega) \in H_{\varrho})$,    $ 
\mathcal{O}^n = ( [0,T] \times \Omega \ni (t,\omega) \mapsto (\tilde{\mathcal{O}}^n_t( \omega ) + P_n  e^{tA} \xi) \in H_{\varrho} ) 
$,  $X=X$, $ 
O = ( [0,T] \times \Omega \ni (t,\omega) \mapsto (O_t( \omega ) + e^{tA}  \xi) \in H_{\varrho}) 
$,  $q=p$ for $n \in \N$   in the notation of Theorem~\ref{theorem:main}) completes the proof of Corollary~\ref{kura:cor:last}. 
\end{proof}

\subsubsection*{Acknowledgement}
This project has been partially supported by the Deutsche Forschungsgesellschaft (DFG) via RTG 2131
\textit{High-dimensional Phenomena in Probability -- Fluctuations and Discontinuity}.

\bibliographystyle{acm}
\bibliography{bibfile}

\def\polhk#1{\setbox0=\hbox{#1}{\ooalign{\hidewidth
  \lower1.5ex\hbox{`}\hidewidth\crcr\unhbox0}}}
\begin{thebibliography}{10}

\bibitem{BeckerJentzen2016}
{\sc {Becker}, S., and {Jentzen}, A.}
\newblock {Strong convergence rates for nonlinearity-truncated Euler-type
  approximations of stochastic Ginzburg-Landau equations}.
\newblock {\em arXiv:1601.05756\/} (2016), 58 pages.

\bibitem{BloemkerJentzen2013}
{\sc Bl{\"o}mker, D., and Jentzen, A.}
\newblock Galerkin approximations for the stochastic {B}urgers equation.
\newblock {\em SIAM J. Numer. Anal. 51}, 1 (2013), 694--715.

\bibitem{Bogachev2007}
{\sc Bogachev, V.~I.}
\newblock {\em Measure {T}heory. {V}ol. {I}, {II}}.
\newblock Springer-Verlag, Berlin, 2007.

\bibitem{Brzezniak2013}
{\sc Brze{\'z}niak, Z., Carelli, E., and Prohl, A.}
\newblock Finite-element-based discretizations of the incompressible
  {N}avier-{S}tokes equations with multiplicative random forcing.
\newblock {\em IMA J. Numer. Anal. 33}, 3 (2013), 771--824.

\bibitem{CoxHutzenthalerJentzen2014}
{\sc {Cox}, S., {Hutzenthaler}, M., and {Jentzen}, A.}
\newblock Local {L}ipschitz continuity in the initial value and strong
  completeness for nonlinear stochastic differential equations.
\newblock {\em arXiv:1309.5595\/} (2013), 88 pages.

\bibitem{CoxWelti2016}
{\sc {Cox}, S., {Hutzenthaler}, M., {Jentzen}, A., {van Neerven}, J., and
  {Welti}, T.}
\newblock Convergence in {H}{\"o}lder norms with applications to {M}onte
  {C}arlo methods in infinite dimensions.
\newblock {\em arXiv:1605.00856\/} (2016), 38 pages.

\bibitem{dz92}
{\sc Da~Prato, G., and Zabczyk, J.}
\newblock {\em Stochastic equations in infinite dimensions}, vol.~44 of {\em
  Encyclopedia of Mathematics and its Applications}.
\newblock Cambridge University Press, Cambridge, 1992.

\bibitem{DuanErvin2001_Kuramoto_Sivashinsky}
{\sc Duan, J., and Ervin, V.~J.}
\newblock On the stochastic {K}uramoto-{S}ivashinsky equation.
\newblock {\em Nonlinear Anal. 44}, 2, Ser. A: Theory Methods (2001), 205--216.

\bibitem{Graham2013}
{\sc Graham, C., and Talay, D.}
\newblock {\em Stochastic simulation and {M}onte {C}arlo methods}, vol.~68 of
  {\em Stochastic Modelling and Applied Probability}.
\newblock Springer, Heidelberg, 2013.
\newblock Mathematical foundations of stochastic simulation.

\bibitem{gm05}
{\sc Gy{\"o}ngy, I., and Millet, A.}
\newblock On discretization schemes for stochastic evolution equations.
\newblock {\em Potential Anal. 23}, 2 (2005), 99--134.

\bibitem{GoengySabanisS2015}
{\sc Gy{\"o}ngy, I., Sabanis, S., and {\v{S}}i{\v{s}}ka, D.}
\newblock Convergence of tamed {E}uler schemes for a class of stochastic
  evolution equations.
\newblock {\em Stoch. Partial Differ. Equ. Anal. Comput.\/} (2015), 1--21.

\bibitem{h81}
{\sc Henry, D.}
\newblock {\em Geometric theory of semilinear parabolic equations}, vol.~840 of
  {\em Lecture Notes in Mathematics}.
\newblock Springer-Verlag, Berlin, 1981.

\bibitem{Hu1996}
{\sc Hu, Y.}
\newblock Semi-implicit {E}uler-{M}aruyama scheme for stiff stochastic
  equations.
\newblock In {\em Stochastic analysis and related topics, {V} ({S}ilivri,
  1994)}, vol.~38 of {\em Progr. Probab.} Birkh\"auser Boston, Boston, MA,
  1996, pp.~183--202.

\bibitem{hj11}
{\sc Hutzenthaler, M., and Jentzen, A.}
\newblock Convergence of the stochastic {E}uler scheme for locally {L}ipschitz
  coefficients.
\newblock {\em Found. Comput. Math. 11}, 6 (2011), 657--706.

\bibitem{Hutzenthaler2015}
{\sc Hutzenthaler, M., and Jentzen, A.}
\newblock Numerical approximations of stochastic differential equations with
  non-globally {L}ipschitz continuous coefficients.
\newblock {\em Mem. Amer. Math. Soc. 4\/} (2015), 1--112.

\bibitem{hjk11}
{\sc Hutzenthaler, M., Jentzen, A., and Kloeden, P.~E.}
\newblock Strong and weak divergence in finite time of {E}uler's method for
  stochastic differential equations with non-globally {L}ipschitz continuous
  coefficients.
\newblock {\em Proc. R. Soc. Lond. Ser. A Math. Phys. Eng. Sci. 467\/} (2011),
  1563--1576.

\bibitem{HutzenthalerJentzenKloeden2012}
{\sc Hutzenthaler, M., Jentzen, A., and Kloeden, P.~E.}
\newblock Strong convergence of an explicit numerical method for {SDE}s with
  non-globally {L}ipschitz continuous coefficients.
\newblock {\em Ann. Appl. Probab. 22}, 4 (2012), 1611--1641.

\bibitem{HutzenthalerJentzenKloeden2013}
{\sc Hutzenthaler, M., Jentzen, A., and Kloeden, P.~E.}
\newblock Divergence of the multilevel {M}onte {C}arlo {E}uler method for
  nonlinear stochastic differential equations.
\newblock {\em Ann. Appl. Probab. 23}, 5 (2013), 1913--1966.

\bibitem{Jentzen2015SPDElecturenotes}
{\sc Jentzen, A.}
\newblock Stochastic {P}artial {D}ifferential {E}quations: {A}nalysis and
  {N}umerical {A}pproximations.
\newblock Lecture notes, ETH Zurich, summer semester 2015, available online at
  \url{https://people.math.ethz.ch/~grsam/NASPDE_MATH_FS15/script/NASPDE_130.pdf},
  September 2015.

\bibitem{jk09d}
{\sc Jentzen, A., and Kloeden, P.~E.}
\newblock The numerical approximation of stochastic partial differential
  equations.
\newblock {\em Milan J. Math. 77}, 1 (2009), 205--244.

\bibitem{Jentzen2009c}
{\sc Jentzen, A., and Kloeden, P.~E.}
\newblock Overcoming the order barrier in the numerical approximation of
  stochastic partial differential equations with additive space-time noise.
\newblock {\em Proc. R. Soc. Lond. Ser. A Math. Phys. Eng. Sci. 465}, 2102
  (2009), 649--667.

\bibitem{Jentzen2011b}
{\sc Jentzen, A., Kloeden, P.~E., and Winkel, G.}
\newblock Efficient simulation of nonlinear parabolic {SPDE}s with additive
  noise.
\newblock {\em Ann. Appl. Probab. 21}, 3 (2011), 908--950.

\bibitem{jp2015}
{\sc {Jentzen}, A., and {Pu{\v s}nik}, P.}
\newblock {Strong convergence rates for an explicit numerical approximation
  method for stochastic evolution equations with non-globally Lipschitz
  continuous nonlinearities}.
\newblock {\em arXiv:1504.03523\/} (2015), 38 pages.

\bibitem{Kloeden2007}
{\sc Kloeden, P.~E., and Neuenkirch, A.}
\newblock The pathwise convergence of approximation schemes for stochastic
  differential equations.
\newblock {\em LMS J. Comput. Math. 10\/} (2007), 235--253.

\bibitem{kll2015}
{\sc {Kov{\'a}cs}, M., {Larsson}, S., and {Lindgren}, F.}
\newblock {On the discretization in time of the stochastic Allen-Cahn
  equation}.
\newblock {\em arXiv:1510.03684\/} (2015), 34 pages.

\bibitem{Kurniawan2014}
{\sc Kurniawan, R.}
\newblock Numerical approximations of stochastic partial differential equations
  with non-globally {L}ipschitz continuous nonlinearities.
\newblock {\em Master thesis, University of Zurich and ETH Zurich\/} (2014), 74
  pages.

\bibitem{lt13}
{\sc Lord, G.~J., and Tambue, A.}
\newblock Stochastic exponential integrators for the finite element
  discretization of {SPDE}s for multiplicative and additive noise.
\newblock {\em IMA J. Numer. Anal. 33}, 2 (2013), 515--543.

\bibitem{pr07}
{\sc Pr{\'e}v{\^o}t, C., and R{\"o}ckner, M.}
\newblock {\em A concise course on stochastic partial differential equations},
  vol.~1905 of {\em Lecture Notes in Mathematics}.
\newblock Springer, Berlin, 2007.

\bibitem{qi2016accelerated}
{\sc Qi, R., and Wang, X.}
\newblock An accelerated exponential time integrator for semi-linear stochastic
  strongly damped wave equation with additive noise.
\newblock {\em arXiv:1602.06050\/} (2016), 20 pages.

\bibitem{Sabanis2013}
{\sc Sabanis, S.}
\newblock {A note on tamed Euler approximations}.
\newblock {\em Electron. Commun. Probab. 18\/} (2013), no. 47, 1--10.

\bibitem{Sabanis2013E}
{\sc {Sabanis}, S.}
\newblock {Euler approximations with varying coefficients: the case of
  superlinearly growing diffusion coefficients}.
\newblock {\em arXiv:1308.1796\/} (2015), 24 pages.

\bibitem{Stroock2010}
{\sc Stroock, D.~W.}
\newblock {\em Probability Theory}, second~ed.
\newblock Cambridge University Press, 2010.
\newblock Cambridge Books Online.

\bibitem{TretyakovZhang2013}
{\sc Tretyakov, M.~V., and Zhang, Z.}
\newblock {A fundamental mean-square convergence theorem for SDEs with locally
  Lipschitz coefficients and its applications}.
\newblock {\em SIAM J. Numer. Anal. 51}, 6 (2013), 3135--3162.

\bibitem{WangGan2013}
{\sc Wang, X., and Gan, S.}
\newblock {The tamed Milstein method for commutative stochastic differential
  equations with non-globally Lipschitz continuous coefficients}.
\newblock {\em J. Difference Equ. Appl. 19}, 3 (2013), 466--490.

\bibitem{wang2015note}
{\sc Wang, X., and Qi, R.}
\newblock A note on an accelerated exponential {E}uler method for parabolic
  {SPDE}s with additive noise.
\newblock {\em Applied Mathematics Letters 46\/} (2015), 31--37.

\end{thebibliography}

\end{document}